\documentclass[a4paper, 11pt,reqno]{amsart}
\usepackage[top = 1in, bottom = 1in, left=1.2in, right=1.2in]{geometry}

\usepackage[usenames,dvipsnames]{xcolor}
\usepackage{amssymb,amsmath,amsthm,graphicx,mathtools,enumerate,mathrsfs,tabularx,bbm,tikz,url}
\usepackage[shortlabels]{enumitem}
\usepackage[british]{babel}
\usepackage[utf8]{inputenc}
\usepackage{upgreek}
\usepackage{cite}
\usepackage{easyReview}

	\definecolor{mygreen}{RGB}{106, 168, 79}
	\definecolor{myblue}{RGB}{98, 160, 234}
	\definecolor{myred}{RGB}{246,97,81}
	\definecolor{black}{rgb}{0,0,0}

\allowdisplaybreaks

\usepackage{hyperref}
\hypersetup{colorlinks, linkcolor={red!50!black}, citecolor={green!50!black}, urlcolor={blue!50!black}}

\usepackage{caption}
\usepackage{subcaption}

\usepackage[mathlines]{lineno}
\usepackage{etoolbox} 

\newcommand*\linenomathpatch[1]{%
	\expandafter\pretocmd\csname #1\endcsname {\linenomath}{}{}%
	\expandafter\pretocmd\csname #1*\endcsname{\linenomath}{}{}%
	\expandafter\apptocmd\csname end#1\endcsname {\endlinenomath}{}{}%
	\expandafter\apptocmd\csname end#1*\endcsname{\endlinenomath}{}{}%
}
\newcommand*\linenomathpatchAMS[1]{%
	\expandafter\pretocmd\csname #1\endcsname {\linenomathAMS}{}{}%
	\expandafter\pretocmd\csname #1*\endcsname{\linenomathAMS}{}{}%
	\expandafter\apptocmd\csname end#1\endcsname {\endlinenomath}{}{}%
	\expandafter\apptocmd\csname end#1*\endcsname{\endlinenomath}{}{}%
}

\expandafter\ifx\linenomath\linenomathWithnumbers
\let\linenomathAMS\linenomathWithnumbers
\patchcmd\linenomathAMS{\advance\postdisplaypenalty\linenopenalty}{}{}{}
\else
\let\linenomathAMS\linenomathNonumbers
\fi

\linenomathpatchAMS{gather}
\linenomathpatchAMS{multline}
\linenomathpatchAMS{align}
\linenomathpatchAMS{alignat}
\linenomathpatchAMS{flalign}
\linenomathpatch{equation}

\usepackage{cleveref}

\theoremstyle{plain}
\newtheorem{theorem}{Theorem}[section]
\crefname{theorem}{Theorem}{Theorems}

\newtheorem{proposition}[theorem]{Proposition}
\crefname{proposition}{Proposition}{Propositions}

\crefname{corollary}{Corollary}{Corollaries}

\newtheorem{lemma}[theorem]{Lemma}
\crefname{lemma}{Lemma}{Lemmas}

\crefname{conjecture}{Conjecture}{Conjectures}

\newtheorem{claim}[theorem]{Claim}
\crefname{claim}{Claim}{Claims}

\crefname{observation}{Observation}{Observations}

\crefname{setup}{Setup}{Setups}

\newtheorem{fact}[theorem]{Fact}
\crefname{fact}{Fact}{Facts}

\crefname{algorithm}{Algorithm}{Algorithms}

\crefname{remark}{Remark}{Remarks}

\crefname{example}{Example}{Examples}

\theoremstyle{definition}
\newtheorem{definition}[theorem]{Definition}
\crefname{definition}{Definition}{Definitions}

\crefname{construction}{Construction}{Constructions}

\crefname{question}{Question}{Questions}

\crefname{problem}{Problem}{Problem}

\numberwithin{equation}{section}

\crefname{section}{Section}{Sections}
\crefname{appendix}{Appendix}{Appendix}

\crefname{figure}{Figure}{Figures}

\newcommand{\rf}[1]{\cref{#1} (\nameref*{#1})}

\def\rfCR#1{\cref{#1}  \\ \small\textup{(}\nameref*{#1}\textup{)}}
\def\rfCRphantom#1{\phantom{\cref{#1}}  \\ \phantom{\small(\nameref*{#1})}}

\newenvironment{proofclaim}[1][Proof of the claim]{\begin{proof}[#1]}{\end{proof}}

\long\def\OLD#1{}
\newcommand{\EXPLANATION}[1]{} 

\let\epsilon\varepsilon
\let\rho\varrho
\newcommand{\vecb}{\mathbf}

\newcommand{\es}{\emptyset}
\newcommand{\eps}{\varepsilon}
\renewcommand{\rho}{\varrho}
\newcommand{\sm}{\setminus}
\renewcommand{\subset}{\subseteq}

\newcommand{\NATS}{\mathbb{N}}

\newcommand{\PP}{\mathbb{P}}

\newcommand\restr[2]{{
		\left.\kern-\nulldelimiterspace 
		#1 
		\vphantom{\big|} 
		\right|_{#2} 
}}

\newcommand{\dist}{\operatorname{d}}
\newcommand{\cA}{\mathcal{A}}
\newcommand{\cB}{\mathcal{B}}
\newcommand{\cC}{\mathcal{C}}

\newcommand{\cF}{\mathcal{F}}

\newcommand{\cM}{\mathcal{M}}

\newcommand{\cP}{\mathcal{P}}
\newcommand{\cQ}{\mathcal{Q}}
\newcommand{\cR}{\mathcal{R}}
\newcommand{\cS}{\mathcal{S}}

\newcommand{\cU}{\mathcal{U}}
\newcommand{\cV}{\mathcal{V}}
\newcommand{\cW}{\mathcal{W}}

\newcommand{\cY}{\mathcal{Y}}

\newcommand{\oH}{{\operatorname H}}
\newcommand{\oG}{{\operatorname G}}
\DeclareMathOperator{\spl}{split}

\usepackage{pgfplots}
\pgfplotsset{compat=1.15}
\usepackage{mathrsfs}
\usetikzlibrary{arrows}

\usetikzlibrary{arrows.meta,    
	positioning,
	quotes}         

\title{Resilience for Loose Hamilton Cycles}
\date{\today}

\author[Alvarado]{José D. Alvarado}

\author[Kohayakawa]{Yoshiharu Kohayakawa}

\author[Lang]{Richard Lang}

\author[Mota]{Guilherme O. Mota}

\author[Stagni]{Henrique Stagni}

\address[R.~Lang]{Fachbereich Mathematik,
	Universität Hamburg,
	20146 Hamburg, Germany}
\email{richard.lang@uni-hamburg.de}

\address[J. D. Alvarado, Y. Kohayakawa, G. O. Mota,
H. Stagni]{Ins\-ti\-tu\-to de Ma\-te\-m\'a\-ti\-ca e Estat\'{\i}stica,
  Universidade de S\~ao Paulo, Rua do Mat\~ao 1010, 05508--090 S\~ao
  Paulo, Brazil}
\email{\{josealvarado.mat17|yoshi|mota|stagni\}@ime.usp.br}

\thanks{This research was partly supported by DFG (450397222),
  H2020-MSCA (101018431), FAPESP (2021/11020-9, 2018/04876-1,
  2019/13364-7), CNPq (311412/2018-1, 406248/2021-4, 306620/2020-0,
  406248/2021-4) and CAPES (Finance Code 001).  CAPES is the
  Coordena\c c\~ao de Aperfei\c coamento de Pessoal de N\'ivel
  Superior.  CNPq is the National Council for Scientific and
  Technological Development of Brazil.  FAPESP is the S\~ao Paulo
  Research Foundation.}

\begin{document}

\begin{abstract}
	We study the emergence of loose Hamilton cycles in subgraphs of random
	hypergraphs.  Our main result states that the minimum
	$d$-degree threshold for loose Hamiltonicity relative to the
	random $k$-uniform hypergraph $\oH_k(n,p)$ coincides with
	its dense analogue whenever $p \geq n^{- (k-1)/2+o(1)}$.
	The value of $p$ is approximately tight for $d>(k+1)/2$.
	This is particularly interesting because the dense
	threshold itself is {not known beyond the cases when
	$d \geq k-2$.}
\end{abstract}

\footskip=30pt
\maketitle

\vspace{-0.5cm}

\section{Introduction}\label{sec:introduction}
A widely studied topic in combinatorics is the existence of
vertex spanning substructures in graphs and hypergraphs.
Since the corresponding decision problems are in many cases
computational intractable, a large branch of research has
focused on obtaining sufficient and easily verifiable
conditions assuring the existence of these structures.
A classic example of such a result is Dirac's theorem,
which states that every graph on $n\geq 3$ vertices and
minimum degree at least $n/2$ contains a Hamilton cycle.
Since its inception, Dirac's theorem has been generalised in numerous
ways~\cite{Gou14,KO14}.  While the situation is
increasingly well-understood for graphs, many problems in the
setting of hypergraphs remain largely open.

In the past decades much effort has been dedicated to studying which phenomena
of dense settings can be transferred to sparser but well-structured
settings.  The study of combinatorial theorems relative to a
random set has seen much progress in the recent years, which
has gone alongside the development of powerful new methods and
tools.  One of the most fascinating aspects of this field is
that certain dense parameters can be transferred to the sparse
setting, without knowing their precise value, a famous
example being hypergraph Túran densities as shown
independently by Conlon and Gowers~\cite{CG16} and
Schacht~\cite{Sch16}.  For a more detailed exposition of these
results, see the survey of Conlon~\cite{Con14}.

Here we are interested in \emph{transference} results for
spanning substructures.  A commonly studied object in this
line of research is the random graph introduced by Erdős and
Rényi.  We denote by $\oG(n,p)$ the binomial random graph with
$n$ vertices that contains every possible edge independently
with probability $p$.  The study of Hamilton cycles in random graphs
dates back to Pósa~\cite{Pos76} and Korshunov~\cite{Kor76},
who, independently, showed that $ \oG(n,p)$ contains a
Hamilton cycle with high probability\footnote{Meaning with
	probability going to $1$ as $n$ tends to infinity.} (w.h.p.\
for short) provided that~$p$ is somewhat greater than $(\log
	n) /n$.  More precise and stronger results were obtained by
Komlós and Szemerédi~\cite{KS83}, Ajtai, Komlós and
Szemerédi~\cite{AKS85} and Bollobás~\cite{bollobas84}.

Given this, we can ask which subgraphs of $\oG(n,p)$ typically
contain a Hamilton cycle.  Sudakov and Vu~\cite{SV08}
conjectured the following random analogue of Dirac's theorem,
which was proved by Lee and Sudakov~\cite{LS12}: For any $\eps
	> 0$, if $p$ is somewhat greater than $(\log n)/n$, then a
random graph $\oG(n, p)$ typically has the property that every
spanning subgraph with minimum degree at least $(1 +
	\eps)np/2$ has a Hamilton cycle.  Since this work, there has
been a lot of interest in such \emph{resilience} theorems for
other types of spanning or almost-spanning subgraphs (see,
e.g., \cite{ABET20,ABH+16,BCS11,Mon20,SST18}), introducing
many new ideas and techniques.

We study this problem in the setting of hypergraphs.  A
\emph{$k$-uniform hypergraph} (\emph{$k$-graph} for brevity)
$G$ consists of a set of vertices $V(G)$ and a set of edges
$E(G)$ with each edge containing exactly $k$ vertices.\footnote{For convenience, we write \emph{subgraph} instead of \emph{subhypergraph}.}  The \emph{degree}
$\deg(S)$ of a set $S \subset V(G)$ is the number of edges
that contain $S$.  For $1 \leq d \leq k-1$, we denote the
\emph{minimum $d$-degree} by $\delta_d(G)$, which is defined
as the maximum $m$ such that every set of $d$ vertices
(\emph{$d$-set} for short) has degree at least $m$ in $G$.

An intriguing open question in extremal hypergraph theory is to
determine asymptotically optimal minimum degrees conditions
that force hypergraph analogues of Hamilton cycles.  A
\emph{loose cycle} in a $k$-graph is a cyclic sequence of
edges such that each two consecutive edges overlap in exactly
one vertex, and no pair of non-consecutive edges have vertices
in common. {The \emph{order} of a loose cycle is the number of
		used vertices, and its \emph{length} is the number of used edges.}
A loose cycle is \emph{Hamilton} if it spans all
vertices.
\begin{definition}
	The minimum \emph{$d$-degree threshold for loose
		Hamilton cycles} $\mu_d(k)$ is defined as the least $\mu \in
		[0,1]$ such that for every $\gamma>0$ and large enough $n$
	divisible by $k-1$, every $n$-vertex $k$-graph $G$ with
	$\delta_d(G) \geq (\mu + \gamma) \binom{n-d}{k-d}$ contains a
	loose Hamilton cycle.
\end{definition}
For $d=k-1$, the threshold $\mu_d(k)$
was determined by Kühn and Osthus~\cite{KO06} when $k=3$ and in
general independently by Keevash, Kühn, Osthus and
Mycroft~\cite{KKMO11} and by Hàn and Schacht~\cite{HS10}.
{Beyond this, the threshold is also known for $d=k-2$, where it was determined by Buß, Hàn and Schacht~\cite{BHS13} 
for $k=3$ and in general by Bastos, Mota, Schacht,
Schnitzer and Schulenburg~\cite{BMSSS17}.}
For later reference, we note that $\mu_1(k) \geq 2^{-(k-1)}$.
This follows by considering the disjoint union of two cliques of the same order.
Additional lower bounds can be found in the work of Han and Zhao~\cite{HZ16}.
We remark that
other types of cycles have also been studied, and we refer the
reader to a recent survey for a more detailed history
of the problem~\cite{Zha16}.

Returning to the random setting, the binomial random $k$-graph
$\oH_k(n,p)$ on $n$ vertices is defined analogously to
$\oG(n,p)$ and contains every possible edge independently with
probability $p$.
Loose Hamiltonicity in the random setting was investigated by
Dudek, Frieze, Loh and Speiss~\cite{DFL+12}, who showed that
w.h.p. $\oH_k(n,p)$ contains a loose Hamilton cycle provided
that $p$ is somewhat greater than $(\log n)/n^{k-1}$ after preliminary work of Frieze~\cite{Fri10}.
The value of $p$ is asymptotically optimal, since below this
threshold, w.h.p. $\oH_k(n,p)$ contains isolated vertices.
It should be noted that these results have been recovered  recently in a much more general setting~\cite{FKN+21,PP22}.

Given the resilience results in the graph setting, it is
natural to ask whether extensions to the hypergraph setting are possible.
For loose cycles, this question was posed by Frieze~\cite[Problem 58]{Fri19}.
Our main result  gives an essentially optimal answer to this
question for $d > (k+1)/2$, which is in particular interesting
as the threshold $\mu_d(k)$ is not known beyond the cases in
which $d \geq k-2$~\cite{BMSSS17,BHS13,HS10,KKMO11,KMO10}.

\begin{theorem}[Main result]\label{thm:main}
	For every $1 \leq d<k$ and $\gamma >0$ there is a $C >0$
	such that the following holds. If
	$p\geq\max\{n^{-(k-1)/2+\gamma}, Cn^{-(k-d)}\log n\}$ and~$n$ is divisible
	by $k-1$, then w.h.p. $G \sim \oH_k(n,p)$ has the property
	that every spanning subgraph $G' \subset G$ with
	$\delta_d(G') \geq (\mu_d(k)+\gamma) p \binom{n-d}{k-d}$
	contains a loose Hamilton cycle.
\end{theorem}

The value of $p$ in this result is unlikely to be optimal for
the whole range of $d$.  On the other hand,
$p=\Omega(n^{d-k}\log n)$ is a natural lower bound, since
otherwise $\delta_d(G) = 0$ with high probability.  Thus, our
result is approximately tight whenever $d > (k+1)/2$: for
such~$d$, we have $n^{-(k-1)/2+\gamma}\leq Cn^{d-k}\log n$
if~$\gamma$ is small enough, and hence our hypothesis on~$p$
in Theorem~\ref{thm:main} matches the natural lower bound
for~$p$ mentioned above.
Very recently, Petrova and Trujić~\cite{PT22} gave a stronger
bound for the case $(k,d)=(3,1)$, as they proved that for
$k=3$ it suffices to have
$p\geq C \max \left\{n^{-3/2},\, n^{d-3} \right\}\log n$ for
some $C >0$.  However, for all we know, the threshold for the
property described in~\cref{thm:main} could even be of order
$n^{d-k}\log n$ for the whole range of~$d$.

Resilience for Hamiltonicity has also been investigated for other types of
hypercycles.  In particular, Clemens,
Ehrenmüller and Person~\cite{CEP19} studied Hamilton Berge
cycles (which are less restrictive than loose cycles) in
$3$-graphs and Allen, Parczyk and Pfenninger~\cite{APP21}
studied tight Hamilton cycles (which are more restrictive than
loose cycles) for $d=k-1$.  Moreover, Ferber and
Kwan~\cite{FK22} proved an analogous result to \cref{thm:main}
for perfect matchings, which gives approximately tight bounds
for $p$ whenever $d > k/2$.

We also note that Hamiltonicity has been studied in other random graph models such as `random perturbation'~\cite{BFM03,KKS17} and `random robustness'~\cite{KLS14,KMP23,JLS23}.
We return to the latter in \cref{sec:conclusion}.

Our proof is based on the absorption method in combination
with embedding results for the (Weak) Hypergraph Regularity Lemma.
We also benefit from a framework of Ferber and
Kwan~\cite{FK22}, which was introduced to tackle the
resilience problems for matchings.  The main difference to
matchings is that Hamilton cycles come with a notion of
connectivity.  Hence, to prove~\cref{thm:main} we need to
extend this framework with further ideas in essentially every
step.

An important constraint for our proof is that the value of
$\mu_d(k)$ is unknown in almost all cases.  In contrast to the
aforementioned results for
Hamiltonicity~\cite{APP21,CEP19,PT22}, we therefore cannot
rely on any structural insights of past work regarding $\mu_d(k)$.
Our strategy thus uses only the mere existence of loose
Hamilton cycles to extract certain characteristic properties
such as connectivity and covering most vertices and relate them
to the random setting.  Unfortunately, this is not sufficient
in certain critical situations, where multiple Hamilton cycles
have to be combined.  We overcome the issues arising in this
situation by showing that the threshold $\mu_d(k)$ actually
allows us to find a Hamilton cycle with additional properties
such as certain vertices being far apart from each other.
Hence, in order to find a Hamilton cycle in the sparse
setting, we develop in parallel a way to find an `enhanced'
Hamilton cycle in the dense setting.  This process is
illustrated in \cref{fig:diagram}.

\medskip

The rest of the paper is organised as follows.  In the next
section, we present two main lemmas from which
we derive~\cref{thm:main}.
In \cref{sec:preliminaries}, we introduce a series of auxiliary results and machinery that we deploy throughout the proofs.
The rest of the paper, Sections~\ref{sec:connection-lemma-sparse-proof}--\ref{sec:absorber-lemma-dense-proof}, is
dedicated to the proofs of the above mentioned two lemmas.

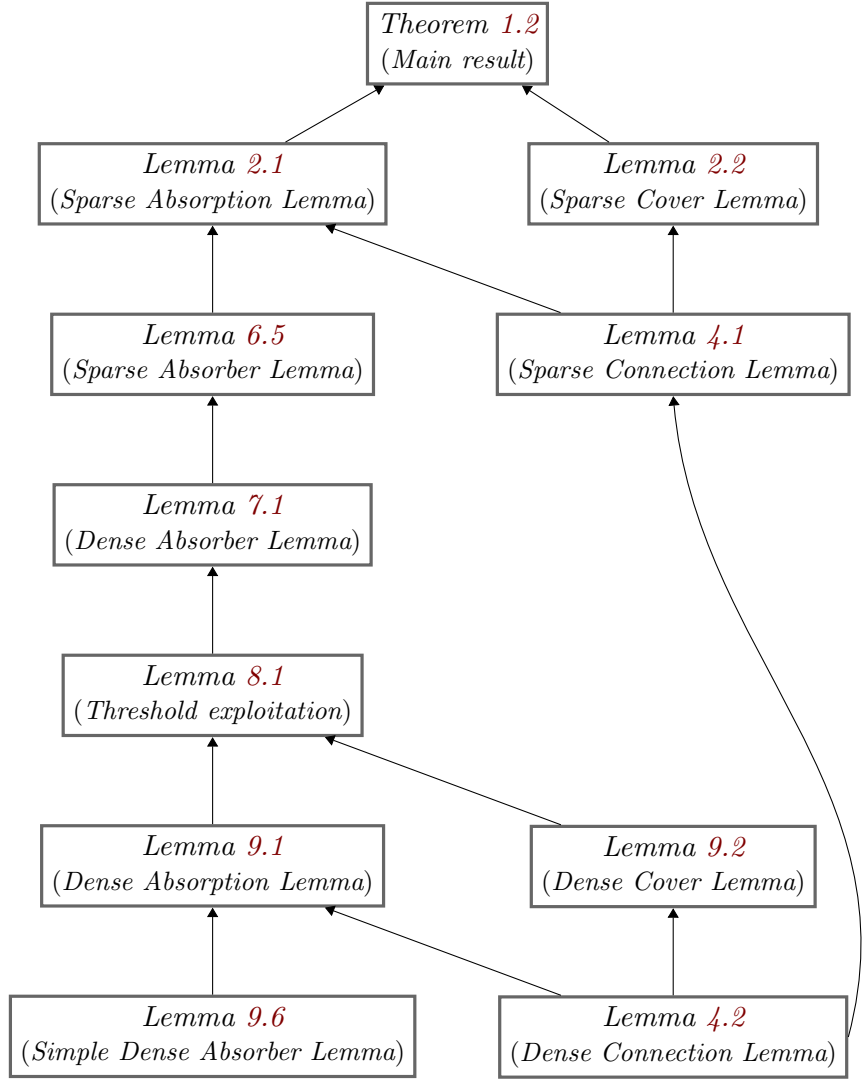
\begin{figure}[h]
	\begin{tikzpicture}[
			node distance = 3em and 2em,
			squarednode/.style = {draw=black!60,
					very thick,
					align=center,
					minimum width=4em,
					minimum height=2em,
					font=\itshape},
			every edge/.style = {draw, -Triangle},       
			every edge quotes/.style = {auto=right,
					font=\small\sffamily}   
		]
		\begin{scope}[every node/.style = {squarednode}]    
			\node   (main)                   {\rfCR{thm:main}};
			\node[draw=white,minimum width=0,minimum height=0] [below=of main]  (phantom-main-bot) {};
			\node   (cover-sparse) [right=of phantom-main-bot]    {\rfCR{lem:cover-sparse}};
			\node   (absorption-sparse) [left=of phantom-main-bot]    {\rfCR{lem:absorption-sparse}};
			\node   (connection-sparse) [below=of cover-sparse]    {\rfCR{lem:connection-sparse}};
			\node   (absorber-sparse) [below=of absorption-sparse]    {\rfCR{lem:absorber-sparse}};
			\node   (absorber-dense) [below=of absorber-sparse]    {\rfCR{lem:absorber-dense}};
			\node   (improved-threshold) [below=of absorber-dense]    {\rfCR{lem:threshold-exploitation}};
			\node   (absorption-dense) [below=of improved-threshold]    {\rfCR{lem:absorption-dense}};
			\node   (absorber-dense-simple) [below=of absorption-dense]    {\rfCR{lem:absorber-dense-simple}};
			\node[draw=white]   (auxtop) [below=of connection-sparse] {\rfCRphantom{thm:main}};
			\node[draw=white]   (auxbot) [below=of auxtop] {\rfCRphantom{thm:main}};
			\node   (cover-dense) [below=of auxbot]    {\rfCR{lem:cover-dense}};
			\node   (connection-dense) [below=of cover-dense]    {\rfCR{lem:connection-dense}};
		\end{scope}
		\draw   (cover-sparse) edge  (main);       
		\draw   (absorption-sparse) edge  (main);       
		\draw   (absorber-sparse) edge  (absorption-sparse);       
		\draw   (connection-sparse) edge  (absorption-sparse);       
		\draw   (connection-sparse) edge  (cover-sparse);       
		\draw   (absorber-dense) edge  (absorber-sparse);       
		\draw   (improved-threshold) edge  (absorber-dense);       
		\draw   (connection-dense.east) edge[bend right=30, in=170]  (connection-sparse.south);       
		\draw   (absorber-dense-simple) edge  (absorption-dense);       
		\draw   (absorption-dense) edge  (improved-threshold);       
		\draw   (cover-dense) edge  (improved-threshold);       
		\draw   (connection-dense) edge  (absorption-dense);       
		\draw   (connection-dense) edge  (cover-dense);       
	\end{tikzpicture}
	\caption{Proof diagram.}
	\label{fig:diagram}
\end{figure}

\section{Proof of \rf{thm:main}}\label{sec:proof-main-result}
We apply the method of absorption to find a Hamilton cycle in
the proof of \cref{thm:main}.  Informally, this technique
separates the argument into two parts.  First we find a
special path $A$ that allows us to integrate any small set of
vertices into a larger path.  Then we cover all but few
vertices with a loose cycle that contains $A$ as a subpath.
We conclude the proof by using the property of $A$.

A \emph{loose path} $P$ in a $k$-graph is a sequence of edges
such that each two consecutive edges overlap in exactly one
vertex, and no pair of non-consecutive edges have vertices in
common.
	{If the context is clear, we simply speak of a \emph{path}.}
The \emph{order} of $P$ is the number of its vertices.
For vertices $u$ and $v$, we say that $P$ is a \emph{loose
	$(u,v)$-path} if $u$ is in the first edge, $v$ is in the
last edge and no other edge contains $u$ or $v$.
For convenience, the constant hierarchies are
expressed in standard $\ll$-notation in the remainder of the
paper.\footnote{To be precise, we write $x \ll y$ to mean
	that for any $y \in (0, 1]$ there exists an $x_0 \in (0,1)$
	such that for all $x \leq x_0$ the subsequent statements
	hold.  Hierarchies with more constants are defined in a
	similar way and are to be read from the right to the left.
		{The first constant in a hierarchy is always assumed to be positive.}
	If we write $1/x$ in a hierarchy, we implicitly
	mean that $x \in \NATS$.  We also occasionally write~$o(1)$
	for a function that tends to~$0$ as~$n\to\infty$.}
Moreover,
given an eventually positive function~$f(n)$ of~$n$, the
expression~$\omega(f(n))$ denotes a function~$g(n)$ such that
$g(n)/f(n)\to\infty$ as~$n\to\infty$.

\begin{lemma}[Sparse Absorption Lemma]\label{lem:absorption-sparse}
	Let $ \eta \ll \alpha \ll 1/k, \,1/d,\,\gamma$ and $1/C \ll 1/k,\gamma$ with
	$k\geq 3$ and $p
		\geq\max\{n^{-(k-1)/2+\gamma},Cn^{-(k-d)}\log n\}$.
	Then w.h.p.\ $G \sim \oH_k(n,p)$ has the following property.

	For any spanning subgraph $G' \subset G$ with
	$\delta_d(G') \geq (\mu_d(k) + \gamma ) p
		\binom{n-d}{k-d}$, there is a set $A \subset V(G')$
	with $|A| \leq \alpha n$ and two vertices $u,v \in A$
	such that for any subset $W \subset V(G')\setminus A$
	with $|W| \leq \eta n$ divisible by $k-1$, the induced
	graph $G'[A \cup W]$ has a loose $(u,v)$-path covering
	$A \cup W$.
\end{lemma}

The next lemma allows us to cover most vertices with a single loose path.

\begin{lemma}[Sparse Cover Lemma]\label{lem:cover-sparse}
	Let $1\leq d \leq k-1$ with $k\geq 3$, $\eta >0$, $ \alpha
		\ll 1/k,\,1/d,\,\gamma$ and $1/C \ll 1/k,\gamma$ and $p \geq\max \{Cn^{-(k-d)}\log n,\,
		Cn^{-(k-2)}\log n\}$,  then w.h.p.\ $G \sim \oH_k(n,p)$ has the
	following property.

	For any spanning subgraph $G' \subset G$ with $\delta_d(G')
		\geq (\mu_d(k) + \gamma ) p \binom{n-d}{k-d}$, $Q \subset
		V(G')$ with $|Q| \leq \alpha n$ and $u,v \in V(G')\setminus
		Q$, there
	is a loose $(u,v)$-path $P$ in $G'-Q$ that covers all but
	$\eta n$ vertices of $G' - Q$.
\end{lemma}

These two lemmas combine easily to a proof of our main result.

\begin{proof}[Proof of \cref{thm:main}]
	The case when $k=2$ is covered by the work of Lee and
	Sudakov~\cite{LS12}.  So in the following we assume that
	$k\geq3$.  Consider $\alpha$ and~$\eta$ with $ \eta \ll \alpha
		\ll 1/k,\,1/d,\,\,\gamma$.  Then w.h.p.\ $G \sim \oH_k(n,p)$
	satisfies the outcomes of
	\cref{lem:absorption-sparse,lem:cover-sparse}.  Now let $G'
		\subset G$ be a spanning subgraph with $\delta_d(G') \geq
		(\mu_d(k) + \gamma ) p \binom{n-d}{k-d}$.  By assumption on
	$G$, we may pick a set $A \subset V(G')$ and two vertices
	$u,v \in A$ with the property described in
	\cref{lem:absorption-sparse}.  Let $Q = A \sm \{u,v\}$.  By
	assumption on $G$, there is a loose $(u,v)$-path $P$ in
	$G'-Q$ that covers all but
	$\eta n$ vertices of
	$G' - Q$.  Let~$W$ be the set of uncovered vertices.  Then
	$|W|\leq\eta n$ and a simple argument shows that~$|W|$ is
	divisible by $k-1$.
	To finish the proof, we use the property of $A$
	to find a loose $(u,v)$-path $P'$ covering $A \cup W$.  It
	follows that $P \cup P'$ is a loose Hamilton cycle.
\end{proof}

The remaining sections are dedicated to the proofs of \cref{lem:absorption-sparse,lem:cover-sparse}.

\section{Preliminaries}\label{sec:preliminaries}
In this section, we introduce a series of tools and technical facts that will be used throughout the rest of the paper.
Much of the exposition closely follows the work of Ferber and Kwan~\cite{FK22}.

\subsection{The Sparse Regularity Lemma}
The Sparse Regularity Lemma allows us to approximately encode
the local edge densities of sparse graphs, provided that its
edges are nowhere too concentrated.  This can be formalised in
terms of regular partitions and upper-uniformity.  For the
following definitions, suppose $\varepsilon,\,\eta>0$, $D>1$ and
$0<p\leq1$.
\begin{itemize}
	\item \textbf{Density:} Given non-empty disjoint vertex sets
	      $X_{1},\dots,X_{k}$ in a $k$-graph $G$, we write
	      $e\left(X_{1},\ldots,X_{k}\right)$ for the number of edges
	      with a vertex in each $X_i$. The \emph{density} $
		      d(X_{1},\ldots,X_{k})$ is defined by
	      \[
		      d\left(X_{1},\ldots,X_{k}\right)=\frac{e\left(X_{1},\ldots,X_{k}\right)}{\left|X_{1}\right|\dots\left|X_{k}\right|}.
	      \]
	\item \textbf{Regular tuple:} A $k$-partite $k$-graph with
	      parts $V_{1},\dots,V_{k}$ is
	      \emph{$\left(\varepsilon,p\right)$-regular} if, for every
	      $X_{1}\subseteq V_{1},\dots,X_{k}\subseteq V_{k}$ with
	      $\left|X_{i}\right|\ge\varepsilon\left|V_{i}\right|$ for
	      $i\in[k]$, we have
	      \[
		      \left|d\left(X_{1},\dots,X_{k}\right)-d\left(V_{1},\dots,V_{k}\right)\right|\le\varepsilon p.
	      \]
	\item \textbf{Regular partition:} A partition $\mathcal{V} =
		      \{V_{1},\dots,V_{t}\}$ of the vertex set of a $k$-graph is
	      said to be \emph{$\left(\varepsilon,p\right)$-regular} if it
	      is an equipartition (meaning that the sizes of the parts
	      differ by at most one), and for all but
	      $\varepsilon
		      \binom {t}{k}$ of the $k$-sets
	      $\left\{V_{i_{1}},\dots,V_{i_{k}}\right\}$ from
	      $\mathcal{V}$, the $k$-partite $k$-graph induced by
	      $V_{i_{1}},\dots,V_{i_{k}}$ in~$G$ is
	      $\left(\varepsilon,p\right)$-regular.
	\item \textbf{Upper-uniformity:} A $k$-graph $G$ is
	      \emph{$\left(\lambda,p,D\right)$-upper-uniform} if for any
	      disjoint subsets of vertices $X_{1},\dots,X_{k}$ each of
	      size at least $\lambda \left|V(G)\right|$, we
	      have $d\left(X_{1},\dots,X_{k}\right)\le Dp$.
\end{itemize}

The (Weak) Sparse (Hypergraph) Regularity Lemma then states
that every upper-uniform hypergraph admits a regular partition
of constant size.

\begin{lemma}[Sparse Regularity Lemma {\cite[Lemma 4.2]{FK22}}]
	\label{lem:sparse-regularity}
	Let $1/n \ll\lambda\ll 1/r_1 \ll
		1/r_0,\,\eps,\,1/k,\,1/D$ and $p\in(0,1]$.
	Then any $\left(\lambda,p,D\right)$-upper-uniform
	$n$-vertex $k$-graph $G$ admits an
	$\left(\varepsilon,p\right)$-regular partition
	$V_{1},\dots,V_{r}$ of its vertex set into $r_{0}\le
		r\le r_1$ parts.
\end{lemma}

By using the notion of a \emph{reduced graph}, we can track
the parts where a regular partition is (relatively) dense and
regular.

\begin{definition}[Reduced graph]
	Given an $\left(\varepsilon,p\right)$-regular partition
	$V_{1},\dots,V_{r}$ of the vertex set of a $k$-graph $G$,
	the associated \emph{reduced hypergraph} is the $k$-graph whose
	vertices are the clusters $V_{1},\dots,V_{r}$, and we put an
	edge $\left\{ V_{i_{1}},\dots,V_{i_{k}}\right\}$ whenever
	$d\left(V_{i_{1}},\dots,V_{i_{k}}\right)>2\varepsilon p$ and
	the $k$-partite $k$-graph induced by
	$V_{i_{1}},\dots,V_{i_{k}}$ in~$G$ is
	$\left(\varepsilon,p\right)$-regular.
\end{definition}

As it turns out, the reduced graph approximately inherits the degree conditions of the original graph, but in a dense form.
This is a key fact to transition from the dense to the sparse setting.

\begin{lemma}[Degree inheritance {\cite[Lemma 4.4]{FK22}}]
	\label{lem:reduced-graph-degree-inheritance}
	Let $1/n \ll 1/r_1 \ll 1/r_0,\,\eps' \ll \eps \ll
		1/k,\,1/d,\,\delta$
	and $p\in(0,1]$ be given.  Let
	$G$ be an $\left(o(1),p,1+o(1) \right)$-upper-uniform
	$n$-vertex $k$-graph.  Let $G'\subseteq G$ be a spanning
	subgraph in which all but $\eps' n^d$ of the $d$-sets
	of vertices have degree at least $\delta
		\binom{n-d}{k-d}p$. Let $\mathcal{R}$ be the
	$r$-vertex reduced $k$-graph obtained by applying
	\cref{lem:sparse-regularity} to $G'$ with parameters
	$r_{0}$,~$p$ and~$\varepsilon$.  Then all but
	$\sqrt{\varepsilon}\binom{r}{d}$ of the $d$-sets of
	vertices of $\mathcal{R}$ have degree at least $\delta
		\binom{r-d}{k-d}-(4\sqrt \varepsilon+k/r_0) r^{k-d}$.
\end{lemma}

\subsection{Prepartitions and regularity}
The following are two technical versions of
\cref{lem:sparse-regularity}, which allow for prepartitions
and more precise control of the degrees.

\begin{lemma}[Prepartitioning regular partition {\cite[Lemma 4.5]{FK22}}]
	\label{lem:prepartition-regularity}
	Suppose that a $k$-graph $G$ has its vertices partitioned
	into sets $P_{1},\dots,P_{h}$. In the
	$\left(\varepsilon,p\right)$-regular partition
	$V_{1},\dots,V_{r}$ guaranteed by
	\cref{lem:sparse-regularity}, we can assume that all but
	$\varepsilon h r$ of the clusters~$V_{i}$ are contained
	in some~$P_{j}$.
\end{lemma}

\begin{definition}
	Consider a $k$-graph $G$, and let $P_{1},\dots,P_{h}$
	be disjoint sets of vertices. Also, consider an
	$\left(\varepsilon,p\right)$-regular partition
	$V_{1},\dots,V_{r}$ of the vertices of $G$. Then the
	\emph{partitioned reduced graph} $\mathcal{R}$
	with \emph{threshold}~$\tau$ is the $k$-graph defined as follows.
	The vertices of $\mathcal{R}$ are the clusters $V_{i}$ which are
	completely contained in some $P_{j}$, with an edge $\left\{ V_{i_{1}},\dots,V_{i_{k}}\right\} $
	if $d\left(V_{i_{1}},\dots,V_{i_{k}}\right)>\tau p$ and the
	$k$-partite $k$-graph induced by
	$V_{i_{1}},\dots,V_{i_{k}}$ in~$G$ is
	$\left(\varepsilon,p\right)$-regular.
\end{definition}

\begin{lemma}[Sparse Regularity Lemma {\cite[Lemma 4.7]{FK22}}]
	\label{lem:transfer-cluster-graph-partition}
	Let $1/n \ll\lambda\ll 1/r_1\ll
		1/r_0,\,\eps,\,1/k,\,1/d,\,\delta$
	and $p\in(0,1]$.  Let $G$ be an $\left(
		\lambda,p,1+ \lambda \right)$-upper-uniform
	$n$-vertex $k$-graph with a partition
	$P_{1},\dots,P_{h}$ of its vertices into parts of
	sizes $n_{1},\dots,n_{h}$, respectively.  Let~$G'$ be
	a spanning subgraph of~$G$ and let $\mathcal{R}$ be
	the partitioned reduced $r$-vertex $k$-graph with
	threshold $\tau$ obtained by applying
	\cref{lem:sparse-regularity,lem:prepartition-regularity}
	to $G'$ with parameters $r_{0}$, $p$ and
	$\varepsilon$.

	For every $1\leq i\leq h$, let $\mathcal{P}_{i}$ be
	the set of clusters contained in $P_{i}$, and let
	$r_{i}=|\mathcal P_i|$. Also, for every $J\subseteq
		\{1,\dots,h\}$, write $P_J:=\bigcup_{j\in J} P_j$,
	$n_J=|P_J|$, $\mathcal{P}_{J}=\bigcup_{j\in J}
		\mathcal P_j$ and $r_J=|\mathcal{P}_{J}|$. Then the
	following properties hold.

	\begin{enumerate}[\upshape(1)]
		\item\label{it:3.6.1} For each~$i$, we have
		$r_{i}\ge\left(n_{i}/n\right)r-\varepsilon hr$.
		\item\label{it:3.6.2}
		Consider some $1\leq i\leq h$, {$1\leq d'
					\leq d$} and some $J\subseteq \{1,\dots,h\}$,
		and suppose that all but $o(n^{d'})$ of the
		${d'}$-sets of vertices $X\subset P_{i}$
		satisfy
		\[ \deg_{P_{J}}\left(X\right)\ge\delta'
			p\binom{n_{J}-{d'}}{k-{d'}}
		\]
		for some $\delta'\geq\delta$.
		Then, in the reduced graph $\mathcal{R}$, for all
		but at most $\sqrt{\varepsilon}\binom{r}{{d'}}$ of
		the ${d'}$-sets of clusters
		$\mathcal{X}\subset\mathcal{P}_{i}$, we have
		\[
			\deg_{\mathcal{P}_{J}}\left(\mathcal{X}\right)\ge\delta'
			\binom{r_{J}-{d'}}{k-{d'}}-\left(\tau+\varepsilon
			h+\sqrt{\varepsilon}+k/r_{0}\right)r^{k-{d'}}.
		\]
	\end{enumerate}
\end{lemma}

Note that \cref{lem:reduced-graph-degree-inheritance} is
actually a special case of
\cref{lem:transfer-cluster-graph-partition} (taking $h=1$ and
threshold $\tau = 2\varepsilon$).

\subsection{The Embedding Lemma}
A hypergraph is \emph{linear} if every two edges intersect in
at most one vertex. The following embedding lemma allows us to
embed linear subgraphs of the reduced graph into the original
graph.

\begin{definition}
	Consider a $k$-graph $H$ with vertex set $\left\{
		1,\dots,r\right\} $ and let
	$\mathcal{G}\left(H,n,m,p,\varepsilon\right)$ be the
	collection of all $k$-graphs $G$ obtained in the following
	way. The vertex set of $G$ is a union of pairwise disjoint sets
	$V_{1},\dots, V_{r}$ each of size $n$. For every
	edge $\left\{ i_{1},\dots,i_{k}\right\} \in
		E\left(H\right)$, we add to $G$ an
	$\left(\varepsilon,p\right)$-regular $k$-graph with~$m$
	edges across $V_{i_{1}},\dots,V_{i_{k}}$.  These are the
	only edges of $G$.
\end{definition}

\begin{definition}
	For $G\in\mathcal{G}\left(H,n,m,p,\varepsilon\right)$, let
	$\#_{H}(G)$ be the number of ``canonical copies'' of $H$ in
	$G$, meaning that the copy of every vertex~$i$ from $H$ must
	come from~$V_{i}$.
\end{definition}

\begin{definition}
	\label{def:k-density}
	The $k$-\emph{density} $m_{k}\left(H\right)$ of a $k$-graph
	$H$ with more than $k$ vertices is defined as
	\[
		m_{k}\left(H\right)=\max\left\{
		\frac{e\left(H'\right)-1}{v\left(H'\right)-k}\colon H'\subseteq
		H\text{ with }v\left(H'\right)>k\right\}.
	\]
\end{definition}

The following result appears in the work of Ferber and
Kwan~{\cite[Lemma 4.11]{FK22}} and can be proved using the
methods of Conlon, Gowers, Samotij and Schacht~\cite{CGS14}.

\begin{lemma}[Sparse Embedding Lemma]
	\label{lem:embedding}
	For every linear $k$-graph~$H$ and every $\tau>0$,
	there exist $\varepsilon,\,\zeta>0$ with the following
	property. For every $\kappa>0$, there is $C>0$ such
	that if $p\ge CN^{-1/m_{k}\left(H\right)}$, then with
	probability $1-e^{-\Omega\left(N^{k}p\right)}$ the
	following holds in $G\sim\oH_k(N,p)$. For every
	$n\ge\kappa N$, $m\ge\tau pn^{k}$ and every subgraph
	$G'$ of $G$ in
	$\mathcal{G}\left(H,n,m,p,\varepsilon\right)$, we have
	$\#_{H}(G')>\zeta p^{e(H)} n^{v(H)}$.
\end{lemma}

\subsection{Properties of random graphs and subgraphs}

In what follows we prove a simple consequence of the
Chernoff bound and we list a series of technical statements
from the work of Ferber and Kwan~\cite{FK22} that describe
certain properties of random graphs.
For reals $x,y,z$ we write $x = y \pm z$ to mean that $x \in
	[y-z,y+z]$.

\begin{lemma}\label{lem:count-edges_containing_2-graph}
	For every integer $k\geq 3$ and every $0< \lambda$, $\eta<
		1$, there exists $C>0$ such that if {$p \geq
				Cn^{-(k-2)} \log{n}$},
	then w.h.p. $G\in
		\oH^{k}\left(n,p\right)$ has the following property. Let
	$U_1,\ldots,U_k$ be disjoint subsets of~$V(G)$ each of size
	at least
	$\lambda n$. Let $M\subset U_{k-1} \times U_k$ be of size at
	least $ C/(pn^{k-3})$. Then $G$ has $(1 \pm \eta) p
		|M| \prod_{j=1}^{k-2}{|U_j|}$ edges $e = \{v_1,\ldots,
		v_k\}$ with $v_i\in U_i$ for $i\in [k-2]$ and $(v_{k-1},v_k)
		\in M$.
\end{lemma}

\begin{proof}
	Fix $M$ and $U_1,\ldots, U_k$ as in the statement. We denote
	by $Z=Z(M; U_1,\ldots, U_k)$ number of edges $\{v_1,\ldots,
		v_k\} \in E(G)$ with $v_i\in U_i$ for $i\in [k-2]$ and
	$(v_{k-1},v_k) \in M$.  Since $Z$ has binomial distribution,
	by using the Chernoff bound, with probability
	$1- 2 \exp{(-c \, |M| n^{k-2} p )}$ we have
	\[
		Z = (1 \pm \eta) p |M| \prod_{j=1}^{k-2}{|U_j|},
	\]
	where $c := \eta^2 \lambda^{k-2}/3$.
	Then, taking $C := 4 \log{(k+1)} \,  c^{-1}$, by the union bound on the
	choices of $U_1,\ldots, U_k$ and $M$, the probability that
	the property described in the statement fails is at most
	\begin{align*}
		\sum_{(M; U_1,\ldots, U_k)}{  2 \, \exp{(- c \,  |M|
		n^{k-2} p)}} & \leq 2 (k+1)^n \sum_{m \geq C/(pn^{k-3})}{{ n^2 \choose m } \exp{(- c \,  m n^{k-2}p )} } \\
		             & \leq 2 (k+1)^n \sum_{m \geq C/(pn^{k-3})}{\exp{(-m(c n^{k-2} p -  2 \log{n} ))}}          \\
		             & \leq 2 (k+1)^n \sum_{m \geq C/(pn^{k-3})}{\exp{(-m(cn^{k-2}p/2))}}                        \\
		             & \leq 4 (k+1)^n \exp{(- 2\log{(k+1)} \, n )}                                               \\
		             & = o(1),
	\end{align*}
	which concludes the proof.
\end{proof}

\begin{lemma}[Corollary 5.2 in~\cite{FK22}]
	\label{lem:random-upper-uniform}
	Fix $k\geq2$, let $p=\omega(n^{1-k}\log n)$, and consider $G\in \oH^{k}\left(n,p\right)$. Then $G$ is
	$\left(o(1),p,1+o(1)\right)$-upper-uniform with probability at least $1-e^{-\omega(n\log n)}$.
\end{lemma}

\begin{lemma}[Lemma 5.3 in~\cite{FK22}]
	\label{lem:random-degree}
	For every $\alpha>0$ and $k\geq 3$, there is $C>0$
	such that if $p\ge C n^{2-k}$, then w.h.p.\ $G\sim
		\oH^{k}\left(n,p\right)$ has the following
	property. For every vertex $w$ and every set
	$X\subseteq V(G)$ of at most $\alpha n$ vertices,
	there are at most $2\alpha np\binom{n-2}{k-2}$ edges
	in~$G$ containing~$w$ and a vertex
	of~$X\setminus\{w\}$.
\end{lemma}

For vertex sets $S$ and~$X$ in a $k$-graph $G$, let $Z_G(S,X)$ be
the number of edges $e\in E(G)$ that contain~$S$ and have
non-empty intersection with~$X\setminus S$.

\begin{lemma}[Lemma 5.4 in~\cite{FK22}]
	\label{lem:random-degree-strong}
	Fix $\alpha>0$ and positive integers $d< k$. For
	$p = \omega(n^{d-k})$, w.h.p.\
	$G\sim \oH^{k}\left(n,p\right)$ has the following
	property.  For every subset $X\subseteq V(G)$ of size
	$|X|\leq \alpha n$, there are~$o(n^d)$ $d$-sets
	$S\subseteq V(G)$ such that
	$Z_G(S,X)>2\binom k d\alpha np\binom{n-d-1}{k-d-1}$.
\end{lemma}

\begin{lemma}[Lemma~5.6 in~\cite{FK22}]
	\label{lem:degree-inheritance_induced-subgraphs}
	Suppose $1\leq d<k$, $\alpha\ll1/k,\,\gamma$ and
	$0<\sigma\leq1-\alpha$.
	If $p=\omega(n^{d-k})$, then the
	following holds w.h.p.\ for
	$G\sim \oH^{k}\left(n,p\right)$.

	Let~$G'$ be a spanning subgraph of~$G$ with
	$\delta_d(G')\geq\left(\mu+\gamma\right)p\binom{n-d}{k-d}$,
	and let $Q\subset V(G')$ with $|Q|\leq\alpha n$ be given.
	If~$Y\subset V(G')\setminus Q$ with $|Y| = \sigma n$
	is chosen uniformly at random, then w.h.p.\ all but
	$o(n^d)$ $d$-sets in $V(G')\setminus Q$
	have degree at least
	$\left(\mu+\gamma/2\right)p\binom{\sigma n-d}{k-d}$
	into $Y$.
\end{lemma}

We remark that Lemma~5.6 in Ferber and Kwan~\cite{FK22}
implies \cref{lem:degree-inheritance_induced-subgraphs} above
for $\alpha = 0$.  The lemma above for $\alpha>0$ can be
proved following the proof of their Lemma~5.6, making use of
\cref{lem:random-degree-strong}.

\begin{lemma}[Lemma 3.4 in~\cite{FK22}]
	\label{lem:random-subset-degrees}
	Let $1\leq d<k$ and let $c \ll 1/k$.
	Consider an $n$-vertex $k$-graph
	$G$ where all but $\delta\binom{n}{d}$ of the $d$-sets
	have degree at least $\left(\mu+\eta\right)\binom{n-d}{k-d}$.
	Let~$S$ be a subset of $s\ge 2d$ vertices of~$G$
	chosen uniformly at random.  Then with probability
	at least $1-\binom{s}{d}\left(\delta+e^{-c\eta^{2}s}\right)$,
	the random induced subgraph $G\left[S\right]$ has minimum $d$-degree
	at least $\left(\mu+\eta/2\right)\binom{s-d}{k-d}$.
\end{lemma}

\section{Connecting vertices}\label{sec:connection-lemma-sparse-proof}
In the proofs of
\cref{lem:absorption-sparse,lem:cover-sparse}, we need to
connect vertices with a short path of uniform order while
avoiding a few other vertices. This is the purpose of the
following lemma.

\begin{lemma}[Sparse Connection Lemma]\label{lem:connection-sparse}
	Let $k\geq 3$ and suppose $1\leq d<k$,
	$\alpha\ll1/k,\,\gamma$ and {$1/K \ll 1/k,\gamma$}.  Suppose further that
	$0<\nu\leq1-\alpha$ and $\rho\ll1/k,\,\nu$. If
		{$p \geq \max\{ Kn^{-(k-d)}\log n,\,\omega(n^{-(k-2)})\}$,}  then w.h.p.\
	$G \sim \oH_k(n,p)$ has the following property.

	Let $G' \subset G$ be a spanning subgraph with
	$\delta_d(G') \geq (\mu_d(k) + \gamma ) p \binom{n-d}{k-d}$
	and $Q \subset V(G')$ with $|Q| \leq \alpha n$.  Let
	$C \subset V(G')\setminus Q$ be a $\nu n$-set taken
	uniformly at random.  Then with probability at least $2/3$
	the following holds.  For any $R \subset C$ with
	$|R| \leq \rho n$ and distinct
	$u,\,v \in V(G')\setminus (Q\cup R) $, there is a loose
	$(u,v)$-path~$P$ in~$G'$ of order $8(k-1)+1$
	with $V(P)\setminus\{u,v\} \subset C\sm R$.
\end{lemma}

Before we get to the details of the proof of
Lemma~\ref{lem:connection-sparse}, let us give an outline of
the argument.

\begin{proof}[Sketch of the proof]
	In this sketch we focus on the special case where $Q = \es$, $C =V(G)$ and $R=\es$, since the general statement follows in a
	quite similar way.

	It is also instructive to verify first that the lemma holds
	in the dense setting when $p=1$.  Indeed,
	in this case $G'$ contains a loose Hamilton cycle, and
	therefore any two vertices $u,\,v \in V(G)$ are trivially
	connected by a loose $(u,v)$-path~$P$.  However, we also
	need to control the order of $P$, which will be $4(k-1)+1$
	when $p=1$. To be precise, what matters in our
	applications of \cref{lem:connection-sparse} is that the
	path $P$ we obtain has the \emph{same} order for any choice
	of vertices $u$ and $v$.  As it turns out, this can easily
	be satisfied in the dense setting: we first note that the
	neighbourhoods of $u$ and $v$ have a substantial
	intersection.  When $d\geq 2$, this is trivial.  When $d=1$,
	this follows from the simple fact that $\mu_1(k) \geq
		2^{-k+1}$.  In either case, this provides enough room to
	construct a loose $(u,v)$-path of order $4(k-1)+1$ when $p=1$.

	The argument in the sparse setting is more involved.
	We apply the Weak Hypergraph Regularity Lemma to find
	a regular partition $\cV$ of $V(G')$ together with a
	reduced dense $k$-graph $\cR$ (whose vertices
	are the clusters of $\cV$) that approximately captures
	the local edge densities of~$G'$.  Crucially, $\cR$
	has almost everywhere $d$-degree at least
	$(\mu_d(k) + \gamma/2)\binom{v(\cR)-d}{k-d}$.  Using a
	prepartition, we may also assume that $u$ has a short
	path to every vertex of a cluster $W_u$, and likewise
	$v$ has a short path to every vertex of a cluster
	$W_v$.  Together with a Hypergraph Embedding Lemma
	this reduces the problem of connecting $u$ and $v$ in
	$G'$ to finding a loose $(W_u,W_v)$-path of order
	$4(k-1)+1$ in $\cR$.  To obtain such a path, we avoid
	$d$-sets of low degree by selecting a subgraph $\cR'
		\subset \cR$ with $W_u,\,W_v \in V(\cR')$ and
	$\delta_d(\cR') \geq (\mu_d(k) + \gamma/4)
		\binom{v(\cR')-d}{k-d}$.  This can be done by choosing
	a random induced subgraph whose order is much smaller
	in comparison to the error in the $d$-degree of $\cR$.  Once such
	an $\cR'$ is obtained, it follows by the above
	argument for the dense setting that $\cR'$ has the
	desired loose $(W_u,W_v)$-path.
\end{proof}

\subsection{Details of the proof of \cref{lem:connection-sparse}}\label{sec:connection-lemma-sparse-proof-appendix}
We begin by stating a result analogous to~\cref{lem:connection-sparse} for the dense setting, whose proof is deferred to \cref{sec:connection-lemma-dense-proof}.

\begin{lemma}[Dense Connection Lemma]\label{lem:connection-dense}
	Let $1/n \ll 1/k,\,1/d,\,\gamma$.
	Suppose that $G$ is an $n$-vertex $k$-graph  with $\delta_d(G) \geq (\mu_d(k)+ \gamma) \binom{n-d}{k-d}$, and let $u,\,v \in V(G)$ be distinct.
	Then there is a loose $(u,v)$-path of order $4(k-1)+1$
	in~$G$.
\end{lemma}

In what follows we prove~\cref{lem:connection-sparse}. For
that, recall that for vertex sets~$S$ and~$X$ in a $k$-graph~$G$, we
let $Z_G(S,X)$ be the number of edges of~$G$ that
contain~$S$ and have a non-empty intersection with~$X\setminus
	S$.  We also use
the fact that $\delta_1(G) / \binom{v(G)-1}{k-1} \geq
	\delta_d(G)/ \binom{v(G)-d}{k-d}$.

\begin{proof}[Proof of \cref{lem:connection-sparse}]
	We introduce the constants required in the proof in
	steps.  Let
	$1 \leq d<k$ and $\alpha\ll1/k,\,\gamma$ be as in
	\cref{lem:degree-inheritance_induced-subgraphs}.  Suppose we
	have~$\nu$ with $0<\nu\leq1-\alpha$.  In what follows, we
	shall apply \cref{lem:degree-inheritance_induced-subgraphs}
	with~$\sigma=\nu$.  Now let
	\begin{align*}
		\frac1{r_0} \ll \frac1s \ll \tau\ll\eta
		\ll \frac1k,\,\frac1d,\,\gamma
	\end{align*}
	be such that in particular $s$ can play the role of $n$ when
	applying Lemma~\ref{lem:connection-dense} with
	parameters~$k$ and~$d$ and~$\gamma$ of
	Lemma~\ref{lem:connection-dense} equal to~$\eta$.
	With~$\tau$ at hand, we apply Lemma~\ref{lem:embedding} to
	obtain~$\epsilon$ and~$\zeta$ that allow us to embed any
	linear hypergraph with at most $6(k-1)+1$ vertices in any
	suitable `$\epsilon$-regular system'.  Finally, we apply
	Lemma~\ref{lem:transfer-cluster-graph-partition} with
	\begin{equation*}
		\frac{1}{n}\ll\lambda \ll \frac1{r_1}
		\ll\frac1{r_0},\,\epsilon,\,\frac1k,\,\frac1d,\,\delta=\mu_d(k)\text{
			and }p.
	\end{equation*}
	We shall apply Lemma~\ref{lem:embedding} with
	$\kappa= \nu/(2r_1)$.
	Finally, let $1/K \ll 1/k,\gamma$ and $p \geq \max\{ Kn^{-(k-d)}\log n,\,\omega(n^{-(k-2)})\}$ such that we can apply \cref{lem:embedding,lem:random-upper-uniform,lem:random-degree,lem:random-degree-strong,lem:degree-inheritance_induced-subgraphs} with $k$-density $1/(k-1)$.
	In the following, we also
	assume that $n$ is large enough to satisfy
	\cref{lem:transfer-cluster-graph-partition} with the
	above constants.

	Note that from the choice of $p$, we may
	apply~\cref{lem:embedding,lem:random-upper-uniform,lem:random-degree,lem:random-degree-strong,lem:degree-inheritance_induced-subgraphs}.
	The next claim contains all properties that we
	require from the random $k$-graph.

	\begin{claim} \label{cla:connection-properties-random-graph}
		W.h.p.\ $G \sim \oH_k(n,p)$ satisfies the following properties.
		\begin{enumerate}[\upshape(1)]
			\item \label{itm:connection-embedding-lemma}
			      For each $\nu n$-set $C \subset V(G)$, the induced $k$-graph $G[C]$ satisfies the conclusion of \cref{lem:embedding} for embedding all graphs $H$ with $v(H)\leq 6(k-1) + 1$ and $m_k(H) \leq 1/(k-1)$.

			\item \label{itm:connection-upper-uniform} For each $\nu n$-set $C \subset V(G)$, the induced $k$-graph $G[C]$ is $(\lambda,p,1+\lambda)$-upper-uniform.

			\item \label{itm:connection-intersection*}
			      For every
			      $\beta \leq  \alpha+ \rho^{1/2}$, $\beta n$-set $X\subset V(G)$ and $w \in V(G)$, there are at most $2 \beta p n^{k-1}$
			      edges in $G$ containing $w$ and a vertex
			      of~$X\setminus\{w\}$.

			\item \label{itm:connection-intersection-tuples} For every
			      $\beta \leq \alpha+ \rho^{1/2}$ and
			      $X \subset V(G)$ with $|X| \leq \beta n$, there are
			      $o(n^d)$ $d$-sets $S$ such that $Z_G(S,X)
				      > 2 {k \choose d}\beta np{n-d-1 \choose k-d-1}$.

			\item \label{itm:connection-degree-subgraph} Let $G'$ be a
			      spanning subgraph of $G$ with $\delta_d(G') \geq
				      (\mu_d(k)+\gamma) p \binom{n-d}{k-d}$ and let $Q \subset
				      V(G')$ with $|Q| \leq \alpha n$ be given.  Let $C \subset V(G')\sm
				      Q$ be a $\nu n$-set chosen uniformly at random.  Then
			      with probability at least $5/6$ all but $o(n^d)$
			      $d$-sets of vertices in $V(G')\setminus Q$
			      have $d$-degree at least
			      $(\mu_d(k)+\gamma/2) p \binom{\nu n-d}{k-d}$ into~$C$.
		\end{enumerate}
	\end{claim}

	\begin{proof}
		To begin, observe that for a fixed set $C\subset V(G)$ of
		size $\nu n$, the induced subgraph $G[C]$ has probability
		distribution $\oH_k(\nu n,p)$.  For
		parts~\ref{itm:connection-embedding-lemma}
		and~\ref{itm:connection-upper-uniform}, note that there
		are ${\binom{n}{\nu n} \leq 2^n}$ ways to
		choose~$C$ and at most $O(1)$ ways to choose~$H$. We
		apply~\cref{lem:embedding} (with $\nu n$ playing the role
		of $N$) and \cref{lem:random-upper-uniform} and then we
		take the union bound over all choices for $H$ and $C$.

		Parts~\ref{itm:connection-intersection*}
		and~\ref{itm:connection-intersection-tuples} follow from
		applying~\cref{lem:random-degree,lem:random-degree-strong}, respectively,
		with $\beta$ playing the role of $\alpha$ in each case.
		Finally, part~\ref{itm:connection-degree-subgraph} follows by
		applying~\cref{lem:degree-inheritance_induced-subgraphs}, where $C$ and $\nu$ play the roles of $Y$ and $\sigma$, respectively.
	\end{proof}

	Let us fix a deterministic graph $G$ that satisfies
	the properties stated
	in~\cref{cla:connection-properties-random-graph}.  Let
	$G'$ be a spanning subgraph of $G$ with
	$\delta_d(G') \geq (\mu_d(k) + \gamma ) p
		\binom{n-d}{k-d}$ and let $Q \subset V(G')$ with $|Q| \leq
		\alpha n$ be fixed.  Let $G''=G'-Q$.
	For the next claim, given a vertex $x\in V(G'')$, we write
	$L_{G''}(x)$ for the \emph{link graph} of $x$ in $G''$, which is the $(k-1)$-graph on $V(G'')\sm \{x\}$ with a $(k-1)$-edge $e$ whenever $e \cup \{x\}$ is an edge in $G''$.

	\begin{claim}\label{cla:matchings}

		For any $x \in V(G'')$, there is a matching $M_x \subset E(L_{G''}(x))$ of size $10 k \rho \nu^{-(k-1)} n$.
	\end{claim}
	\begin{proof}
		Suppose that the claim is false, and
		let $M \subset E(L_{G''}(x))$ be a matching of maximum
		size with fewer than $(k-1)10k \rho \nu^{-(k-1)} n$ vertices.
		By~\cref{cla:connection-properties-random-graph}\ref{itm:connection-intersection*}
		with $V(M) \cup Q$ playing the role of $X$ there are at
		most $2k(\alpha + (k-1)10k\rho \nu^{-(k-1)}) pn^{k-1}$ edges
		in $G'\subset G$ containing $x$ and a vertex of $V(M) \cup
			Q$.  On the other hand, $\deg_{G'}(x) \geq (\mu_d(k) +
			\gamma ) p \binom{n-1}{k-1}$, which contradicts the
		maximality of~$M$.
	\end{proof}

	Next, let $C \subset V(G'')$ be a $\nu n$-set taken
	uniformly at random.  For $x\in V(G'')$, denote by $M_x'$
	the edges of $M_x$ that are entirely in $C$.  Using the
	Chernoff bound, the union bound and
	\cref{cla:connection-properties-random-graph}\ref{itm:connection-degree-subgraph},
	each of the following holds with probability at least $5/6$:
	\begin{enumerate}[(i)]
		\item for every $x \in V(G'')$, we have $|M_x'| \geq
			      9k \rho n$,\label{eq:1}
		\item all but $o(n^d)$ $d$-sets of vertices in
		      $V(G')\setminus Q$
		      have $d$-degree at least $(\mu_d(k) +\gamma/2) p
			      \binom{\nu n-d}{k-d}$ into $C$.\label{eq:2}
	\end{enumerate}
	Therefore, \ref{eq:1} and \ref{eq:2} hold with probability
	at least $2/3$. We proceed by fixing such a set $C$.

	Now let $R \subset C$ with $|R| \leq \rho n$ be given and consider
	distinct vertices $x_1,\,x_2 \in V(G')\sm (Q\cup R) $.
	Our goal is to find a loose $(x_1,x_2)$-path $P$ of order
	$8(k-1)+1$ with $V(P)\setminus\{x_1,x_2\}\subset C\sm
		R$.  For $i \in
		\{1,2\}$, let~$M_i$ be obtained from $M_{x_i}'$ by
	deleting all edges that contain vertices in~$R$.  After
	deleting some additional edges if necessary, we may assume
	that $\{x_1\}\cup V(M_{1})$ and $\{x_2\}\cup V(M_{2})$ are
	disjoint and both~$M_1$ and~$M_2$ have cardinality~$3 \rho
		n$.

	By
	\cref{cla:connection-properties-random-graph}\ref{itm:connection-upper-uniform},
	$G[C]$ is $(\lambda,p,1-\lambda)$-upper-uniform.  Thus, we may
	apply~\cref{lem:sparse-regularity} to $G''[C]$ with
	parameters $\eps$, $r_0$ and $r_1$.  Using
	\cref{lem:transfer-cluster-graph-partition}
	with prepartition $W_1 =
		V({M}_1),\,W_2=V({M}_2), \, U = C\sm (W_1\cup W_2
		\cup R\cup\{x_1,x_2\})$ we obtain a reduced $k$-graph $\cR$
	on $r \geq r_0$ vertices.  Let $\cW_{1}$, $\cW_2$, $\cU
		\subset V(\cR)$ be the sets of clusters contained in
	$W_1$, $W_2$ and~$U$, respectively.

	\begin{claim}\label{claim:consequenceRegularity}
		We have that
		\begin{itemize}
			\item all but
			      $\sqrt{\eps}|C|$ vertices
			      $X\in\cW_1 \cup \cW_2$ of $\cR$ satisfy $\deg_{\cU}(X) \geq
				      (\mu_d(k) + \gamma/4) {|\cU| -1 \choose k - 1 }$ and
			\item all but
			      $\sqrt{\eps}|C|^d$ $d$-sets $D
				      \subset \cU$ of $\cR$ satisfy $\deg_{\cU}(D) \geq (\mu_d(k) +
				      \gamma/4) {|\cU| -d \choose k - d }$.
		\end{itemize}
	\end{claim}
	\begin{proof}
		The conclusions follow directly
		from~\cref{lem:transfer-cluster-graph-partition}, once
		we verify the degree hypothesis for $G''[C]$ (as a
		spanning subgraph of $G[C]$).  Since the proofs of both
		statements are similar, we only prove the second one.  It
		is sufficient to show that all but $o(n^d)$ $d$-sets $D$
		in $U$ satisfy $\deg_U(D) \geq
			(\mu_d(k) + \gamma/3) p {|U| -d \choose k - d }$.  Note
		that, $\deg_U(D) \geq \deg_{C}(D) - (Z_G(D,W_1) +
			Z_G(D,W_2) + Z_G(D,R))$.  Therefore, if $\deg_U(D) <
			(\mu_d(k) + \gamma/3) p {|U| -d \choose k - d }$, then
		either
		$\deg_C(D) < (\mu_d(k) + \gamma/2) p {|U| -d \choose k - d }$,
		or $Z_G(D,W_1) > (\gamma/18) p {|U| -d \choose k - d }$,
		or $Z_G(D,W_2) > (\gamma/18) p
				{|U| -d \choose k - d }$ or $Z_G(D,R) > (\gamma/18) p
				{|U| -d \choose k - d }$.  In any case, by
		\cref{cla:connection-properties-random-graph}\ref{itm:connection-intersection-tuples}
		and~\ref{eq:2}, it follows
		that the number of $d$-sets $D \subset U$ which satisfy
		$\deg_U(D) < (\mu_d(k) + \gamma/3) p {|U| -d \choose k -
					d }$ is $o(n^d)$.
	\end{proof}

	Now we consider two vertices in $\cR$: fix $X_i\in
		\cW_i$ with $\deg_{\cU}(X_i) \geq
		(\mu_d(k) + \gamma/4) {|\cU| -1 \choose k - 1 }$ for
	$i\in\{1,2\}$.  We want to find a loose
	$(X_1,X_2)$-path $\cP$ of order $6(k-1) + 1$ in $\cR$
	such that
	$V(\cP)\setminus\{X_1,X_2\}\subset\cU$.
	For that, let
	$S \subset \cU$ be an $s$-set chosen uniformly at
	random. By~\cref{claim:consequenceRegularity}, we
	can apply~\cref{lem:random-subset-degrees} to obtain
	that $\delta_d(\cR[S]) \geq (\mu_d(k) +
		\eta)\binom{s -d}{k-d}$ with probability at least $2/3$.
	Moreover, standard concentration inequalities show that $\deg_{\cR[S]}(X_i) \geq (\mu_d(k) + \eta)
		\binom{s-1}{k-1}$ for $i \in\{1,2\}$ with probability at least $2/3$.
	Fix a set $S$ that satisfies both of these conditions.
	Let $F_1$ and $F_2$ be two disjoint edges in $\cR[S
			\cup \{X_1,X_2\}]$ containing $X_1$ and $X_2$
	respectively.
	Note that such edges can be found due to the large
	degree of the vertices $X_1$ and $X_2$ in
	$\cR[S]$.
	Select vertices $Y_1\in F_1 \sm
		\{X_1\}$ and $Y_2\in F_2 \sm \{X_2\}$, and let $S'$ be
	obtained from~$S$ by deleting the vertices of $F_1
		\cup F_2$ except for~$Y_1$ and~$Y_2$.  This deletion of
	vertices may reduce the minimum degree a little, but
	we are still guaranteed to be able to apply
	Lemma~\ref{lem:connection-dense},
	which tells us that there is a loose $(Y_1,Y_2)$-path
	in $\cR[S']$ of order $4(k-1)+1$.  We then extend this
	path by $F_1$ and $F_2$ to obtain the desired loose
	$(X_1,X_2)$-path~$\cP$.

	Finally, by
	\cref{cla:connection-properties-random-graph}\ref{itm:connection-embedding-lemma},
	we may apply \cref{lem:embedding} to obtain a loose
	$(x'_1,x'_2)$-path $P$  in $G[C]$ with ends
	$x'_1 \in X_1$ and $x'_2\in X_2$ of order $6(k-1) +
		1$ avoiding the set $R$.  This is possible because the $k$-density of~$P$
	is $m_k(P) = 1/(k-1)$.  To finish the proof, we
	augment~$P$ with two edges from $M_1$ and $M_2$ to
	obtain the desired loose $(x_1,x_2)$-path.
\end{proof}

\section{Covering most vertices}
\label{sec:cover-lemma-proof}

This section is dedicated to the proof of \cref{lem:cover-sparse}.
Let us begin with the following outline.

\begin{proof}[Sketch of the proof of \rf{lem:cover-sparse}]
	We begin by reserving a linear sized subset $C \subset V(G')
		\sm Q$ via a random choice.  By
	\cref{lem:connection-sparse}, we can guarantee that a
	constant number of (arbitrary) vertices is pairwise
	connectible
	through loose paths whose inner vertices lie in $C$ and
	avoid~$Q$.

	Next, we cover $V(G') \sm (C \cup Q)$ with a constant number
	of loose paths.  To this end, we apply the Weak Hypergraph
	Regularity Lemma to find a regular partition $\cV$ of
	$V(G')$ together with a reduced dense $k$-graph $\cR$ (whose
	vertices are the clusters of $\cV$) of constant order $r$
	that approximately
	captures the local edge densities of $G'$.  Crucially, most
	$d$-sets of vertices of $\cR$ have $d$-degree at least
	$(\mu_d(k) + \gamma/2)\binom{r-d}{k-d}$.
	To avoid $d$-sets of low degree altogether,
	we find a partition $\cU$ of $\cR$ into parts~$U$ of small
	size such that most parts $U \in \cU$ satisfy
	$\delta_d(\cR[U]) \geq (\mu_d(k) + \gamma/4)
		\binom{|U|-d}{k-d}$.  As~$|U|$ is chosen to be small with
	respect to the fraction of $d$-tuples that fail to have
	large enough $d$-degree, this
	can be done by choosing~$\cU$ randomly (see
	Lemma~\ref{lem:random-subset-degrees}).  Let $\cU' \subset
		\cU$ be the set of parts $U\in\cU$ such that
	$\delta_d(\cR[U])\geq (\mu_d(k) +
		\gamma/4)\binom{|U|-d}{k-d}$.    We may
	assume that the vertices of $Q$ (that have to be avoided)
	are in the parts outside of $\cU'$.  The next task is to
	find in $G'$, for each $U \in \cU'$, a loose path $P_U$ that
	covers most of the clusters of $U$. Then, it will remain to
	connect up the paths~$P_U$ for $U \in \cU'$ into a long loose
	path~$P$ using the reserved vertex set~$C$.

	We may suppose that, with exception of possibly one part,
	every part in the partition~$\cU$ has cardinality~$s$, which
	is chosen to be divisible by $k-1$.  Fix $U \in \cU'$ and
	suppose $|U|=s$.  By the definition of $\mu_d(k)$ and the
	fact that $k-1$ divides $s=|U|$, we have that the induced
	$k$-graph~$\cR[U]$ contains a loose Hamilton cycle~$F$.
	Suppose that the clusters of $U$ are labelled
	$W_1,\dots,W_s$ according to the ordering of~$F$.  To cover
	most vertices of $G'[\bigcup_{1\leq i\leq s} W_i]$, we build
	a path~$P_U$
	edge by edge by selecting for the next edge an edge that
	leaves many possibilities to continue the process.  The
	nature of loose cycles supplies us with a good number of
	possible edge extensions at every stage, whose distribution
	can be captured by a dense ($2$-uniform) regular pair.
	Having observed this, we may extend~$P_U$ without having to
	deal with any major technicalities until it covers most of
	the vertices of~$G'[\bigcup_{1\leq i\leq s}W_i]$.

	Finally, we connect up the paths~$P_U$ for $U \in \cU'$ into
	a long loose path~$P$ using the reserved vertex set~$C$.
	Since the order of~$\cR$ is constant, the
	number of paths to be connected up does not pose any
	problem.  In the same way, we may extend~$P$ to become a
	loose $(u,v)$-path.  This gives the desired almost spanning
	loose path.
\end{proof}

Now we are ready give the complete argument of the proof.

\begin{proof}[{Proof of \cref{lem:cover-sparse}}]
	We define the constants required in the proof in several
	steps.  Given $1\leq d \leq k-1$ with $k\geq 3$ and positive $\eta$
	and~$\gamma$, we introduce constants $\nu,\,  \eps,\, s,\, r_1\, \lambda,\, r_0,
		\,\tau$ with~$s$ divisible by $k-1$ such that
	\begin{align*}
		\eps        & \ll \frac{1}{s},\,\nu \ll \frac{1}{k},\,\gamma
		,\, \eta \,, \text{ and}                                         \\
		\frac{1}{n} & \ll \lambda \ll\frac{1}{r_1} \ll \frac{1}{r_0} \ll
		\eps \ll \tau \ll \frac{1}{k},\,\gamma,\, \delta = \mu_d(k)
	\end{align*}
	and note that the hierarchy of constants
	in~\cref{lem:transfer-cluster-graph-partition} is satisfied.
	Following the quantification
	in~\cref{lem:connection-sparse}, we introduce~$\alpha$ and~$\rho$
	with $\alpha \ll 1/k,\,\gamma$ and  $\rho\ll1/k,\,\nu$.
		{Finally, let $1/K \ll 1/k,\gamma$ and $p \geq \max \{K n^{-(k-d)}\log n,\,  K n^{-(k-2)} \log n\}$ so that we can
			apply~\cref{lem:connection-sparse,lem:random-upper-uniform,lem:count-edges_containing_2-graph}.}
	The following claim contains all properties that we require from the
	random graph.

	\begin{claim} \label{cla:cover-properties-random-graph} W.h.p.~$G \sim
			\oH_k(n,p)$ satisfies the following properties.
		\begin{enumerate}[\upshape(1)]
			\item \label{itm:cover-connection} Let $G' \subset G$ be a
			      spanning subgraph of~$G$ with $\delta_d(G') \geq (\mu_d(k) +
				      \gamma)p\binom{n-d}{k-d}$ and let $Q \subset V(G')$ with $|Q| \leq \alpha
				      n$ be given.
			      Then there is a $\nu n$-set $C \subset V(G')\setminus Q$ with the following property.
			      For any $R \subset C$ with $|R| \leq \rho n$
			      and distinct $u,\,v \in V(G')\sm (Q\cup R) $, there is a loose
			      $(u,v)$-path $P$ of order $8(k-1)+1$ with $V(P)\setminus\{u,v\}
				      \subset C\sm R$.

			\item
			      \label{itm:cover-upper-uniform} $G$ is
			      $(\lambda,p,1+\lambda)$-upper-uniform.

			      \item\label{itm:d-degree-X}
			      For every $X\subset V(G)$ with $|X|\leq(\alpha+\nu)n$, all
			      but~$o(n^d)$ $d$-tuples $D\subset V(G)\setminus X$ have
			      $d$-degree at least
			      $(\mu_d(k)+\gamma/2)p\binom{|V(G)\setminus X|-d}{k-d}$ in~$G-X$.

			\item \label{itm:cover-sparseness} Let $U_1,\dots,U_k$ be disjoint
			      subsets of~$V(G)$ each of size at least ${\eps n/r_1}$. Let $M
				      \subset U_{k-1} \times U_k$ be of size at most $(\tau -
				      \sqrt{\eps})|U_{k-1}||U_k|$.  Then $G$ has at most $(1+ \eps)
				      (\tau - \sqrt{\eps}) p \prod_{j=1}^{k}{|U_j|}$ edges
			      $e=\{v_1,\dots,v_k\}$ with $v_i \in U_i$ for $i \in [k-2]$ and
			      $(v_{k-1},v_k) \in M$.
		\end{enumerate}
	\end{claim}

	\begin{proof}
		Properties~\ref{itm:cover-connection}
		and~\ref{itm:cover-upper-uniform} follow,
		respectively, from
		Lemmas~\ref{lem:connection-sparse}
		and~\ref{lem:random-upper-uniform}.
		Property~\ref{itm:d-degree-X} follows from an
		application of
		Lemma~\ref{lem:random-degree-strong}.
		To show that~\ref{itm:cover-sparseness} holds, we
		start by
		applying \cref{lem:count-edges_containing_2-graph}
		with $\eps/r_1$ and $\eps$ playing the role of
		$\lambda$ and $\eta$ respectively.  Then we know
		that w.h.p.\ $G$ satisfies the conclusion of
		\cref{lem:count-edges_containing_2-graph}, which
		states that for any disjoint subsets
		$U_1,\ldots,U_k$ each of size at least $(\eps/r_1)
			n$ and any $M\subset U_{k-1} \times U_k$ of size at
		least $ C/(pn^{k-3})$, we have $Z(M; U_1,\ldots,U_k)
			= (1 \pm \eps) p |M| \prod_{j=1}^{k-2}{|U_j|}$,
		where we recall that $Z(M; U_1,\ldots,U_k)$ denotes
		the number of edges $e = \{v_1,\ldots, v_k\}$ with
		$v_i\in U_i$ for $i\in [k-2]$ and $(v_{k-1},v_k) \in
			M$.

		Fix disjoints subsets $U_1,\ldots,U_k$ each of size
		at least $(\eps/r_1) n$ and $M\subset U_{k-1} \times
			U_k$ such that $M$ has size at most $(\tau -
			\sqrt{\eps}) |U_{k-1}||U_k|$.
		If~$|M| \geq C/p n^{k-3}$, then from the conclusion
		of \cref{lem:count-edges_containing_2-graph} we
		obtain
		\begin{align*}
			Z(M; U_1,\ldots,U_k) \leq (1+\eps) (\tau - \sqrt{\eps}) p  \prod_{j=1}^{k}{|U_j|},
		\end{align*}
		which finishes the proof in this case. On the other
		hand, if $|M| < C/p n^{k-3}$, then since $1/p
			n^{k-3} \ll n^2$ one can add pairs to~$M$
		to obtain a set $M^{+}\subset U_{k-1}\times U_k $
		such that $C/p n^{k-3} \leq |M^+| \leq (\tau -
			\sqrt{\eps}) |U_{k-1}||U_k|$. Therefore, as before,
		since $Z(M; U_1,\ldots,U_k) \leq Z(M^+;
			U_1,\ldots,U_k)$, we obtain the desired bound
		on~$Z(M; U_1,\ldots,U_k)$, completing the proof of
		our claim.
	\end{proof}

	Let us fix a deterministic graph~$G$ that satisfies
	the properties described
	in~\cref{cla:cover-properties-random-graph}.  Let $G'
		\subset G$ be a spanning subgraph of~$G$ with $\delta_d(G')
		\geq (\mu_d(k) + \gamma ) p \binom{n-d}{k-d}$, and let
	$Q
		\subset V(G')$ with $|Q| \leq \alpha n$ and $u,\,v \in
		V(G')\sm Q$ be given.  We shall find a loose $(u,v)$-path $P
		\subset G'$ in $G'-Q$ that covers all but
	$\eta n$ vertices of $G' - Q$.

	To begin, we reserve a set $C \subset V(G')\setminus
		Q$ of $\nu n$ vertices satisfying the conclusion
	of~\cref{cla:cover-properties-random-graph}\ref{itm:cover-connection}.
	We will use
	$C$ for connections later on.

	By
	\cref{cla:cover-properties-random-graph}\ref{itm:connection-upper-uniform},
	we may apply \cref{lem:sparse-regularity} to $G'$ with
	parameters~$\eps$, $r_0$ and~$r_1$.
	We apply
	\cref{lem:prepartition-regularity,lem:transfer-cluster-graph-partition}
	with threshold~$\tau$ and prepartition $\{C\cup Q,\,V(G)\sm (C
		\cup Q) \}$ to obtain an $r$-vertex graph $\cR$ with
	$r_0 \leq r \leq r_1$.  Denote by~$\cQ$ the set of clusters that
	contain the vertices of $C \cup Q$.  We partition the
	vertex set of $\cR - \cQ$ randomly into parts~$S$ of
	size~$s$ and a `residue' part of size at most~$s$.
	We now invoke
	\cref{cla:cover-properties-random-graph}\ref{itm:d-degree-X}
	with~$X=Q\cup C$ and
	Lemma~\ref{lem:transfer-cluster-graph-partition} to
	note that most $d$-tuples in $\cR-\cQ$ have large
	degree, from which we deduce, by an application of
	Lemma~\ref{lem:random-subset-degrees}, that
	$\delta_d{(\cR[S])} \geq (\mu_d(k) + \gamma/8)
		\binom{s-d}{k-d}$ for all but
	$ (\eta/2)r/s $ parts~$S$ with~$|S|=s$.
	Let~$\cS$ be the set of such parts~$S$.

	\begin{claim}
		For each $S \in \cS$, the induced graph
		$G'\big[\bigcup S\big]$ contains a loose path of
		order at least $(1-2\sqrt{\eps})sn/r$.
	\end{claim}
	\begin{proofclaim}[Proof]
		By the definition of $\mu_d(k)$, there is a loose
		Hamilton cycle~$L$ in $\cR[S]$, whose cyclically
		ordered vertex set we denote by $V_1,\dots,V_s$.
		Without loss of generality, we can assume that~$V_1$
		has degree~$2$ in~$L$.  Let
		$T=\{i\in[s]\colon
			i\equiv1\pmod{k-1}\}$,
		and note
		that~$V_i$ has degree~$2$ in~$L$ if and only
		if~$i\in T$.  In the following, all index
		computations are taken modulo $s$.  To construct a
		loose path $P$ of the desired order, we begin with
		some preliminaries.

		Suppose $W \subset V_1 \cup \dots \cup V_s$ is such
		that $|V_i \sm W| \geq 2\sqrt{\eps} |V_i|$ for every $i
			\in [s]$.  Let $V_i' = V_i \sm W$.  For every $i\in T$,
		let $H(i,W)$ be the bipartite graph with parts $V_i'$
		and $V_{i+k-1}'$ that has an edge $v_iv_{i+k-1}$
		whenever there are vertices $v_{i+1} \in V_{i+1}',\,
			\dots ,\, v_{i+k-2} \in V_{i+k-2}'$ such that
		$\{v_i,\dots, v_{i+k-1}\}$ is an edge in $G'$. We
		claim that
		\begin{center}
			$H(i,W)$ is \emph{$\sqrt{\eps}$-lower regular
				with threshold density $\tau$},
		\end{center}
		which means that for any $A\subset V_i'$ with $|A|
			\geq \sqrt{\eps}|V_i'|$ and $B\subset V_{i+k-1}'$ with
		$|B| \geq \sqrt{\eps} |V_{i+k-1}'|$, we
		have $d(A,B) \geq \tau - \sqrt{\eps}$. To see this,
		fix $i\in T$, consider such sets $A$ and $B$ and
		note that $|A| \geq \sqrt{\eps} |V_i'| \geq \eps
			|V_i|$ and $|B| \geq \sqrt{\eps} |V_{i+k-1}'| \geq
			\eps |V_{i+k-1}|$. For the sake of contradiction,
		assume that $d(A,B) < \tau - \sqrt{\eps}$.  By
		\cref{cla:cover-properties-random-graph}\ref{itm:cover-sparseness}
		and the definition of $H(i,W)$, it follows that
		$d(A,V_{i+1}',\dots,V_{i+k-2}',B) \leq (1+\eps) (\tau
			- \sqrt{\eps}) p \leq (\tau - \sqrt{\eps}/2) p$.  On
		the other hand, since the tuple $(V_i, \dots,
			V_{i+k-1})$ is $(\eps,p)$-regular with
		$d(V_i,\dots,V_{i+k-1}) \geq \tau p$, it follows that
		$d(A,V_{i+1}',\dots,V_{i+k-2}',B) \geq (\tau -
			\eps)p$, which is a contradiction.  This proves that
		$H(i,W)$ is indeed $\sqrt{\eps}$-lower regular with
		threshold density~$\tau$.

		Now we find a loose path $P$ step by step as follows. Suppose that
		we have constructed a path $P_\ell$ with vertices
		$v_1,\dots,v_\ell$ where $v_i \in V_i$ for $i \in
			[\ell]$ and $\ell\in T$.  Suppose moreover that
		$v_\ell$ has degree
		at least $(\tau - 2\sqrt{\eps}) |V_{\ell+k-1} \sm
			V(P_{\ell})|$ in $H(\ell,V(P_\ell))$.  If $V(P_\ell)$
		has more than $(1-2\sqrt{\eps})|V_i|$ vertices in some
		cluster $V_i$, we stop.  In this case $P_{\ell}$ has
		the desired order.

		Otherwise, we select a neighbour $v_{\ell+k-1}\in
			V_{\ell+k-1}$ of~$v_\ell$ in $H(\ell,V(P_\ell))$
		that has degree at least $(\tau -
			2\sqrt{\eps}) |V_{\ell+2k-2}\sm V(P_{\ell})|$ in
		$H(\ell+k-1,V(P_{\ell}))$.
		This is possible because, on the one hand, almost all
		vertices in $H(\ell+k-1,V(P_{\ell}))$ have this
		property as $H(\ell+k-1,V(P_{\ell}))$ is
		$\sqrt{\eps}$-lower-regular
		with threshold density $\tau$, and on the other
		hand~$v_\ell$ has a sufficiently large degree in
		$H(\ell,V(P_\ell))$.  By the definition of
		$H(\ell,V(P_{\ell}))$, there are vertices
		$v_{\ell+1} \in V_{\ell+1},\dots,v_{\ell+k-2} \in
			V_{\ell+k-2}$ such that $\{v_\ell\dots v_{\ell+k-1}\}$
		is an edge in
		$G'-(V(P_{\ell})\setminus
			\{v_{\ell}\})$.
		Hence, we can extend~$P_\ell$ to a loose
		path~$P_{\ell+k-1}$ with vertices
		$v_1,\dots,v_{\ell+k-1}$, with~$v_{\ell+k-1}$ having
		degree at least
		$(\tau - 2\sqrt{\eps}) |V_{\ell+2k-2}\sm V(P_{\ell})|$
		in $H(\ell+k-1,V(P_{\ell}))$.
	\end{proofclaim}

	For each $S \in \cS$, apply the above claim to obtain
	a path~$P_S$ covering most of~$S$.
	To finish the proof of \cref{lem:cover-sparse}, we
	connect up these paths to obtain the
	desired loose $(u,v)$-path~$P$ using~$|\cS|+1$ pairwise
	disjoint loose paths of order $8(k-1)+1$ in~$G'[C]$.
	This is is possible, since by
	\cref{cla:cover-properties-random-graph}\ref{itm:cover-connection}
	we may avoid (within~$C$) up to~$\rho n$ vertices~$R$
	in addition to those of~$Q$.
	It follows that~$P$ covers all but
	$(\nu+\eta/2+2\sqrt{\eps})n \leq \eta n$ vertices of
	$G'-Q$.
\end{proof}

\section{Absorbing vertices}
\label{sec:absorption-lemma-sparse-proof}
This section is dedicated to the proof of
\cref{lem:absorption-sparse}.  The basic idea is to combine many small
absorbing structures into a larger one.  We define the former as
follows.

\begin{definition}[Absorber]\label{def:absorber} Let
	$X=\{x_1,\dots,x_{k-1}\}$ be a set of vertices in a $k$-graph $G$.
	A collection $A=\{P_1^j, P_2^j\}_{j \in [q]}$ is an \emph{absorber
		rooted in $X$} if the following hold:
	\begin{itemize}
		\item $P_i^1,\dots,P_i^q$ are pairwise vertex-disjoint loose paths for $i \in [2]$;
		\item $V(\bigcup_{j \in [q]} P_1^j) = V(\bigcup_{j \in [q]} P_2^j)
			      \cup X$ and $X \cap V(\bigcup_{j \in [q]} P_2^j) =\es$;
		\item $P^j_1$ has the same starting and terminal vertices as $P^j_2$
		      for each $j \in [q]$.
	\end{itemize}
	We refer to $\bigcup_{j \in [q]} P_1^j$ as the absorbers
	\emph{active} state and $\bigcup_{j \in [q]} P_2^j$ as its
	\emph{passive} state.  The \emph{vertices of $A$} are the vertices
	in $\bigcup_{j \in [q]} P_2^j$ and we let $V(A)=\bigcup_{j \in [q]}V(P_2^j)$.
	Note that $V(A)\cap X=\emptyset$.
	The \emph{order} of~$A$ is~$|V(A)|$.
\end{definition}

Next, we define templates, which encode the relative position of the
absorbers in our construction.

\begin{definition}[Template]
	An $r$-graph $T$ is an $(r,z)$-\emph{template} if there is a $z$-set
	$Z \subset V(T)$ such that $T-W$ has a perfect matching for any set
	$W \subset Z$ of size less than $z/2$ with $v(T-W)$ divisible by
	$r$.  We call $Z$ the \emph{flexible} set of $T$.
\end{definition}

Templates were introduced by Montgomery~\cite{Mon14} to adjust
the absorption method of Rödl, Rucinski and Szemerédi~\cite{RRS09a} to
the sparse setting.  The next
lemma was derived by Ferber and Kwan~\cite{FK22} from the work of
Montgomery and states that there exist sparse templates.

\begin{lemma}[Lemma 7.3 in~\cite{FK22}] \label{lem:template}
	For $1/z,\,1/L \ll 1/r$, there is an $(r,z)$-template~$T$
	with $v(T),\,e(T) \leq Lz$ and~$v(T)\equiv0\bmod r$.
\end{lemma}

Consider a graph $G' \subset G$ as in the setting of
\cref{lem:absorption-sparse}, where~$G$ is a typical instance of
$\oH_k(n,p)$.  The next lemma states that a typical linear sized set
$Z \subset V(G')$ has the property that any sublinear set
$W \subset Z$ can be matched into $Z$.\footnote{The statement of
	\cref{lem:richness-property} slightly differs from Lemma~7.5
	in~\cite{FK22}, as we state that a random $\nu n$-set~$Z$ has the
	desired property.  This version follows immediately from their
	proof.}

\begin{lemma}[Lemma 7.5 in~\cite{FK22}] \label{lem:richness-property}
	Let $\eta \ll 1/k,\,1/d,\,\gamma,\,\nu$ and
	$p \geq n^{-(k-1)} \log^3{n}$.  Then w.h.p.\ $G \sim \oH_k(n,p)$ has
	the following property. Let $G' \subset G$ be a spanning subgraph
	with $\delta_d(G') \geq \gamma p \binom{n-d}{k-d}$.  Let
	$Z \subset V(G)$ be chosen uniformly at random among all
	$\nu n$-sets.  Then w.h.p.\ for any $W \subseteq
		V(G)\setminus Z$ with
	$|W| \leq \eta n$, there is a matching in~$G'$ covering all vertices
	in~$W$, each edge of which contains one vertex of~$W$ and $k-1$
	vertices of~$Z$.
\end{lemma}

Finally, we need the following lemma, which provides us with a single
absorber avoiding a small set of fixed vertices.

\begin{lemma}[Sparse Absorber
		Lemma]\label{lem:absorber-sparse}
	Let~$k\geq3$ and suppose $1/M \ll \alpha \ll
		1/k,\,1/d,\,\gamma$ and $1/C \ll 1/k,\gamma$.
	If $p
		\geq\max\{n^{-(k-1)/2+\gamma},Cn^{-(k-d)}\log n\}$,
	then w.h.p.\ $G \sim \oH_k(n,p)$ has the following property.

	For any spanning subgraph $G' \subset G$ with $\delta_d(G')
		\geq (\mu_d(k) + \gamma ) p \binom{n-d}{k-d}$, any set~$Q$ of
	at most~$\alpha n$ vertices and any $(k-1)$-set~$X$ in $V(G)
		\sm Q$, there is an absorber $A$ in~$G'$ rooted in $X$ that
	avoids~$Q$ and has order at most $M$.
\end{lemma}

Lemma~\ref{lem:absorber-sparse} is proved in
Section~\ref{sec:absorber-lemma-sparse-proof}.  Before we come
to the proof of \cref{lem:absorption-sparse}, let us give an
outline of the argument.

\begin{proof}[Sketch of the proof of \rf{lem:absorption-sparse}]
	Consider $\nu$ and~$M$ with $\eta \ll \nu,\,1/M \ll \alpha$.  We
	begin by reserving a set $Z\subset V(G')$ and vertices $u'\in Z$
	and~$v\notin Z$ such that, for any set~$W\subset
		V(G')\setminus Z$ with $|W| \leq \eta n$,
	we may cover $W'=W\cup\{v\}$ with a loose
	$(u',v)$-path~$P_{W'}$ of order at most~$\sqrt\eta n$
	in $G'[Z\cup W']$.  This is possible by applying
	\cref{lem:connection-sparse} with~$Z$ playing the role
	of~$C$ and \cref{lem:richness-property}.

	Next, we use \cref{lem:template} to find a
	$(k-1 , \nu n)$-template~$T$.
	We assume that~$V(T)\subset V(G')$ and that~$Z$ considered
	above is the flexible set of~$T$.  In the next step, we find
	absorbers rooted in the edges of~$T$.
	More precisely, by \cref{lem:absorber-sparse} there is an
	absorber~$A_e$ in~$G'$ rooted in~$e$ of order at most~$M$ for each
	$e\in E(T)$.  Denote by~$\cP_e$ the collection of loose paths
	of~$A_e$ that do not cover~$e$, that is, the paths
	corresponding to the passive state of~$A_e$.
	We can assume that the collections
	$\cP_e$ are pairwise vertex-disjoint.  This can be guaranteed by
	choosing the above absorbers one after another, avoiding the already
	involved vertices with the help of the set~$Q$ in
	Lemma~\ref{lem:absorber-sparse}.  Finally, fix
	$u\in V(G')\setminus\big(V(T)\cup\{v\}\big)$ arbitrarily and
	integrate the absorbers~$A_e$ ($e\in E(T)$) into a single loose
	$(u,u')$-path~$P_1$ with~$V(P_1)\cap\big(V(T)\cup\{v\}\big)=\{u'\}$
	using \cref{lem:connection-sparse} several times.

	Set $A= V(P_1)\cup V(T)\cup\{v\}$.  We claim that $A$ has the
	properties detailed in
	Lemma~\ref{lem:absorption-sparse}.  We may pick the
	involved constants so that $|A| \leq \alpha n$.  For the absorption property, consider an arbitrary subset $W \subset V(G')\setminus A $
	with $|W| \leq \eta n$ and~$|W|$ divisible by $k-1$.
	We have to show that the
	induced graph $G'[A \cup W]$ has a loose $(u,v)$-path covering
	$A \cup W$.

	To that end, let $W' = W \cup \{v\}$.
	By the choice of $Z$, we
	can find a loose $(u',v)$-path $P_{W'}$ of order at most
	$\sqrt{\eta} n$ in $G'[Z \cup W']$ that covers $W'$.  We now
	concatenate~$P_1$ and $P_{W'}$ to obtain a loose $(u,v)$-path~$P$
	that covers~$W$ and uses at most $|Z|/2$ vertices of~$Z$.  Moreover,
	one can check that $|V(T)\sm V(P)|\equiv0\bmod(k-1)$.  Recall
	that~$T$ is a $(k-1,\nu n)$-{template} with flexible set $Z$.  It
	follows that $T-V(P)$ admits a perfect matching~$\cM$.  We then
	`activate' each absorber $A_e$ with $e\in\cM$ and leave all other
	absorbers in their passive state. For each $e\in\cM$, let $\cP_e'$ be
	the collection of paths of~$A_e$ that cover~$e$. Let~$P'$ be obtained from~$P$ by replacing~$\cP_e$ with~$\cP_e'$ for each
	$e\in\cM$.  It follows that~$P'$ is a loose $(u,v)$-path in
	$G'[A\cup W]$ covering $A\cup W$.
\end{proof}

We now give a complete proof of
Lemma~\ref{lem:absorption-sparse}.  Most of the combinatorial
components of the proof are given in the sketch above, but
there are still some numerical facts to check.

\begin{proof}[Proof of \cref{lem:absorption-sparse}]
	Let $\ell=8(k-1)+1$ and introduce further constants as
	follows:
	let
	\begin{equation*}
		\frac1z,\,\frac1L\ll\frac1{k-1}
	\end{equation*}
	be as in \cref{lem:template} with $r=k-1$.  Let
	\begin{equation*}
		\frac1M\ll\alpha\ll\frac1k,\,\frac1d,\,\gamma
	\end{equation*}
	be as in Lemmas~\ref{lem:connection-sparse}
	and~\ref{lem:absorber-sparse}.  Now let~$\nu$ be such that
	\begin{equation*}
		\nu\ll\alpha,\,\frac1L,\,\frac1M,\,\frac1\ell,
	\end{equation*}
	and let
	\begin{equation*}
		\eta \ll \rho \ll\frac1k,\,\frac1d,\,\gamma,\,\nu
	\end{equation*}
	with $\rho$ as in Lemma~\ref{lem:connection-sparse} and $\eta $ as in Lemma~\ref{lem:richness-property}.
	Finally, let  $1/K \ll 1/k,\gamma$ as in
	Lemma~\ref{lem:connection-sparse} and $p
		\geq\max\{n^{-(k-1)/2+\gamma},Kn^{-(k-d)}\log n\}$, which allows us to apply
	Lemmas~\ref{lem:connection-sparse},
	\ref{lem:richness-property} and~\ref{lem:absorber-sparse}.
	We can now prove the following claim.

	\begin{claim}
		\label{cla:absorption-properties-random-graph} W.h.p.\
		$G \sim \oH_k(n,p)$ satisfies the following property.  For any
		spanning subgraph~$G'$ of~$G$ with
		$\delta_d(G') \geq (\mu_d(k) + \gamma ) p \binom{n-d}{k-d}$ and
		any $Q \subset V(G')$ with $|Q| \leq \alpha n$ the following hold:
		\begin{enumerate}[\upshape(1)]\setcounter{enumi}{-1}
			\item \label{itm:absorption-connection0}
			      Let~$C=V(G')\setminus Q$.  For every pair of distinct
			      vertices $u,\,v\in C$, there is a loose
			      $(u,v)$-path~$P$ of order $\ell=8(k-1)+1$ in~$G'$ with
			      $V(P)\subset C$.
			\item \label{itm:absorption-connection}
			      Let $C \subset V(G')\setminus Q$ be a set of size $\nu n$ taken
			      uniformly at random.
			      Then with probability at least $2/3$ the following holds.
			      For any $R \subset C$ with $|R| \leq \rho n$ and distinct $u,v \in V(G')\setminus (Q\cup R) $,
			      there is a loose $(u,v)$-path $P$ of order
			      $\ell=8(k-1)+1$ in~$G'$ with $V(P)\setminus\{u,v\}
				      \subset C\sm R$.

			\item \label{itm:absorption-richness}
			      Let $Z \subset V(G)$ be chosen uniformly at random among all $\nu n$-sets.
			      Then {with probability at least $2/3$} for any $W \subseteq
				      V(G)\setminus Z$ with $|W| \leq \eta n$, there is a matching in $G'$
			      covering all vertices in $W$, each edge of which contains one
			      vertex of $W$ and $k-1$ vertices of $Z$.

			\item \label{itm:absorption-absorber} For any $(k-1)$-set~$X$ in
			      $V(G) \sm Q$, there is an absorber~$A$ in~$G'$ rooted in
			      $X$ that avoids~$Q$ and has order at most~$M$.\qed
		\end{enumerate}
	\end{claim}

	Fix a graph $G$ that satisfies the properties in
	Claim~\ref{cla:absorption-properties-random-graph} and let~$G'$ be a
	spanning subgraph of~$G$ with minimum $d$-degree
	$\delta_d(G') \geq (\mu_d(k) + \gamma) p {n-d \choose k-d}$.  Our
	goal is to find $A \subset V(G')$ and two vertices $u,v \in A$ that
	satisfy the properties detailed in
	Lemma~\ref{lem:absorption-sparse}.  Note that we may
	assume that~$n$ is large whenever necessary.  We begin
	with the following claim.

	\begin{claim}\label{cla:absorption-path}
		There is a $\nu n$-set $Z\subset V(G')$, a vertex $u'\in Z$ and a
		vertex $v\in V(G')\setminus Z$ with the following property.  For
		every subset $W'\subset V(G)\setminus Z$ with $v\in W'$ of size at most
		${\eta} n+1$, there is a loose $(u',v)$-path~$P_{W'}$ of order at most
		$\sqrt{\eta} n$ in $G'[Z \cup W']$ that covers~$W'$.
	\end{claim}
	\begin{proof}
		Using the union bound, we obtain $Z\subset V(G')$
		satisfying both
		\cref{cla:absorption-properties-random-graph}\ref{itm:absorption-connection}
		with $Z$ and $Q$, where $\es$ plays the role of $C$, and
		\cref{cla:absorption-properties-random-graph}\ref{itm:absorption-richness}.
		Fix any $u'\in Z$ and any $v\in V(G')\setminus Z$.
		Now consider an arbitrary subset $W'\subset
			V(G)\setminus Z$ with $v\in W'$ of size at most
		$\eta n+1$.  We use
		\cref{cla:absorption-properties-random-graph}\ref{itm:absorption-richness}
		to find a matching $\cM$ in $G'[W'\cup Z]$ covering
		all vertices in~$W'$, each edge of which contains
		one vertex of~$W'$ and $k-1$ vertices of~$Z$.  Next,
		we use
		\cref{cla:absorption-properties-random-graph}\ref{itm:absorption-connection}
		to connect up the edges of~$\cM$ to a loose
		$(u',v)$-path in $G'[Z\cup W']$.  This can be done
		by involving at most $2\ell|\cM| \leq \sqrt{\eta}n
			\leq \rho n $ vertices in total.  (Note that the set
		$R$ of
		\cref{cla:absorption-properties-random-graph}\ref{itm:absorption-connection}
		can be used to keep the paths from overlapping in
		the wrong vertices.)
	\end{proof}

	Let~$Z$, $u'$ and~$v$ be as in
	\cref{cla:absorption-path}.  By \cref{lem:template}
	applied with $r=k-1$ and $z=\nu n$, there is a
	$(k-1 , \nu n)$-template $T$ on at most
	$L\nu n\leq(\alpha/3)n$ vertices and at most $L \nu n$
	edges and with~$v(T)\equiv0\bmod(k-1)$.  We
	inject~$V(T)$ into~$V(G')\setminus\{v\}$, mapping the
	flexible set of $T$ onto $Z \subset V(G')$.  For
	convenience, let us identify~$V(T)$ with its image
	under the injection~$V(T)\to V(G')$; thus, in what
	follows, we think of~$V(T)$ as a subset of~$V(G')$
	and~$Z\subset V(T)$ is the flexible set in the
	template~$T$.  Next, we find absorbers rooted in the
	edges of~$T$.  More precisely, by
	\cref{cla:absorption-properties-random-graph}\ref{itm:absorption-absorber}
	there is an absorber~$A_e$ in~$G'$ rooted in~$e$ of
	order at most~$M$ for each $e \in E(T)$.
	Denote by $\cP_e$ the collection of
	loose paths of $A_e$ that do not cover $e$.  We can assume that the
	collections $\cP_e$ are pairwise vertex-disjoint and that they are
	disjoint from~$V(T)$.  This can be guaranteed by choosing the above
	absorbers one after another, avoiding the at most
	$M(e(T)-1)+v(T)\leq2LM \nu n \leq \alpha n$ already involved
	vertices with the help of the set $Q$.

	Next, we integrate these absorbers into a single path.
	Fix an arbitrary vertex
	$u\in V(G) \sm\big(V(T)\cup\bigcup_{e\in
			E(T)}V(A_e)\cup\{v\}\big)$.  Now use
	\cref{cla:absorption-properties-random-graph}\ref{itm:absorption-connection0}
	several times to connect up all paths of all
	collections $\cP_e$ for $e \in E(T)$ to one loose
	$(u,u')$-path $P_1$ of order at most
	$2 \ell Me(T) \leq 2\ell ML \nu n\leq(\alpha/3)n$
	with~$V(P_1)\cap\big(V(T)\cup\{v\}\big)=\{u'\}$.  Note
	that this is possible by choosing the connecting paths
	(of order~$\ell$) one after another, avoiding the at
	most~$\alpha n$ already used vertices (which play the
	part of $Q$ in
	\cref{cla:absorption-properties-random-graph}\ref{itm:absorption-connection0}).

	Set $A= V(P_1)\cup V(T)\cup\{v\}$.  We claim that~$A$ has the
	properties detailed in the statement of
	Lemma~\ref{lem:absorption-sparse}.  Clearly,
	$|A| \leq \alpha n$.  For the second part, consider an arbitrary
	subset $W \subset V(G')\setminus A$ with $|W| \leq \eta n$ divisible
	by $k-1$.  We have to show that the induced graph $G'[A \cup W]$ has
	a loose $(u,v)$-path covering $A \cup W$.

	To this end, let $W' = W \cup \{v\}$.  We now use
	\cref{cla:absorption-path} to find a loose $(u',v)$-path $P_{W'}$ of
	order at most $\sqrt{\eta} n$ in $G'[Z \cup W']$ that covers~$W'$.
	We concatenate~$P_1$ and~$P_{W'}$ to obtain a loose $(u,v)$-path $P$
	that covers~$W$ and uses at most $|Z|/2$ vertices of~$Z$.  {Note
			that
			$V(T)\cap V(P)=V(P_{W'})\setminus(W\cup\{v\})\equiv0\bmod(k-1)$
			because $|W\cup\{v\}|\equiv1\bmod{(k-1)}$ and
			$|V(P_{W'})|\equiv1\bmod(k-1)$.}  It follows that
	$V(T)\setminus V(P)$ is divisible by~$k-1$,
	as~$v(T)\equiv0\bmod(k-1)$.  Recall that $T$ is a
	$(k-1,\nu n)$-{template} with flexible set $Z$.  Hence $T- V(P)$
	admits a perfect matching~$\cM$.  We then `activate' each absorber
	$A_e$ with $e \in\cM$ and leave all other absorbers in their passive state.
	Moreover, for each $e \in\cM$, let $\cP_e'$ be the collection of
	paths of $A_e$ that cover~$e$.  Finally, we obtain a path~$P'$
	from~$P$ by replacing~$\cP_e$ by~$\cP_e'$ for each $e\in\cM$.  It
	follows that~$P'$ is a loose $(u,v)$-path in $G[A\cup W]$ covering
	$A \cup W$.
\end{proof}

\section{Absorbers in the sparse setting}\label{sec:absorber-lemma-sparse-proof}
Here we prove \rf{lem:absorber-sparse}.
The idea of the proof is to find an absorbing structure $A$ in the dense setting and embed it in the random graph via  an embedding lemma for hypergraph regularity.
One issue arising in this approach is that the probability $p$ depends on the `densest spots' of $A$.
So to minimise $p$, we have to find an absorbing structure $A$ whose edges are nowhere too cluttered.
This can be formalised by the following definition, which we
adopt from Ferber and Kwan~\cite{FK22}.

A \emph{Berge cycle} in a (possibly non-uniform) hypergraph is
a sequence of distinct edges $e_1,\dots,e_\ell$ such that
there exist distinct vertices $v_1,\dots,v_\ell$ with $v_i\in
	e_i\cap e_{i+1}$ for all $i$ (where $e_{\ell+1}=e_1$). The
\emph{length} of such a cycle is
its number of edges $\ell$. The
\emph{girth} of a  hypergraph is the length of the shortest
Berge cycle it contains (if the hypergraph  contains no Berge
cycle we say it has infinite girth, or is  \emph{Berge
	acyclic}). We say that a $k$-uniform absorber rooted in a $(k-1)$-set
$X$ is \emph{$K$-sparse} if it has  girth
at least $K$, even after adding the extra edge
$X$.\footnote{Note that adding $X$ as an edge results in an non-uniform hypergraph, but this is in accordance with the definition of the girth.}

The following lemma allows us to find a sparse absorber in the dense setting.
Its proof is deferred to \cref{sec:absorber-lemma-dense-proof}.

\begin{lemma}[Dense Absorber Lemma]\label{lem:absorber-dense}
	Let $1/n \ll\eta\ll1/M \ll \gamma,\, 1/k,\, 1/d,\,1/K$.  Let
	$G$ be an $n$-vertex $k$-graph in which all but $\eta
		n^d$
	$d$-sets have degree at least
	$(\mu_d(k) + \gamma ) \binom{n-d}{k-d}$.  Let $X$ be a
	$(k-1)$-set of vertices each of degree at least
	$(\mu_d(k) + \gamma ) \binom{n-1}{k-1}$.  Then~$G$ contains
	a $K$-sparse absorber~$A$ rooted in $X$ of order at
	most~$M$.
\end{lemma}

The proof of \rf{lem:absorber-sparse} can be found in \cref{sec:absorber-lemma-sparse-proof-appendix}.
In the following, we give an outline of the argument.

\begin{figure}

	\tikzset{every picture/.style={line width=0.75pt}} 

	\tikzset{every picture/.style={line width=0.75pt}} 

	\tikzset{every picture/.style={line width=0.75pt}} 

	\begin{tikzpicture}[x=0.75pt,y=0.75pt,yscale=-1,xscale=1,scale=0.6]

		\draw    (21,159) -- (147.23,27) ;
		\draw    (147.23,27) -- (148,80) ;
		\draw    (148,80) -- (148,148) ;
		\draw    (148,205) -- (148,278) ;
		\draw    (21,159) -- (148,278) ;
		\draw [color=myred  ,draw opacity=1 ][line width=2.25]    (133,98.5) -- (133.95,192.5) ;
		\draw [shift={(134,197.5)}, rotate = 269.42] [fill=myred  ,fill opacity=1 ][line width=0.08]  [draw opacity=0] (14.29,-6.86) -- (0,0) -- (14.29,6.86) -- cycle    ;
		\draw [color=myblue  ,draw opacity=1 ][line width=2.25]    (140.12,222.5) -- (141,260.5) -- (36,158.5) -- (142,45.5) -- (142,72.5) ;
		\draw [shift={(140,217.5)}, rotate = 88.67] [fill=myblue  ,fill opacity=1 ][line width=0.08]  [draw opacity=0] (14.29,-6.86) -- (0,0) -- (14.29,6.86) -- cycle    ;
		\draw  [dash pattern={on 6.75pt off 4.5pt}][line width=2.25]  (367,96.6) .. controls (367,79.7) and (380.7,66) .. (397.6,66) -- (489.4,66) .. controls (506.3,66) and (520,79.7) .. (520,96.6) -- (520,215.9) .. controls (520,232.8) and (506.3,246.5) .. (489.4,246.5) -- (397.6,246.5) .. controls (380.7,246.5) and (367,232.8) .. (367,215.9) -- cycle ;
		\draw [color=myblue  ,draw opacity=1 ][line width=1.5]  [dash pattern={on 5.63pt off 4.5pt}]  (366,131.5) -- (147.23,27.74) -- (366,89.5) ;
		\draw [shift={(250.47,76.71)}, rotate = 25.37] [fill=myblue  ,fill opacity=1 ][line width=0.08]  [draw opacity=0] (11.61,-5.58) -- (0,0) -- (11.61,5.58) -- cycle    ;
		\draw [shift={(263.16,60.47)}, rotate = 195.76] [fill=myblue  ,fill opacity=1 ][line width=0.08]  [draw opacity=0] (11.61,-5.58) -- (0,0) -- (11.61,5.58) -- cycle    ;
		\draw [color=myred  ,draw opacity=1 ][line width=2.25]  [dash pattern={on 6.75pt off 4.5pt}]  (394,83.5) -- (496,130.5) ;
		\draw [color=myred  ,draw opacity=1 ][line width=2.25]  [dash pattern={on 6.75pt off 4.5pt}]  (396,142.5) -- (482,178.5) ;
		\draw [color=myred  ,draw opacity=1 ][line width=2.25]  [dash pattern={on 6.75pt off 4.5pt}]  (400,188.5) -- (476,226.5) ;
		\draw [color=myblue  ,draw opacity=1 ][line width=2.25]  [dash pattern={on 6.75pt off 4.5pt}]  (395,130.5) -- (471,85.5) ;
		\draw [color=myblue  ,draw opacity=1 ][line width=2.25]  [dash pattern={on 6.75pt off 4.5pt}]  (423,175.5) -- (508,116) ;
		\draw [color=myblue  ,draw opacity=1 ][line width=2.25]  [dash pattern={on 6.75pt off 4.5pt}]  (392,228.5) -- (501,193.5) ;

		\draw    (148,148.51) -- (148.97,205) ;
		\draw [color=myred  ,draw opacity=1 ][line width=1.5]  [dash pattern={on 5.63pt off 4.5pt}]  (365,182.5) -- (148,148.51) -- (369,147.5) ;
		\draw [shift={(249.78,164.45)}, rotate = 8.9] [fill=myred  ,fill opacity=1 ][line width=0.08]  [draw opacity=0] (11.61,-5.58) -- (0,0) -- (11.61,5.58) -- cycle    ;
		\draw [shift={(265.3,147.97)}, rotate = 179.74] [fill=myred  ,fill opacity=1 ][line width=0.08]  [draw opacity=0] (11.61,-5.58) -- (0,0) -- (11.61,5.58) -- cycle    ;

		\draw (6,119) node [anchor=north west][inner sep=0.75pt]   [align=left] {$ $};
		\draw (1,116.4) node [anchor=north west][inner sep=0.75pt]  [font=\large]  {$x_{i}$};
		\draw (95,3.9) node [anchor=north west][inner sep=0.75pt]  [font=\large]  {$x_{i}^{j}$};
		\draw (159,64.4) node [anchor=north west][inner sep=0.75pt]  [font=\large]  {$x_{i}^{k-1}$};
		\draw (162,185) node [anchor=north west][inner sep=0.75pt]  [font=\large]  {$x_{i}^{2k-2}$};
		\draw (310,222.4) node [anchor=north west][inner sep=0.75pt]  [font=\Large]  {$A^{j}$};
		\draw (70,130) node [anchor=north
			west][inner sep=0.75pt]
		[font=\Large,color=myred ] {$P_{2}^i$};
		\draw (30,50) node [anchor=north west][inner
			sep=0.75pt] [font=\Large,color=myblue ]
		{$P_{1}^i$};

		\node [draw,fill=black,circle,inner sep=2pt] at (147.23,27) {};
		\node [draw,fill=black,circle,inner sep=2pt] at (148,80) {};
		\node [draw,fill=black,circle,inner sep=2pt] at (21.97,159) {};
		\node [draw,fill=black,circle,inner sep=2pt] at (148,148) {};
		\node [draw,fill=black,circle,inner sep=2pt] at (148,205) {};
		\node [draw,fill=black,circle,inner sep=2pt] at (148,278) {};

	\end{tikzpicture}

	\caption{The absorber appearing in the proof of \cref{lem:absorber-sparse}.
		Note that the path $P_1^i$ and the paths of type $P_1$ of the absorber $A_j$ (dashed) cover all (depicted) vertices, whereas the path $P_2^i$ and the paths of type $P_2$ of $A_j$ (dashed) leaves $x_i$ uncovered.
	}
	\label{fig:complete-absorber}

\end{figure}
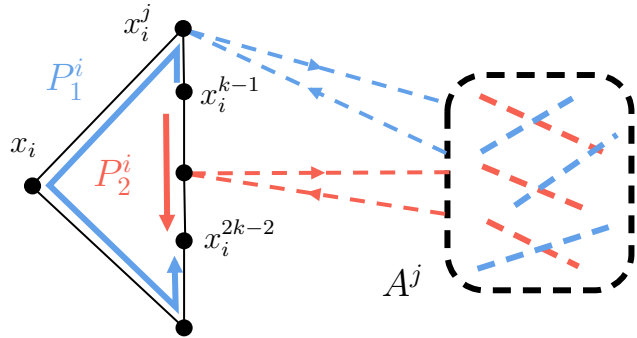

\begin{proof}[Sketch of the proof of \cref{lem:absorber-sparse}]
	Let $X=\{x_1,\dots,x_{k-1}\}$, and suppose that $Q $ is empty for simplicity.
	We begin by finding pairwise vertex-disjoint loose cycles $F_1,\dots,F_{k-1}$ each of order $16(k-1)$ such that $x_i$ is a vertex of degree $2$ in $F_i$.
	This is possible by applying \cref{lem:connection-sparse} twice for each vertex of~$X$.
	For the following argument, it is crucial that each
	of these cycles has the same order.
	We also remark that in the actual proof, we
	choose~$\Omega(n)$ such cycles for each vertex in~$X$,
	as this lets us have more flexibility when targeting certain
	objects with our embeddings.  However, in this sketch, we omit the
	details of this technical argument.

	For $i \in [k-1]$, denote the vertices of $F_i$ by
	$\,v_{i}^1,\dots,v_{i}^\ell,x_i$ along the
	cyclic ordering of~$F_i$, where $\ell = 16(k-1)-1$.
	We define $(k-1)$-sets~$Y_j$ with $j \in [\ell]$ by setting $Y_j=\{v_{1}^j,\dots,v_{k-1}^j\}$.
	Next, we find a set of pairwise disjoint absorbers
	$A_1,\dots,A_\ell$ of constant order where each~$A_j$
	is rooted in~$Y_j$.  This is possible by a
	contraction argument of Ferber and Kwan~\cite{FK22}
	based on weak hypergraph regularity, and the fact that
	we can find~$\Omega(n)$ loose cycles~$F_i$
	containing~$x_i\in X$: intuitively, the fact that we
	have many choices for the~$F_i$ gives us many choices for
	the~$Y_j$, and thus one is able to find
	the~$A_j$ for a suitable choice of the~$F_i$.
	The details of this step can be found in~\cref{sec:absorber-lemma-sparse-proof-appendix}.
	Here, we only note that during this step, we apply \cref{lem:absorber-dense} to find a copy of $A_j$ in a suitable reduced graph, which we then embed into~$G'$ using a sparse embedding lemma.
	Since the absorbers~$A_j$ are suitably sparse, the
	hypothesis $p \geq n^{-(k-1)/2+\gamma}$ is enough to
	let us find the embedding.

	We now describe the absorber rooted in~$X$, which is also illustrated in \cref{fig:complete-absorber}.
	To begin, we define two additional
	loose paths for each $i \in [k-1]$ as follows.  Let
	$P_1^i$ be the loose $(v_{i}^{k-1},v_{i}^{2k-2})$-path
	of order $15(k-1)-1$ with vertex ordering
	\begin{align*}
		v_{i}^{k-1},\, v_{i}^{k-2},\, \dots,\, v_{i}^1,\, x_i,\, v_{i}^\ell,\, v_{i}^{\ell-1},\, \dots, v_{i}^{2k-2}.
	\end{align*}
	Let $P_2^i$ be the loose $(v_{i}^{k-1},v_{i}^{2k-2})$-path consisting of a single edge with vertex ordering
	\begin{align*}
		v_{i}^{k-1},\, v_{i}^{k}, \dots,\,  v_{i}^{2k-2}.
	\end{align*}
	We then set $B_i = \{P_1^i,P_2^i\}$.

	To finish, we claim that the collection
	$A_1,\dots,A_\ell$, $B_1,\dots,B_{k-1}$ is an absorber
	rooted in~$X$.
	Indeed, if $X$ shall be absorbed, we take the paths $P_1^i$ from each of the absorbers $B_i$.
	This covers $X$ and the vertices of each $Y_i$ except for $v_{i}^{k}, \dots, v_{i}^{2k-3}$.
	To cover the latter, we `activate' the absorbers $A_{k}, \dots, A_{2k-3}$.
	All other absorbers~$A_j$ are left in their passive state.

	If, on the other hand, $X$ shall not be absorbed, we take the paths~$P_2^i$ from each of the absorbers~$B_i$.
	This leaves only $X$ and the vertices $v_{i}^{2k-1},\dots,v_{i}^{\ell},v_{i}^{1},\dots,v_{i}^{k-2}$ uncovered.
	To include the latter,  we `activate' the absorbers $A_{2k-1},\dots,A_{\ell},A_{1},\dots, A_{k-2}$.
	All other absorbers~$A_j$ are left in their passive state.
	Since it is clear from the construction that these absorbers consist of paths with shared starting and ending points, this finishes the proof.
\end{proof}

\subsection{Details of the proof}\label{sec:absorber-lemma-sparse-proof-appendix}
Let us now prove~\cref{lem:absorber-sparse} in detail.  Our first
result, Lemma~\ref{lem:sparse-to-mkdensity}, tells us that a
`book'
version of the $K$-sparse absorber of \cref{lem:absorber-dense} has
low~$k$-{density}.  The proof of Lemma~\ref{lem:sparse-to-mkdensity}
is deferred to the end of this section.

\begin{lemma}
	\label{lem:sparse-to-mkdensity}
	Suppose {$1/K \ll 1/T, \,\gamma,\, 1/k$} and let~$H$ be the union of~$T$
	copies of a $K$-sparse $k$-uniform absorber that are all rooted in
	the same $(k-1)$-set but are otherwise disjoint.  Then
	$m_{k}(H) \leq 2/(k-1)+\gamma$.
\end{lemma}

Note that, in Lemma~\ref{lem:sparse-to-mkdensity}, we consider~$H$ to
be the $k$-graph with vertex set $X\cup\bigcup_{1\leq j\leq T}V(A_j)$,
where~$X$ is the $(k-1)$-set that is the common root of the otherwise
disjoint $K$-sparse absorbers~$A_j$ ($1\leq j\leq T$).  Furthermore,
the edges of~$H$ are the edges that occur in the paths that define
the absorbers~$A_j$.

One drawback of \rf{lem:embedding} is that we cannot force the
embedded copy of the desired graph to contain a specified
collection of vertices (in particular, we cannot force the
embedding of the absorber of Lemma~\ref{lem:absorber-dense} to
have a specified root set~$X$ in the embedding).  However,
because we are guaranteed to find canonical copies, we can
force specified vertices to belong to specified sets of
vertices.  This fact can be exploited to find an absorber with
a specified root~$X$ in the sparse setting using
Lemma~\ref{lem:embedding}.  We employ an idea introduced by
Ferber and Kwan~\cite{FK22}: we apply \cref{lem:embedding} to
embed a book version of the absorber of
Lemma~\ref{lem:absorber-dense} (as in
Lemma~\ref{lem:sparse-to-mkdensity}).  The embedding is not
quite realized in the original sparse hypergraph, but in
the hypergraph after contracting certain sets of vertices.
We then show that, by
expanding these contracted vertices, we have an absorber with
the desired root~$X$.

\begin{definition}[Contracted graph]\label{def:contracted-graph}
	Let $G$ be a $k$-graph, let $\cF \subset V(G)^\ell$ be a
	family of disjoint $\ell$-tuples, and let $\cP$ be a family
	of disjoint sets $U_1,\dots,U_\ell \subset V(G)$ such that
	the tuples of $\cF$ and the sets in $\cP$ do not share
	vertices.

	Let $G(\cF,\cP)$ be the $k$-graph obtained from
	$G[U_1]\cup\dots\cup G[U_\ell]$ by adding for each $\ell$-tuple
	$\vecb v \in \cF$ a new vertex $w_{\vecb v}$.  Moreover, for
	each $\vecb v = (v_1,\dots,v_\ell) \in \cF$ and
	each $1\leq i\leq\ell$, for each $(k-1)$-set $f \subset U_i$
	with $ f \cup \{v_i\} \in E(G)$, add the edge
	$f \cup \{w_{\vecb v}\}$ to $G(\cF,\cP)$.
\end{definition}

We remark that~$G(\cF,\cP)$ can be obtained from a \textit{certain
	subgraph of}~$G\big[\bigcup\cF\cup\bigcup_{1\leq
			i\leq\ell}U_i\big]$ by collapsing each member~$\vecb v$
of~$\cF$ to a vertex~$w_{\vecb v}$.

\begin{proof}[Proof of \cref{lem:absorber-sparse}]
	We introduce the constants required in the proof in steps.  Let
	$T=\ell=16(k-1)-1$ and $\beta = 1/(3\ell)$.  Let
	$\alpha\ll  1/k,\,\gamma$ following
	Lemmas~\ref{lem:degree-inheritance_induced-subgraphs}
	and~\ref{lem:connection-sparse}.
	We now suppose
	$\nu\ll\alpha,\,\beta,\,1/k$.

	Suppose $1/K\ll\gamma,\,1/T$ as in
	Lemma~\ref{lem:sparse-to-mkdensity} and
	$\eta\ll1/M\ll\gamma,\,1/k,\,1/d,\,1/K$ as in
	Lemma~\ref{lem:absorber-dense}.  We now suppose
	$1/r_0,\,\tau\ll1/k,\,\eta,\,\beta$ such that in particular $r_0/(2\ell)$ can play the role of $n$ when applying Lemma~\ref{lem:absorber-dense} with the above constants.  With~$\tau$ at hand, we apply
	Lemma~\ref{lem:embedding} to obtain~$\epsilon$ and~$\zeta$ that
	allow us to embed any linear hypergraph with at most $M/2+k-1$
	vertices in any suitable `$\epsilon$-regular system'.  Finally, we
	apply Lemma~\ref{lem:transfer-cluster-graph-partition} to obtain
	\begin{equation*}
		\frac{1}{n}\ll\lambda \ll \frac1{r_1}
		\ll\frac1{r_0},\,\epsilon,\,\frac1k,\,\frac1d,\,\delta=\mu_d(k)
	\end{equation*}
	and let
	\begin{equation*}
		\kappa=\frac1{r_1}\big((k-1)\nu+\ell\beta\big)
		=\frac1{r_1}\left((k-1)\nu+\frac13\right).
	\end{equation*}
	We shall apply Lemma~\ref{lem:embedding} with this value
	of~$\kappa$.
	Finally, let $1/D \ll 1/k,\gamma$  and $p \geq\max\{n^{-(k-1)/2+\gamma},Dn^{-(k-d)}\log n\}$ so that we can apply \cref{lem:random-upper-uniform,lem:embedding,lem:connection-sparse,lem:degree-inheritance_induced-subgraphs} with with $k$-density $2/(k-1)+\gamma/2$.
	In the following, we also assume that $n$ is large enough to satisfy \cref{lem:transfer-cluster-graph-partition} with the above constants.

	The next claim summarises the properties that we require from the random graph.

	\begin{claim} \label{cla:absorber-sparse-properties-random-graph}
		W.h.p.\ $G \sim \oH_k(n,p)$ satisfies each of the following
		properties.
		\begin{enumerate}[\upshape(1)]
			\item \label{itm:absorber-upper-uniform} For each $\cF$ and~$\cP$ as
			      in \cref{def:contracted-graph} with $|\cF|=(k-1)\nu n$ and
			      $|U_j|=\beta n$ for every~$1\leq j\leq\ell$, the $k$-graph
			      $G(\cF,\cP)$ is $(\lambda,p,1+\lambda)$-upper-uniform.
			\item \label{itm:absorber-embedding-lemma} Furthermore, $G(\cF,\cP)$
			      satisfies the conclusion of \rf{lem:embedding} for embedding any
			      linear $k$-graph~$H$ with $m_k(H) \leq 2/(k-1) + {\gamma/2}$ on at
			      most~$M/2+k-1$ vertices with parameters~$\epsilon$, $\zeta$, $\tau$
			      and~$\kappa$.
			\item \label{itm:absorber-connection} Let $G'$ be a spanning
			      subgraph of~$G$ with
			      $\delta_d(G') \geq (\mu_d(k) + \gamma ) p \binom{n-d}{k-d}$ and
			      fix $Q \subset V(G')$ with $|Q| \leq 2\alpha n$.  For any distinct
			      $u,v \in V(G')\setminus Q$, there is a loose $(u,v)$-path~$P$ of
			      order $8(k-1)+1$ in $G'-Q$.
			\item \label{itm:absorber-degree-subgraph} Let~$G'$ be a spanning
			      subgraph of~$G$ with
			      $\delta_d(G') \geq (\mu_d(k) + \gamma ) p \binom{n-d}{k-d}$ and
			      fix~$Q\subset V(G')$ with $|Q|\leq2\alpha n$.
			      Let~$U\subseteq V(G')\sm Q$ with $|U|=\beta n$ be chosen uniformly
			      at random.  Then w.h.p.\ all but~$o(n^d)$ of the $d$-sets of
			      vertices in~$V(G') \sm Q$ have $d$-degree at least
			      $(\mu_d(k) +\gamma/2) p \binom{\beta n -d}{ k-d}$
			      into~$U$.
		\end{enumerate}
	\end{claim}
	\begin{proof}
		To begin, we remark that $G (\cF, \cP)$ can be coupled to a subgraph
		of the binomial random $k$-graph $G \sim \oH_k(n,p)$ from which it
		is obtained as a contracted graph.  This is because for each
		possible edge~$e$ of~$G (\cF, \cP)$, there is exactly one $k$-set
		$e'$ whose presence as an edge in~$G$ determines whether or not~$e$
		is in~$G(\cF,\cP)$.

		Furthermore, there are $\exp(O(n\log n))$ ways to choose~$\cF$
		and~$\cP$.  By the choice of~$p$, we can apply
		Lemma~\ref{lem:random-upper-uniform} and take the union
		bound over all possibilities for~$\cF$ and~$\cP$ to
		deduce~\ref{itm:absorber-upper-uniform}.

		To prove~\ref{itm:absorber-embedding-lemma}, it suffices to combine
		the argument above with
		\cref{lem:embedding}.
		Assertion~\ref{itm:absorber-connection} follows from
		Lemma~\ref{lem:connection-sparse}.
		For~\ref{itm:absorber-degree-subgraph}, we
		apply~\cref{lem:degree-inheritance_induced-subgraphs} with $2\alpha,\beta$ playing the role of $\alpha,\sigma$.
	\end{proof}

	Consider a graph~$G$ that satisfies the properties in
	Claim~\ref{cla:absorber-sparse-properties-random-graph}.  We
	regard~$G$ as deterministic graph from now on.  Let~$G'$ be a spanning
	subgraph of~$G$ with
	$\delta_d(G') \geq (\mu_d(k) + \gamma ) p \binom{n-d}{k-d}$.
	Let~$Q\subset V(G')$ be a set of at most~$\alpha n$ vertices, and let
	$X = \{x_1,\dots,x_{k-1}\} \subset V(G) \sm Q$.  We have to show that
	there is an absorber~$A$ rooted in~$X$ that avoids~$Q$ and has order
	at most~$M$.

	\begin{claim}\label{cla:cycle-cover-X}
		For every vertex $x\in X$ we can fix a family~$\cC_x$ of~$\nu n$
		loose cycles in~$G'-Q$ of order $\ell+1=16(k-1)$ each containing~$x$
		as a vertex of degree~$2$ but disjoint otherwise.  Furthermore, we
		may require every cycle in~$\cC_x$ and every cycle in~$\cC_{x'}$ to
		be disjoint for all distinct~$x$ and~$x'$ in~$X$.
	\end{claim}
	\begin{proof}
		For any~$x\in X$, we can apply
		\cref{cla:absorber-sparse-properties-random-graph}\ref{itm:absorber-connection}
		twice to obtain a cycle of order $\ell+1$ containing~$x$ as a vertex
		of degree~$2$.  We can repeat this $\nu n$ times for each~$x$,
		avoiding all previously selected cycles using the (extended) set~$Q$.
	\end{proof}

	Given families~$\cC_x$ as in \cref{cla:cycle-cover-X}, let us define
	$\cF_1,\dots,\cF_{k-1}\subset V(G')^\ell$ with~$|\cF_i|=\nu n$ for
	every~$i\in[k-1]$ and such that, for every $i \in[k-1]$,
	$\vecb v \in \cF_i$ and
	$\vecb w \in \cF := \bigcup_{i \in [k-1]} \cF_i$ with
	$\vecb w \neq \vecb v$, we have that
	\begin{itemize}
		\item $\vecb v$ and $\vecb w$ are vertex-disjoint;
		\item $\vecb v$ together with $x_i$ appended forms a loose cycle $C_{\vecb v}
			      \subset G'$ along the ordering of $\vecb v$;
		\item the $(k-1)$st vertex and $(2k-2)$nd vertex of $\vecb v$ both
		      have degree~$2$ in~$C_{\vecb v}$.
	\end{itemize}
	We denote by $V(\cF)$ the set of vertices in the tuples in~$\cF$.

	\begin{claim}
		\label{claim:Us}
		There are disjoint $\beta n$-sets
		$U_1,\dots,U_\ell \subset V(G) \sm (Q \cup V(\cF)\cup X)$ such that
		all but $o(n^d)$ $d$-sets of vertices in $G'$ have degree at least
		$(\mu_d(k) + \gamma/2) \binom{\beta n-d}{k-d}$ into each $U_j$.
	\end{claim}
	\begin{proof}
		Choose a set $U\subset V(G)\sm(Q\cup V(\cF)\cup X)$ of
		cardinality~$\ell\beta n$ uniformly at random and partition~$U$ randomly
		into~$\ell$ parts of cardinality~$\beta n$ each to obtain the~$U_j$
		($1\leq j\leq\ell$).  We claim that this procedure will succeed in
		producing the required~$U_j$ with high probability.  Fix
		$1\leq j\leq\ell$.  Clearly, $U_j$~is a subset
		of~$V(G)\sm(Q\cup V(\cF)\cup X)$ of cardinality~$\beta n$ chosen
		uniformly at random.  By
		\cref{cla:absorber-sparse-properties-random-graph}\ref{itm:absorber-degree-subgraph},
		we know that~$U_j$ satisfies the required degree property with high
		probability.  Since~$\ell$ is a constant, our claim follows from the
		union bound.
	\end{proof}

	Let~$\cP$ be the family of the sets~$U_j$ ($1\leq j\leq\ell$) given by
	Claim~\ref{claim:Us}.  We now consider $G'(\cF,\cP)$ and note that
	$G'(\cF,\cP)$ is a subgraph of~$G(\cF,\cP)$.  For later reference, let
	us call~$W_i$ the set of vertices of~$G'(\cF,\cP)$ obtained by
	contracting the members of~$\cF_i$.  Note that $|W_i|=\nu n$ for every
	$1\leq i\leq\ell$.

	\begin{claim}
		\label{claim:7.7}
		We can pick an $\ell$-tuple
		$\vecb v_i = ( v_{i}^1,\dots, v_{i}^\ell)$ in~$\cF_i$ simultaneously
		for all $i\in[k-1]$ in such a way that the following holds: for each
		$j \in [\ell]$, there is an absorber~$A_j$ with
		$V(A_j)\subset U_j$ rooted in $X^j = \{v_{1}^j,\dots,v_{k-1}^j\}$ of
		order at most $M/(2\ell)$.
	\end{claim}
	\begin{proof}
		Note first that~$G'(\cF,\cP)$ is endowed with the partition
		\begin{equation}
			\label{eq:h_partition}
			W_1,\dots,W_{k-1},U_1,\dots,U_\ell
		\end{equation}
		of its vertex set.  We now invoke
		Lemma~\ref{lem:transfer-cluster-graph-partition} to obtain an
		$\epsilon$-regular partition $V_1,\dots,V_r$ with
		$r_0\leq r\leq r_1$, with most of the~$V_i$ contained in some set in
		the partition~\eqref{eq:h_partition}.  Moreover, we define the
		corresponding reduced graph~$\cR$ using as the `threshold parameter'
		the constant~$\tau$ defined earlier.  Let $\cU_j$
		$(1\leq j\leq\ell)$ and~$\cW_i$ ($1\leq i \leq k-1$) be the sets of
		clusters contained in each $U_j$ and~$W_i$ respectively.

		We now claim that almost all the clusters in~$\cW_i$ have degree at
		least $(\mu_d(k) + \gamma/4) \binom{|\cU_j|}{k-1}$ into each~$\cU_j$
		in~$\cR$.  To check this claim, recall first that the~$U_j$ are such
		that all but $o(n^d)$ $d$-sets of vertices of~$G'$ have degree at
		least $(\mu_d(k)+\gamma/2)\binom{\beta n-d}{k-d}$ into~$U_j$.
		It follows that all but~$o(n)$ vertices of~$G'$ have degree
		$(\mu_d(k)+\gamma/3)\binom{\beta n-1}{k-a}$ into~$U_j$.  Our claim
		now follows from
		Lemma~\ref{lem:transfer-cluster-graph-partition}\ref{it:3.6.2}
		applied with~$d'=1$.

		Note that, because of our claim above, we can fix clusters
		$Y_i\in\cW_i$ ($1\leq i \leq k-1$) that satisfy the degree property above
		for all $\cU_j$ ($1 \leq j \leq \ell$) simultaneously.  Let
		$\cY = \{Y_1,\dots,Y_{k-1}\}$.

		Note that $\cR[\cU_j \cup \cY]$ satisfies the assumptions of
		\cref{lem:absorber-dense} for each $j \in [\ell]$.  Hence we may
		find a $K$-sparse absorber~$\cA_j$ rooted in~$\cY$ of order at most
		$M/(2\ell)$ with $V(\cA_j) \subset \cU_j$.  Let
		$\cA = \bigcup \cA_j$.  Note that~$\cA$ is a `book' of the absorbers
		of \cref{lem:absorber-dense} with $\ell$~`pages' and `spine'~$\cY$.
		It follows by \cref{lem:sparse-to-mkdensity} that
		$m_k(\cA) \leq 2/(k-1)+{\gamma/2}$.  Note also that~$\cA$ has at most
		$M/2+k-1$ vertices.  By
		\cref{cla:absorber-sparse-properties-random-graph}\ref{itm:absorber-embedding-lemma}
		we can find a canonical copy~$A$ of~$\cA$ in~$G'(\cF,\cP)$.  For
		$i \in [k-1]$, let $w_i\in Y_i$ be the vertex of~$A$ that
		corresponds to the vertex~$Y_i$ of~$\cA$.  By definition,
		$w_i=w_{\vecb v_i}$ for some
		$\vecb v_i = ( v_{i}^1,\dots, v_{i}^\ell) \in \cF_i$.

		We now proceed to find the required absorbers~$A_j$.  Fix
		$j\in[\ell]$ and consider~$\cA_j$.  For every cluster
		$Z\in V(\cA_j)\subset U_j$, there is a vertex~$z\in Z$ contained
		in~$A$.  Let~$T_j$ be the set of vertices $v_{i}^j$ ($1\leq i \leq k-1$)
		and such vertices~$z$.  Then, by the definition of the contracted
		graph~$G'(\cF,\cP)$, one sees that the graph~$G'[T_j]$ contains a
		copy of an absorber~$A_j$ rooted in
		$X^j = \{v_{1}^j,\dots,v_{k-1}^j\}$ with~$V(A_j)\subset U_j$.
	\end{proof}

	We now construct the required absorber~$A$ rooted in~$X$ of order at
	most~$M$ using the~$A_j$ and~$\vecb v_i=(v_{i}^1,\dots,v_{i}^\ell)$
	from Claim~\ref{claim:7.7}.  Most of~$A$ is
	$\bigcup_{j \in [\ell]} A_j$; we need only incorporate two paths for
	each $\vecb v_i$ ($1\leq i \leq k-1 $).  The final absorber~$A$ will be such
	that
	\begin{align*}
		V(A)=\bigcup_{j\in[\ell]}V(A_j)\cup X^j
		= \bigcup_{j\in[\ell]}V(A_j) \cup \bigcup_{i\in[k-1]}\{v_{i}^1,\dots,v_{i}^\ell\}.
	\end{align*}
	Next, we describe describe how $A$ works.

	First note that for each $j \in [\ell]$ we may use $A_j$ to either
	integrate or not integrate $X^j=\{v_{1}^j,\dots,v_{k-1}^j\}$.  Now, if
	$X$ is to be covered, then we {activate} $A_j$ to absorb $X^j$ for
	$k \leq j \leq 2k-3$.  We do not activate any other~$A_j$.  To include
	the remaining vertices, we then take for each $i \in [k-1]$ the loose
	path, called $P_1^i$ in the sketch of the proof, that runs from the $(k-1)$st vertex of~$\vecb v_i$ backwards
	to~$x_i$ and then to the $(2k-2)$nd vertex of~$\vecb v_i$, namely
	\begin{align*}
		v_{i}^{k-1},\, v_{i}^{k-2},\, \dots,\, v_{i}^1,\, x_i,\, v_{i}^\ell,\, v_{i}^{\ell-1},\, \dots, v_{i}^{2k-2}.
	\end{align*}
	Conversely, if~$X$ should be left uncovered, we activate~$A_j$ to
	absorb~$X^j$ for $1 \leq j \leq k-2$ and $2k-1 \leq j \leq \ell$.  We
	do not activate any other~$A_j$.\footnote{The absorbers $A_{k-1}$ and
		$A_{2k-2}$ are not used at all and exist only for notational
		convenience.}  To include the remaining vertices, we take for each
	$i \in [k-1]$ the loose path with one edge that runs from the
	$(k-1)$st vertex of~$\vecb v_i$ to the $(2k-2)$nd vertex
	of~$\vecb v_i$, namely $v_{i}^{k-1},\dots, v_{i}^{2k-2}$.
	For an illustration, see \cref{fig:complete-absorber}.

	The discussion above shows that~$A$ is indeed an absorber rooted
	in~$X$.  Moreover, since the order of the~$A_j$ is at
	most~$M/(2\ell)$, we see that $A$~has order at most
	$M/2+(k-1)\ell\leq M$, as required.
\end{proof}

\subsection{Proof of Lemma~\ref{lem:sparse-to-mkdensity}}
\label{sec:proof-lemma-sparse-to-mk}
We shall see that Lemma~\ref{lem:sparse-to-mkdensity} holds because
(\textit{a})~a $K$-sparse absorber has maximum $k$-density not much
bigger than~$2/(k-1)$ and (\textit{b})~when we `amalgamate' such
absorbers along their root set forming a `book', the maximum
$k$-density increases by a controlled amount.
For convenience, let us set
\[
	d_{k}\left(H\right)=
	\frac{e\left(H\right)-1}{v\left(H\right)-k}
\]
for a graph $H$ with $v(H) > k$.
Hence, we can write
\[
	m_{k}\left(H\right)=\max\left\{
	d_k(H')\colon H'\subseteq
	H\text{ with }v\left(H'\right)>k\right\}.
\]

We need the following observation, whose easy proof we omit.
\begin{fact}\label{prop:acyclic}
	The following hold.
	\begin{enumerate}[\upshape (a)]
		\item \label{itm:acyclic-a} Every Berge acyclic $k$-graph $F$ satisfies $v(F)\geq(k-1)e(F)+1$.
		\item \label{itm:acyclic-b} Every minimal $k$-uniform Berge cycle $C$ of length at least $3$ is a loose cycle. In particular, we have $v(C) = (k-1) \, e(C)$.
	\end{enumerate}

\end{fact}

Fact~\ref{prop:acyclic} has the following consequence.

\begin{proposition}
	\label{claim:m_k_abs}
	The following hold.
	\begin{enumerate}[\upshape(i)]
		\item\label{claim:acy_mk} If~$F$ is a Berge acyclic $k$-graph, then
		$m_k(F)\leq1/(k-1)$.
		\item\label{claim:abs_mk} If~$A$ is a $K$-sparse $k$-uniform
		absorber with~$K\geq3$, then ${m_k(A)}\leq2/(k-1)+1/(K(k-1)-k)$.
	\end{enumerate}

\end{proposition}
\begin{proof}
	We begin by showing~\ref{claim:acy_mk}.
	To this end, observe that if~$F'$ is Berge acyclic $k$-graph, then $v(F')\geq(k-1)e(F')+1$ by {\cref{prop:acyclic}~\ref{itm:acyclic-a}},
	which implies that
	\[
		d_k(F') \leq \frac{e(F')-1}{(k-1)e(F')+1-k}=\frac1{k-1}.
	\]
	Thus $m_k(F)\leq1/(k-1)$ for any Berge acyclic~$F$,
	and~\ref{claim:acy_mk} is proved.

	Now suppose~$H'$ is a subgraph of an absorber~$A$ as
	in~\ref{claim:abs_mk} with $v(H')>k$. We have to show
	that~$d_k(H')=(e(H')-1)/(v(H')-k)\leq2/(k-1)+1/(K(k-1)-k)$. If~$H'$
	is acyclic, then this follows from~\ref{claim:acy_mk}.
	On the other hand, suppose that $C$ is a cycle of minimum length in $H'$.
	By the $K$-sparseness hypothesis, we
	have~$e(C)\geq K \geq 3$.
	Hence, by applying~\cref{prop:acyclic}~\ref{itm:acyclic-b}, it follows that $v(H')\geq v(C)\geq K(k-1)$.  Note that
	an absorber is a union of two Berge acyclic graphs, say~$F_1$
	and~$F_2$ (each~$F_i$ can be taken to be a vertex disjoint union of
	loose paths).  We may suppose that $V(F_1)=V(F_2)=V(H')$.  We then
	have
	\begin{align*}
		d_k(H') & \, {\leq} \, \frac{e(F_1)+e(F_2)-1}{v(H')-k} \\
		        & =\frac{e(F_1)-1}{v(H')-k}
		+\frac{e(F_2)-1}{v(H')-k}
		+\frac1{v(H')-k}                                       \\
		        & \leq m_k(F_1)+m_k(F_2)
		+\frac{1}{K(k-1)-k}                                    \\
		        & \leq\frac2{k-1}+\frac{1}{K(k-1)-k}\,,
	\end{align*}
	as required.
\end{proof}

We remark that the previous proof can be extended to a general bound for the $k$-density of the `amalgamation' of two graphs in terms of their $k$-densities and the girth of the whole graph.

\begin{proof}[Proof of \cref{lem:sparse-to-mkdensity}]
	Let $A_1,\dots,A_T$ be copies of a $K$-sparse
	$k$-uniform absorber sharing their root $(k-1)$-set~$X$ with
	$V(A_j)\cap V(A_{j'})=\emptyset$ for all~$j\neq j'$. Let~$H$ be the $k$-graph on
	\begin{equation*}
		X\cup\bigcup_{1\leq j\leq T}V(A_j)
	\end{equation*}
	whose edges are the edges that occur in the paths that define the
	absorbers~$A_j$.  We have to show that $m_k(H)\leq2/(k-1)+\gamma$,
	that is, we have to show that, for any subgraph~$H'$ of~$H$
	with $v(H')>k$, we have
	\begin{equation}
		\label{eq:3}
		d_k(H') = \frac{e(H') - 1}{v(H') - k} \leq \frac2{k-1} + \gamma.
	\end{equation}
	By Claim~\Ref{claim:m_k_abs}\ref{claim:acy_mk},
	inequality~\eqref{eq:3} holds if~$H'$ is acyclic.
	So we can assume that~$H'$ contains a Berge cycle.
	Let $C$ be a Berge cycle in $H'$ of minimum length.
	Recall that by definition of sparse absorbers (at the beginning of \cref{sec:absorber-lemma-sparse-proof}), each $K$-sparse absorber $A_j$ has girth at least $K$ even after adding the extra edge $X$.
	Hence there exists an absorber $A_j$ such that $C$ has at least $K-1$ edges in $A_j$.
	After relabelling, we may assume that $j=1$.

	In what follows, we write~$H'[Y]$ for~$H'[V(H')\cap Y]$ for
	any~$Y\subset V(H)$ for simplicity.
	Without loss of generality, assume that $A_1,\dots,A_t$ are the
	absorbers~$A_j$ for which $H'[X\cup V(A_j)]$ contains an edge (if
	$H'[X\cup V(A_j)]$ contains no edge, then we can consider
	$H'-V(A_j)$ instead of~$H'$).
	Let us further assume that
	$A_1,\dots,A_s$ are the absorbers~$A_j$ such that $H'[X\cup V(A_j)]$
	contains at least two edges.
	Note that $s\geq 1$ since $A_1$ contains at least $K-1$ edges.
	Let~$H''=H'\big[X\cup\bigcup_{1\leq j\leq s}V(A_j)]$.

	\begin{claim}
		\label{claim:H''}
		We have $d_k(H'')\leq2/(k-1)+{\gamma}/{2}$.
	\end{claim}
	\begin{proof}
		Fix~$j$ with~$1\leq j\leq s$.  Let $v_j=v(H'[X\cup V(A_j)])$ and
		$e_j=e(H'[X\cup V(A_j)])\geq2$.
		Note that, by definition of $s$, we have~$v_j>k$.
		Hence $d_k(H'[X\cup
			V(A_j)])=(e_j-1)/(v_j-k)\leq m_k(A_j)\leq2/(k-1)+1/(K(k-1)-k)$, where the last inequality follows from \cref{claim:m_k_abs}~\ref{claim:abs_mk}.

		Let~$e''=e(H'')$ and~$v''=v(H'')$.  We have
		$e'' = \sum_{1\leq j\leq s}e_j$ and
		$v''=\sum_{1\leq j\leq s}v_j-(s-1)k'$, where
		$k'=|X\cap V(H'')|<k$.  By
		Claim~\ref{claim:m_k_abs}~\ref{claim:acy_mk}, we may assume
		that~$H''$ is not acyclic.
		Furthermore, since $A_1$ contains at least $K-1$ edges of $C$, it follows by \cref{prop:acyclic}~\ref{itm:acyclic-b} that $v_1\geq (K-1)(k-1)$.

		Putting this together, we have
		\begin{align*}
			d_k(H'') & =\frac{e''-1}{v''-k}
			=\frac{\sum_{1\leq j\leq s}e_j-1}
			{\sum_{1\leq j\leq s}v_j-(s-1)k'-k}                \\
			         & =\frac{\sum_{1\leq j\leq s}(e_j-1)+s-1}
			{\sum_{1\leq j\leq s}(v_j-k)+ks-(s-1)k'-k}         \\
			         & =\frac{\sum_{1\leq j\leq s}(e_j-1)+s-1}
			{\sum_{1\leq j\leq s}(v_j-k)+(k-k')(s-1)}          \\
			         & \leq\frac{\sum_{1\leq j\leq s}(e_j-1)}
			{\sum_{1\leq j\leq s}(v_j-k)}
			+\frac{s-1}{\sum_{1\leq j\leq s}(v_j-k)}           \\
			         & \leq\frac2{k-1}+\frac1{K(k-1)-k}
			+\frac{s-1}{(K-1)(k-1)-k}                          \\
			         & \leq   \frac2{k-1}
			+\frac{T}{(K-2)(k-1)-1}                            \\
			         & \leq\frac2{k-1}+\frac{\gamma}{2},
		\end{align*}
		as long as $s\leq T$ and $K$ is large enough with respect to $T,1/\gamma$ and $k$.
		Hence, Claim~\ref{claim:H''} follows.
	\end{proof}

	We now deduce~\eqref{eq:3} from Claim~\ref{claim:H''}. Note that $H'$ contains at most $t-s \leq T$ edges that are not in~$H''$.
	Since $K$ was chosen sufficiently large with respect to $T,1/\gamma$ and $k$, it
	\begin{align*}
		d_k(H') & \leq \frac{e''+t-s-1}{v''-k} \leq \frac{e''-1}{v''-k} +  \frac{t-s}{v''-k} \leq\frac2{k-1} + \gamma. \qedhere
	\end{align*}
\end{proof}

Similar to the previous remark, we note that this proof can be extended to a general bound for the $k$-density of the `amalgamation' of two (same rooted) graphs in terms of their $k$-densities and the girth of the whole graph.

\section{Absorbers in the dense setting}\label{sec:absorber-lemma-dense-proof}
To finish the proof of \cref{thm:main}, we have to show
\cref{lem:absorber-dense}.  We begin with the following strengthening
of the minimum degree threshold for loose Hamilton cycles, which
ensures the existence of a loose Hamilton cycle under slightly more
general conditions and with some additional properties. Let us
  say that two vertices in a loose cycle $C$ are at distance $K$ if
  they are $K$ vertices apart with respect to the ordering of $C$.  So
  for instance, if $C$ is $k$-uniform then two consecutive vertices of
  degree $2$ are at distance $k-1$.  We say that a set $X \subset
  V(C)$ is \emph{$K$-spread} if the distance between any (distinct)
  pair of vertices of $X$ is at least $K$.

For $1 \leq d \leq k-1$, we define $\mu_d^*(k)$ as the least~$\mu\in[0,1]$ such that for all $\gamma>0$,
positive integers $t,\,K$ and $n$, where $n$ is divisible by $k-1$ and sufficiently large, the following holds:
if $G$ is an $n$-vertex $k$-graph and $X
	\subset V(G)$ is a $t$-set such that
$\delta_d(G-X)\geq (\mu+\gamma)\binom{n-d}{k-d}$ and
$\deg(x) \geq (\mu+\gamma)\binom{n-1}{k-1}$ for every
$x \in X$, then $G$ contains a loose Hamilton
cycle~$C$ in which $X$ is $K$-spread {and every vertex of $X$ has degree $2$.}
Note that we trivially have $\mu_d^*(k)\geq \mu_d(k)$.
The following lemma, whose proof is deferred to \cref{sec:exploiting-threshold-proof}, shows that the two thresholds actually coincide.

\begin{lemma}[Threshold exploitation]\label{lem:threshold-exploitation}
	We have $\mu_d(k) = \mu_d^*(k)$ for all $1 \leq d \leq k-1$.
\end{lemma}

For the proof of \cref{lem:absorber-dense}, we require two further lemmas.
Let $G$ and $R$ be $k$-graphs and $X \subset V(G) \cap V(R)$.
We say that $G$ contains a \emph{copy of~$R$ rooted in~$X$}, if there is an
embedding of $R$ into $G$ which is the identity on $X$.
For any integer $m>0$, we also
denote by $R^\ast(m,X)$ the $k$-graph obtained from $R$ by blowing up
each vertex outside of $X$ by $m$ and replacing edges with complete $k$-partite $k$-graphs.

\begin{lemma}[Blow-up setup]\label{lem:blow-up-new}
  Let $1/n \ll 1/s \ll \gamma,\, 1/k,\, 1/d,\, 1/t$, $\mu \geq 0$
    and $1/n \ll 1/m$.  Let $G$ be an $n$-vertex $k$-graph and $X
  \subset V(G)$ be a $t$-set such that $\delta_d(G-X)\geq
  (\mu+\gamma)\binom{n-d}{k-d}$ and $\deg_G(x) \geq
  (\mu+\gamma)\binom{n-1}{k-1}$ for every $x \in X$.  Then there is an
  $s$-vertex $k$-graph $R$ with $X \subset V(R)$ such that
	\begin{enumerate}[\upshape(1)]
		\item \label{itm:R-good-degree} $\delta_d(R-X) \geq (\mu + \gamma/2) \binom{s-d}{k-d}$ and $\deg_R(x) \geq (\mu+\gamma/2)\binom{s-1}{k-1}$ for every $x \in X$;
		\item \label{itm:R-many-bu-copies} $G$ contains a copy of $R^\ast(m,X)$ rooted in $X$.
	\end{enumerate}
\end{lemma}

\begin{lemma}[Absorber allocation]\label{lem:absorber-allocation}
	For $r\geq 1$ and $K>k\geq3$, let $K'=K(k-1)$, $m = 2^{2K'(2K'-1)}$ and $q=2^{2K'-1}rm!$.
	Let $X$ be a $(k-1)$-set.
	Let $C_2$ be a $k$-uniform loose cycle of order $q$ with $X\cap V(C_2) = \emptyset$.
	Let $C_1$ be a $k$-uniform loose cycle with $V(C_1) = V(C_2) \cup X$.
	Suppose that there is a vertex that has degree $2$ in $C_1$ and also in $C_2$.
	Let $R = C_1 \cup C_2$. Then
	\begin{enumerate}[\upshape (1)]
		\item\label{itm:aa-1} $R^\ast(2m,X)$ contains an absorber $A$ rooted in $X$.
		\item \label{itm:aa-2} Moreover, if $X$ is $K'$-spread in $C_1$, then $A$ is $K$-sparse.
	\end{enumerate}
\end{lemma}

We remark that the purpose of the constant $r$ in the above statement is to allow us to take $q$ arbitrarily large.
	Moreover, $q$ is divisible by $k-1$ as $m\geq k-1$.
The proofs of \cref{lem:blow-up-new,lem:absorber-allocation} are deferred to \cref{sec:blow-up-new,sec:absorber-allocation}.
Now we are ready to show the main result of this section.

\begin{proof}[Proof of \cref{lem:absorber-dense}]
	We can assume that $K>k$.
	Let $K'=K(k-1)$ and $m = 2^{2K'(2K'-1)}$.
	Introduce $r$ with $1/M \ll 1/r \ll \gamma$.
	Set $q=2^{2K'-1}rm!$, $s= q+k-1$ and $\mu=\mu_d(k)$.
	Given $G$ and $X$ as in the statement, we have to show that $G$ contains a $K$-sparse absorber $A$ rooted in $X$ of order at most $M$.

	We begin by selecting $S
		\subset V(G)$ uniformly at random among all
	$M$-sets that contain $X$. Let $G' = G[S]$. Then $v(G')={M}$.
	Note that with probability at least $2/3$ we have $\deg_{G'}(x)
		\geq(\mu + \gamma/2 ) \binom{M-1}{k-1}$ for every $x \in
		X$, which follows by a standard concentration inequality (see for instance~\cite[Corollary 2.2]{GIK+17}).
	Moreover, with probability at least $2/3$, we have $\delta_d(G'-X) \geq (\mu +
		\gamma/2 ) \binom{M-d}{k-d}$, which follows follow by~\cref{lem:random-subset-degrees}.
	Fix such a $k$-graph $G'$.  In the remainder, we show that $G'$ contains a $K$-sparse
	absorber $A$ rooted in $X$.

	Apply~\cref{lem:blow-up-new} with $k-1,\,2m,\,\gamma/2$ playing the role of $t,\,m,\,\gamma$ to
	$G'$ in order to obtain an $s$-vertex $k$-graph $T$ with $X
		\subset V(T)$ such that
	\begin{enumerate}[(i)]
		\item \label{itm:dal-mindegree} $\delta_d(T-X) \geq (\mu +
			      \gamma/4) \binom{s-d}{k-d}$ and $\deg_T(x) \geq
			      (\mu+\gamma/4)\binom{s-1}{k-1}$ for every $x \in X$;
		\item $G'$ contains a copy of $T^\ast(2m,X)$ rooted in $X$.
	\end{enumerate}

	By definition of $\mu$, it follows that $T-X$ contains a loose
        Hamilton cycle $C_2$.  Fix a vertex $y \in V(C_2)$ of degree
        $2$.  By \cref{lem:threshold-exploitation} {(applied with $X
          \cup \{y\}$ playing the role of $X$)} and
        property~\ref{itm:dal-mindegree} above it follows that $T$
        contains a loose Hamilton cycle $C_1$ such that $X$ is
        $K'$-spread in $C_1$, and $y$ has degree $2$ in
        $C_1$.

      Let $R = C_1\cup C_2$. Since $R^*(2m,X)$ is a subgraph of
      $T^*(2m,X)$, it follows that $G'$ also contains a copy of
      $R^\ast(2m,X)$ rooted in $X$, which
      by~\cref{lem:absorber-allocation}, contains an absorber $A$
      rooted in $X$. Finally, since $X$ is $K$-spread in $C_1$,
      \cref{lem:absorber-allocation} also tells us that the absorber
      $A$ is also $K$-sparse.
\end{proof}

It remains to prove \cref{lem:blow-up-new,lem:absorber-allocation}, which is done in the following two sections.

\subsection{Proof of \rf{lem:blow-up-new}}\label{sec:blow-up-new}

From \cref{lem:random-subset-degrees}, it is not difficult to get the following `weaker' version of
\cref{lem:blow-up-new}.

\begin{lemma}\label{lem:supersaturation}
	Let $1/n \ll 1/s \ll \gamma,\, 1/k,\, 1/d,\, 1/t$ and $\mu \geq 0$.
	Let $G$ be an $n$-vertex $k$-graph and $X \subset V(G)$ be a $t$-set such that $\delta_d(G-X)\geq (\mu+\gamma)\binom{n-d}{k-d}$ and $\deg_G(x) \geq (\mu+\gamma)\binom{n-1}{k-1}$ for every $x \in X$.
	Then there is an $s$-vertex $k$-graph $R$ with $X \subset V(R)$ such that
	\begin{enumerate}[\upshape(1)]
		\item\label{itm:supersat-R} $\delta_d(R-X) \geq (\mu + \gamma/2) \binom{s-d}{k-d}$ and $\deg_R(x) \geq (\mu+\gamma/2)\binom{s-1}{k-1}$ for every $x \in X$;
		\item\label{itm:supersat-many-copies} $G$ contains at least $ 2^{-s^k} n^{s-t} $ copies of $R$ rooted in $X$.
	\end{enumerate}
\end{lemma}
\begin{proof}
	Let $\cR$ be the set of all $s$-vertex $k$-graphs
	$R$ satisfying property~\ref{itm:supersat-R}.
	Now let $S \subset V(G)$ be an $s$-set chosen uniformly at random among all sets that contain $X$.\footnote{This means that $S\setminus X$ is an $(s-t)$-set chosen uniformly at random in $V(G)\setminus X$.}
	By applying~\cref{lem:random-subset-degrees}
	with $G-X$, $0$ and $\gamma$ playing the roles of $G$, $\delta$ and $\eta$, it follows, with probability at least $3/4$, that
	\begin{equation*}
		\delta_d(G[S-X]) \geq (\mu+\gamma/2)\tbinom{s-d}{k-d}.
	\end{equation*}
	Similarly, a standard concentration inequality (see for instance~\cite[Corollary 2.2]{GIK+17}) reveals that with probability at least $3/ 4$, we have that
		\begin{equation*}
			\deg_{G[S]}(x) \geq (\mu+\gamma/2)\tbinom{s-1}{k-1}\text{ for
				every }x\in X.
		\end{equation*}
	Hence $\PP(G[S]\in \cR) \geq 1/2$. This completes the proof of the lemma without part~\ref{itm:supersat-many-copies}.

	For part~\ref{itm:supersat-many-copies}, note that $\cR$ contains up to isomorphism at most $2^{\binom{s}{k}}$ elements.
	So by averaging, we conclude that there must be some $R\in \cR$, for which
	\[\PP(\text{$G[S]$ is a copy of $R$ rooted in $X$})\geq 2^{-\binom{s}{k}-1}.\]
	Since there are $\binom{n-t}{s-t}$ possibilities for choosing
        $S$, we obtain the desired estimate.
\end{proof}

In order to deduce \cref{lem:blow-up-new} from \cref{lem:supersaturation}, one
needs to prove that if
$G$ has positive density of copies of a graph $R$, then $G$ has a copy of a blow-up of $R$.
For that, we use the following result, which is a consequence of the
`supersaturation' phenomenon discovered by Erd{\H o}s and
Simonovits~\cite{ES83}, and the fact that the Turán density of partite
graphs is zero as proved by Erd{\H o}s~\cite{Erd64}.
\begin{theorem}\label{thm:blow-up}
	Let $1/n \ll \xi,\, 1/M, \, 1/q$.
	Then every $n$-vertex $q$-graph $G$  with $e(G) \geq \xi n^q$ contains a copy of the complete $q$-partite $q$-graph with parts of
	size $M$.
\end{theorem}

\begin{proof}[Proof of \cref{lem:blow-up-new}]
	Let $q=s-t$ and introduce $M,\,\xi$ with $1/n \ll 1/M \ll 1/m,\,1/s$ and $1/n \ll \xi,\, 1/M \ll 1/q$.
	Let $G$ be an $n$-vertex $k$-graph and $X \subset V(G)$ be a $t$-set such that $\delta_d(G-X)\geq (\mu+\gamma)\binom{n-d}{k-d}$ and $\deg_G(x) \geq (\mu+\gamma)\binom{n-1}{k-1}$ for every $x \in X$.
	We apply \cref{lem:supersaturation} to obtain an $s$-vertex $k$-graph $R$ with $X \subset V(R)$ that satisfies properties~\ref{itm:supersat-R} and~\ref{itm:supersat-many-copies}.

	Now let us define a $q$-uniform auxiliary $(n-t)$-graph $H$, with
	$V(H)= V(G)\sm X$ by
	adding an edge $Q$ of size $q$ to $E(H)$  whenever $G[Q\cup X]$ is a copy of $R$
	rooted in $X$.
	By \cref{lem:supersaturation}, the $q$-graph $H$ has at least $\xi n^q$ edges.
	It follows from \cref{thm:blow-up}, that $H$
	contains a copy of the complete $q$-partite $q$-graph $L$
	with parts of size $M$.

	Given a copy of $L$ in $H$, one can colour each edge of $L$ by one
		of~$s!$ colours, corresponding to
	which of the $s!$ possible orders the vertices of $R$ are mapped to the
	parts of $L$. {Using~\cref{thm:blow-up} once again}, we find a complete $q$-partite subgraph $L'\subset L$ with parts of size $m$, whose edges correspond to copies of $R$ that have their respective vertices in the same parts.
\end{proof}

\subsection{Proof of \rf{lem:absorber-allocation}}\label{sec:absorber-allocation}
Consider disjoint sets $\cB = \{B_1,\dots,B_\ell\}$ of size $m$.
An \emph{$(m,\ell)$-strip} $H$ is the union of matchings $M_1,\dots, M_\ell$, where $M_i$ matches $B_i$ into $B_{i+1}$ (index computation modulo $\ell$).
Note that one can obtain a permutation $\sigma_H$ of $B_1$ by following the path starting at $b \in B_1$ along the matching edges until it arrives again in a vertex $b' \in B_1$.
We say that $H$ is \emph{cyclical} if $\sigma_H$ is the identity.
In this case, $H$ is the union of $m$ vertex-disjoint cycles $C_1,\dots,C_m$ each of order $\ell$, whose orderings follow the indices of $\cB$.
Finally, given $(m,\ell)$-strips $H_1$ and $H_2$ on $\cB$, we say that $D=H_1 \cup H_2$ is an  \emph{$(m,\ell)$-double-strip} if both strips are cyclical and their edges are disjoint.

We remark that, given any $(m,\ell)$-strip $H$ on $\cB$ and $q = \ell
rm!$ with an integer $r \geq 1$, we can generate a cyclical
$(m,q)$-strip $H'$ by chaining together $rm!$ copies of $H$.  More
precisely, we obtain $H'$ by considering vertex-disjoint copies
$H_1,\dots,H_{rm!}$ of $H$ where $H_i$ has vertex set
$B_{i,1},\dots,B_{i,\ell}$.  For each $i \in [rm!]$, we then replace
each edge $uv$ of type $u\in B_{i,\ell}$ and $v\in B_{i,1}$ with an
edge $uv'$, where $v'$ is the copy of $v$ in $B_{i+1,1}$ (index
computation modulo $rm!$).  Note that $\sigma_{H'}$ is the $rm!$-times
product of $\sigma_H$ and hence the identity.  So $H'$ is indeed
cyclical. Moreover, the constant $r$ allows us to scale $q$ if
  necessary.

{The following proposition guarantees the existence of double-strips with large girth.}

\begin{proposition}\label{prop:strip}
  Let ${K'} \geq 1$, $m=2^{2{K'}(2{K'}-1)}$ and $q = 2^{2{K'}-1} rm!$
  for $r \in \NATS$.  Then there is an $(m,q)$-double-strip $D$ of
  girth at least $2{K'}$. 
\end{proposition}
\begin{proof}
	For $\ell=2^{2{K'}-1}$, consider a $2^{2{K'}}$-regular bipartite graph $G$ with colour classes $X,Y$ each of size $m$ with girth at least ${2{K'}-1} + 5\geq 2{K'}$.
	The existence of such a graph follows from the work of Lazebnik and Ustimenko~\cite{LU95} since $2{K'}-1$ is odd and $2^{2{K'}}$ is a power of a prime.
	Let $\cB= \{B_1,\dots,B_{\ell}\}$ where $B_i$ is copy of $X$ if $i$ is odd and a copy of $Y$ if $i$ is even.
	We obtain an $(m,\ell)$-strip $H$ by taking pairwise edge-disjoint {perfect} matchings $M_1,\dots,M_{\ell} \subset G$ {(using Hall's  theorem)} and placing $M_i$ between $B_i$ and $B_{i+1}$.
	Since $G$ is $2^{{2{K'}}}$-regular, we find another $(m,\ell)$-strip $H'$ that is edge-disjoint with $H$.
	By the above remark applied to $H$ and $H'$ separately, we may chain ourselves an $(m,q)$-double-strip $D$ whose partition we denote by $\cB'$.

	We claim that $D$ has the desired properties.  By
        construction, the union of the edges between any
        $\ell$ consecutive parts of $\cB'$ is isomorphic to a subgraph
      of $H \cup H'$.  {Moreover, every cycle in $H \cup H'$ of length
        $p$ can be associated with a cycle in $G$ of length at most
        $p$.}  Since $G$ has girth at least $\ell \geq {2{K'}}$,
      it follows that $D$ has girth at least $2{K'}$.
\end{proof}

{
\def\DS{D} 
\def\DSm{D^-} 
\def\DSX{D^+} 
\def\DSXsplit{D^+_{\spl}} 
\def\C#1#2{(#2)^{\scriptscriptstyle #1}} 
\def\DELETABLE#1{#1}

\begin{figure}

	\begin{tikzpicture}[x=1.5cm,y=1cm,scale=0.4]
		\draw [rotate around={90:(-14,-0.5)},line width=1pt] (-14,-0.5) ellipse (6.727904070582761cm and 1.7362871833205662cm);
		\draw [rotate around={90:(-9,-0.5)},line width=1pt] (-9,-0.5) ellipse (6.727904070582705cm and 1.736287183320552cm);
		\draw [rotate around={90:(-4,-0.5)},line width=1pt] (-4,-0.5) ellipse (6.7279040705826985cm and 1.7362871833205504cm);
		\draw [line width=3pt,color=myred] (-9,4)-- (-14,4);
		\draw [line width=3pt,color=myblue] (-9,4)-- (-4,2);
		\draw [line width=3pt,color=myblue] (-14,-4)-- (-9,-4);
		\draw [line width=3pt,color=myblue] (-4,-2)-- (-9,-4);
		\draw [line width=3pt,color=myblue] (-9,4)-- (-14,0);
		\draw [line width=3pt,color=myred] (-9,4)-- (-4,4);
		\draw [line width=3pt,color=myred] (-9,-4)-- (-14,-2);
		\draw [line width=3pt,color=myred] (-9,-4)-- (-4,0);
		\draw [rotate around={0:(-9,9)},line width=1pt] (-9,9) ellipse (2.0824024587000016cm and 0.58cm);
		\draw [line width=3pt,color=mygreen] (-9,4)-- (-11,5);
		\draw [line width=3pt,color=mygreen] (-9,4)-- (-7,4.5);
		\draw [line width=3pt,color=mygreen] (-10,9)-- (-11,8);
		\draw [line width=3pt,color=mygreen] (-10,9)-- (-10,8);
		\draw [line width=3pt,color=mygreen] (-8,9)-- (-7,8);
		\draw [line width=3pt,color=mygreen] (-8,9)-- (-9,8);
		\draw [line width=3pt,dotted,color=mygreen] (-9,3)-- (-11,4.5);
		\draw [line width=3pt,dotted,color=myred] (-9,-5)-- (-14,-2);
		\draw [line width=3pt,dotted,color=myblue] (-9,3)-- (-14,0);
		\draw [line width=3pt,dotted,color=myblue] (-9,-5)-- (-14,-4);
		\draw [line width=3pt,color=myblue] (-4,2)-- (-2,2);
		\draw [line width=3pt,color=myblue] (-4,-2)-- (-2,-2);
		\draw [line width=3pt,color=myblue] (-14,-4)-- (-16,-4);
		\draw [line width=3pt,color=myblue] (-14,0)-- (-16,0);
		\draw [line width=3pt,color=myred] (-4,4)-- (-2,4);
		\draw [line width=3pt,color=myred] (-14,4)-- (-16,4);
		\draw [line width=3pt,color=myred] (-14,-2)-- (-16,-2);
		\draw [line width=3pt,color=myred] (-4,0)-- (-2,0);
		\begin{scriptsize}

			\draw [fill=black] (-9,4) circle (4.0pt);
			\draw[color=black] (-7.8,5.0) node {\Large $u_1=v_1$};
			\draw [fill=black] (-9,3) circle (4.0pt);
			\draw[color=black] (-9,1.8) node {\Large $\bar{u}_1$};
			\draw [fill=black] (-9,-5) circle (4.0pt);
			\draw[color=black] (-9.05,-5.9) node {\Large $\bar{u}_m$};
			\draw [fill=black] (-9,-4) circle (4.0pt);
			\draw[color=black] (-9.1,-2.8) node {\Large $u_m$};
			\draw [fill=black] (-4,4) circle (4.0pt);
			\draw[color=black] (-4,4.7) node {\Large $v_1$};
			\draw [fill=black] (-4,2) circle (4.0pt);
			\draw [fill=black] (-14,-4) circle (4.0pt);
			\draw [fill=black] (-14,-2) circle (4.0pt);
			\draw [fill=black] (-14,4) circle (4.0pt);
			\draw[color=black] (-14,4.7) node {\Large $v_q$};
			\draw [fill=black] (-4,-2) circle (4.0pt);
			\draw [fill=black] (-14,0) circle (4.0pt);
			\draw [fill=black] (-4,0) circle (4.0pt);
			\draw [fill=black] (-10,9) circle (4.0pt);
			\draw [fill=black] (-8,9) circle (4.0pt);
			\draw[] (-13.78,7) node {\Large $A_{q}$};
			\draw[] (-3.78,7) node {\Large $A_{2}$};
			\draw[] (-8.78,7) node {\Large $A_1$};
			\draw[] (-5.84,9) node {\Large $X$};
			\draw[color=myred] (-1,4) node {\Large $\hat{S}_1^1$};
			\draw[color=myblue] (-1,2) node {\Large $S_1^2$};
			\draw[color=mygreen] (-11.6,5.3) node {\Large $S_1^1$};
			\draw[color=myred] (-1,0) node {\Large $S_m^1$};
			\draw[color=myblue] (-1,-2) node {\Large $S_m^2$};

		\end{scriptsize}
	\end{tikzpicture}
	\caption{The absorber constructed in the proof of \cref{lem:absorber-allocation}.}
	\label{fig:absorber}
\end{figure}
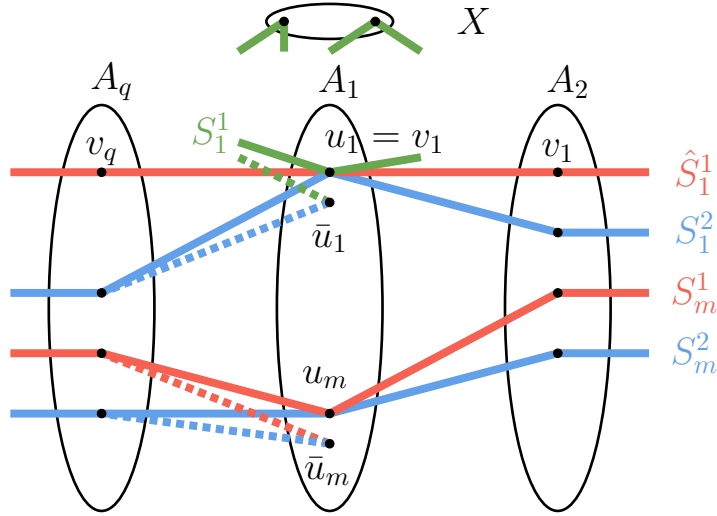

Hereafter, for $r$ divisible by $k-1$, we denote by $\C{k}{u_1,\dots,u_r}$ the $k$-uniform
loose cycle $C$ with $V(C)=\{u_i\}_{i=1}^r$, whose edges
follow the ordering $(u_1,\ldots,u_r,u_1)$.

\begin{proof}[Proof of \cref{lem:absorber-allocation}]
	Let us write $C_2=\C{k}{1,\dots,q}$ and $X=\{q+1,\dots,q+k-1\}$.  Furthermore, let
	$\sigma\colon [q+k-1] \to [q+k-1]$ be a permutation such that
	$C_1 = \C{k}{\sigma(1),\ldots,\sigma(q+k-1)}$ with $\sigma(1) = 1$.
	Without loss of generality, we can assume that the vertex $1$ has degree $2$ in both $C_1$ and $C_2$.

	We begin the construction of the absorber $A \subset R^\ast(2m,X)$ as follows.
	Denote the clusters of $R^\ast(2m,X)$ by $A_1,\dots,A_q$.
	Let $\cB=\{B_{1},\dots,B_{q}\}$,
	where $B_{i} \subset A_i$ is an arbitrary subset of $A_i$ of size $m$
	for every $i \in [q]$.
	Recall that $q =  2^{2K(k-1)-1} rm!$.
	So by \cref{prop:strip} applied with $K(k-1)$ playing the role of $K$, there is an $(m,q)$-double-strip $\DS$ with partition~$\cB$ and girth at least $2K(k-1)$.
	Denote the cycles of the two corresponding $(m,q)$-strips
	by $\widehat{S}^1_{1} ,S^1_2,\dots,S^1_{m}$
	and $S^2_{1},\dots,S^2_{m}$, respectively.

	Let us write
	$\widehat{S}^1_{1} = \C2{v_1,\dots,v_q}$
	with $v_i\in B_i$ for every $i\in [q]$.
	Moreover,
	set $v_x := x$ for every $x\in X$.
	Consider the cycle
	$S^1_{1} = \C2{v_{\sigma(1)},v_{\sigma(2)},\dots, v_{\sigma(q+k-1)}}$,
	which follows the order of $C_1$.
	Note that $S^1_{1}$ covers $X$ and exactly
	one vertex (namely $v_i$) of each cluster $B_i$ of $\cB$.
	Moreover $v(S^1_1)=v(\widehat{S}^1_1)+(k-1)$ and $v(S^1_1)$ is divisible by $k-1$.

	Let $\DSm$ be the
	$2$-graph obtained from $\DS$ by removing the
	edges of $\widehat{S}^1_1$.
	Define   $\DSX=\bigcup_{i=1}^m (S^1_i\cup S^2_i)$
	as the $2$-graph obtained from $\DSm$ by adding the vertices of $X$
	and the edges of $S^1_1$.
	Note that, in $\DSX$, every vertex of $V(D)$
	is contained in exactly two cycles, namely $S^1_j$ and  $S^2_i$ for some $i,j\in [m]$.
	Moreover, the $2m$ cycles $\{S^1_i,S^2_i\}_{i=1}^m$
	are pairwise edge-disjoint.

	Next, we
	turn the $2m$ cycles $\{S^1_i,S^2_i\}_{i=1}^m$ of $\DSX$ into
	$2m$ ($2$-uniform) paths $\{Q^1_i, Q^2_i\}_{i=1}^m$
	by splitting each vertex of $B_1$.
	More formally,
	we denote $B_1=\{u_i\}_{i=1}^m$ and $A_1\setminus B_1=\{\bar{u}_i\}_{i=1}^m$.
	Assume without loss of generality that, for every $i\in [m]$, the cycles
	$S^1_i$ and $S^2_i$ meet in $u_i\in B_1$.
	Hence, we can write
	$S^1_i = u_i L^1_i u_i$ and $S^2_i = u_iL^2_iu_i$ where $L^1_i,L^2_i$ are paths in $\DSX$.
	For every $i\in [m]$, we now define the desired $(u_i,\bar{u}_i)$-paths
	$Q^1_i := u_i L^1_i \bar{u}_i$ and
	$Q^2_i := u_i L^2_i \bar{u}_i$.
	Let $\DSXsplit=\bigcup_{i=1}^m \{Q^1_i, Q^2_i\}$ be
	the $2$-graph obtained by the union
	of those $2m$ edge-disjoint paths.
	By construction, $\DSXsplit$ satisfies the following properties:
	\begin{enumerate}[(a)]
		\item \label{itm:absorber-allocation-vx-set} $V(\DSXsplit) = X\cup A_1\cup B_2\cup\dots\cup B_q$;
		\item $\bigcup_{i=1}^m V(Q^1_i) = X\cup \bigcup_{i=1}^m V(Q^2_i)$ and $X \cap  \bigcup_{i=1}^m V(Q^2_i) = \es$;
		\item the paths $Q^j_1,\dots,Q^j_m$ are pairwise vertex-disjoint for $j \in [2]$;
		\item \label{itm:absorber-allocation-disjoint} $Q^1_i$ and $Q^2_i$ are both $(u_i,\bar{u}_i)$-paths for every $i\in [m]$.
	\end{enumerate}
	See also \cref{fig:absorber} for an illustration.
	Finally, we set $A$ to be the $k$-graph
	$\bigcup_{i=1}^m \{P^1_i, P^2_i\}$
	where $P^j_i$ is the $k$-uniform loose path with the same vertex set and vertex ordering
	as $Q^j_i$.
	Note that this is well defined since $v(Q^j_i)\equiv1\bmod{(k-1)}$.
	We proceed to show that $A$ satisfies the statement of the lemma.

	\medskip

	First, observe that $A$ satisfies the corresponding three
        properties of \cref{def:absorber} as a consequence of
        properties~\ref{itm:absorber-allocation-vx-set}--\ref{itm:absorber-allocation-disjoint}.
        Hence, $A$ is indeed an absorber rooted in $X$.  Secondly, it
        is easy to check that $A\subseteq R^*:=R^{*}(2m,{X})$.
        Indeed, define $f\colon V(R^*)\to [q+k-1]$ either by setting
        $f(w)=i$, if $w\in A_i$ for some $i\in [q]$; or by setting
        $f(w)=w$ if $w\in X$.  By construction, every edge
        $\{w_1,\dots,w_k\}\in E(P^1_1)$ arises in $R^*$ after blowing
        up the edge $\{f(w_1),\dots,f(w_k)\}\in E(C_1)$.  Similarly,
        every edge $\{w_1,\dots,w_k\}\in E(A)\sm E(P^1_1)$ arises in
        $R^*$ after blowing up the edge $\{f(w_1),\dots,f(w_k)\}\in
        E(C_2)$.  Note that at this point it is crucial that the
        vertex $1=\sigma(1)$ has degree $2$ in both $C_1$ and
        $C_2$.\footnote{This is a simple but sensitive point in
            the proof and the reason why
            \cref{lem:threshold-exploitation} is required.}  This
        shows that $A$ is an absorber rooted in $X$ and $A
          \subset R^\ast$ as desired.

	It remains to prove that if $X$ is $K'$-spread in $C_1$, then
	the absorber $A$ is $K$-sparse.
	The following claim  is the corresponding statement with respect to~$\DSXsplit$.

	\begin{claim}\label{cla:biggirth}
		Suppose $X$ is $K'$-spread in the $($$2$-uniform$)$ cycle $S^1_1$.
			Then, the girth of $\DSXsplit$ is at least $K'$,
			even after adding to $\DSXsplit$ the edges $E_X$ of any path
		$P_X$ satisfying $V(P_X)=X$.
	\end{claim}
	\begin{proof}
		Let $M$ be a cycle of $\DSXsplit\cup E_X$.
		We need to show that $v(M)\geq K'$.
		It is clear that $\DSXsplit\cup E_X$
		has girth not smaller than $\DSX\cup E_X$ and, therefore,
		we can assume that $E(M)\subseteq E(\DSX)\cup E_X
			= E(\DSm)\cup E(S^1_1)\cup E_X$.

		First, assume $E(M)\cap E(S^1_1)=\emptyset$.
		Since the edges of $\DSm$ do not intersect $X$,
		we get that $E(M)\subseteq E(\DSm)\subseteq E(D)$,
		which has girth at least $2K'\geq K'$, as required.

		Next, assume $E(M)\cap E(\DSm)=\emptyset$.  In that
                case, we have $E(M)\subseteq E(S^1_1)\cup E_X$.  If
                $E(M)\cap E_X=\emptyset$, then ${M=S^1_1}$, which has
                order $v(M)=v(S^1_1)>q \geq K'$.  If $E(M)$
                  contains an edge in $E_X$, then $M$ contains an
                  $(x_1,x_2)$-path $P\subseteq S^1_1$ where $x_1,x_2
                  \in X$.  But since $X$ is $K'$-spread in $S^1_1$,
                  the path $P$ has at least $K'$ edges.  Thus,
                  $v(M)\geq v(P)\geq K'$, as required.

		We can now assume that $E(M)$ contains edges from \emph{both}
		$E(S^1_1)$ and~$E(\DSm)$.
		Let $P\subseteq M$ be a $(u,w)$-path of maximum order
		satisfying~$E(P)\subseteq E(\DSm)$.
		Let $e\in E(M)\setminus E(P)$ be the (other) edge of $M$
		incident with~$u$.
		First, note that $e\notin E(\DSm)$, as
		otherwise $P$ would not have maximum order.
		Also, we have~$e\notin E_X$,
		since $ u\in V(P) \cap e$ and no
		vertex of $P$ is in~$X$.
		Therefore, we must have $e\in E(S^1_1)$, which
		implies $u=v_i\in B_i$, for some $i\in [q]$.
		Analogously we have $w=v_j\in B_j$,
		for some $j\in [q]$.

		Without loss of generality we can assume that~$i<j$.
		Recall that adjacent vertices of $P\subseteq \DSm$ must
		belong to adjacent blocks of~$\cB$. In particular,
		$e(P)\geq\dist(i,j)$,
		where $\dist(i,j):=\min\{j-i,q+i-j\}$.

		Finally, let $P' \subset
                \widehat{S}^1_1=\C2{v_1,\dots,v_q}$ be a
                $(v_i,v_j)$-path satisfying~$e(P')=\dist(i,j)$. Note
                that $P\cup P'$ is a cycle of $\DS$ and, therefore,
                must have order at least $2K'$.  Using that
                $e(P)\geq\dist(i,j) = e(P')$, it follows that the
              order of $P$ must be at least~$K'$, which
              implies~$v(M)\geq v(P)\geq K'$.
	\end{proof}

	Let us show how \cref{cla:biggirth} helps us to finish the
        proof.  Suppose that $X$ is $K$-spread in the ($k$-uniform)
        cycle $C_1$, and let $A'$ be the hypergraph obtained from $A$
        after adding the extra edge $X=\{q+1,\dots, q+k-1\}$.  Since
        $X$ is $K'$-spread in $C_1$, it follows that $X$ is
        $K'$-spread in the ($2$-uniform) cycle ${S^1_1}$, as in
          the hypothesis of\cref{cla:biggirth}.

          Now let $C_{\text{Berge}}$ be a Berge cycle in $A'$ of
          minimum length.  Hence we need to show that
          $v(C_{\text{Berge}})\geq K' = K (k-1)$.  Our strategy is to
          identify a $2$-uniform cycle $M$ in $\DSXsplit \cup P_x$,
          whose vertices are contained in $C_{\text{Berge}}$ where
          $P_X$ is a path as in \cref{cla:biggirth}. If successful,
        this shows that $K'\leq v(M)\leq v(C_{\text{Berge}})$ by
        \cref{cla:biggirth}, and we are done.

	To find such a cycle $M$, we proceed as follows. If
          $X\in C_{\text{Berge}}$, consider a path $P_X$ as in
          \cref{cla:biggirth} with the same endpoints of $X$ along
          $C_{\text{Berge}}$.  Next, we construct a ($2$-uniform)
        graph $M'$ by replacing each edge $e\in
        E(C_{\text{Berge}})$ either by the corresponding $(k-1)$ edges
        of $\DSXsplit$ that originated $e$ (if~$e\in A$); or by the
        $(k-2)$ edges of $P_X$ (if $e=X$).  Since $\DSXsplit$ is
          the union of pairwise edge-disjoint paths and $P_X$ is
          edge-disjoint from $\DSXsplit$, it follows
        that $M'$ must contain a cycle $M$, as desired.
\end{proof}

\section{Exploiting the threshold}\label{sec:exploiting-threshold-proof}

In this section, we show \cref{lem:threshold-exploitation}.
The proof uses largely the same strategy as the one of \cref{thm:main}.
But since we are in the dense setting, many of the building blocks are much simpler.
The main conceptual difference between the two proofs is the way we construct the absorbers.
We require the following two lemmas, which are dense versions of \cref{lem:absorption-sparse,lem:cover-sparse}.

\begin{lemma}[Dense Absorption Lemma]\label{lem:absorption-dense}
	Let $1/n \ll \eta \ll {\alpha},\, 1/k,\, 1/d,\, \gamma$.
	Let $G$ be an $n$-vertex $k$-graph with $\delta_d(G) \geq (\mu_d(k) + \gamma )  \binom{n-d}{k-d}$.
	Then there is a set $A \subset V(G)$ and two vertices $u,v \in A$ such that
	\begin{itemize}
		\item $|A| \leq \alpha n$;
		\item for any subset $W \subset V(G-A)$ with $|W| \leq \eta n$ divisible by $k-1$, the induced graph $G[A \cup W]$ has a loose $(u,v)$-path covering exactly $A \cup W$.
	\end{itemize}
\end{lemma}

The proof of \cref{lem:absorption-dense} is deferred to the next subsection.

\begin{lemma}[Dense Cover Lemma]\label{lem:cover-dense}
	Let {$1/n \ll \alpha \ll 1/k,1/d,\gamma$ and let $1/n \ll \eta$}.
	Let $G$ be an $n$-vertex $k$-graph with $\delta_d(G) \geq (\mu_d(k) + \gamma ) \binom{n-d}{k-d}$,  $Q \subset V(G)$ with $|Q| \leq \alpha n$ and $u,v \in V(G-Q)$.
	Then there is a loose $(u,v)$-path $P$ in $G-Q$ that covers all but
	$\eta n$ vertices of $G - Q$.
\end{lemma}

\begin{proof}
	Let $G'=G-Q$, and note that $G'$ satisfies $\delta_d(G') \geq (\mu_d(k) + \gamma/2 ) \binom{n-d}{k-d}$ since $\alpha$ is very small in comparison to $\gamma$.
	For $\eta'=\eta/2$, we begin by selecting an $\eta'n$-set $C\subset V(G')$ such that $\delta_d(G'[C]) \geq (\mu_d(k) + \gamma/4) \binom{\eta'n}{k-d}$ and $\deg_{G'[C \cup \{w\}]}(w) \geq (\mu_d(k) + \gamma/4) \binom{\eta' n}{k-1}$ for every $w \in V(G')$.
	Such a set can be found using a standard concentration inequality (see for instance~\cite[Corollary 2.2]{GIK+17}).

	Let $G''=G'-C-u-v$, and note that $\delta_d(G'') \geq (\mu_d(k) + \gamma/4) \binom{n-d}{k-d}$.
	Hence $G''$ contains a loose cycle that covers all but
	$k-2$ of its vertices.
	From this we obtain a loose path $(u',v')$-path $P$ of order at least $(1-\eta/2)n$ with $u',v' \in V(G'')$.

	To finish, we apply \cref{lem:connection-dense} with $G'[C]$ playing the role of $G'$ to find a loose $(u,v)$-path that contains $P$ as a subpath.
	Indeed, by choice of $C$ any vertex $x \in V(G')$ is on an edge $e$ such that $|e \cap C| \geq k-1$.
	We can therefore connect $x$ to any vertex within $C$ using at most $5(k-1)$ further vertices.
	Applying this observation twice gives the desired connections.
\end{proof}

Now we are ready to show the main result of this subsection.

\begin{proof}[Proof of \cref{lem:threshold-exploitation}]
	Introduce $\gamma,\, K,\, t, \, \alpha,\, \eta,\, n$ with $1/n \ll \eta \ll \alpha \ll 1/t,\, 1/K,\, 1/k,\, 1/d, \,\gamma$. Set $\mu = \mu_d(k)$.
	Let $G$ be an $n$-vertex $k$-graph and $X=\{x_1,\dots,x_t\} \subset V(G)$ be a $t$-set such that $\delta_d(G-X)\geq (\mu+\gamma)\binom{n-d}{k-d}$ and $\deg_G(x) \geq (\mu+\gamma)\binom{n-1}{k-1}$ for every $x \in X$.
	We have to show that $G$ contains a loose Hamilton cycle~$C$ in which $X$ is $K$-spread.

	We begin by applying \cref{lem:absorption-dense} with $G-X$ playing the role of $G$ to obtain $A \subset V(G-X)$ and two vertices $u,v' \in A$ such that $|A| \leq \alpha n$ and for any subset $W \subset V(G-X-A)$ with $|W| \leq \eta n$ divisible by $k-1$, the induced graph $G[A \cup W]$ has a loose $(u,v')$-path covering exactly $A \cup W$.

	Let $v \in V(G-X-A)$.
	We cover $X$ by adding its vertices to a $(v',v)$-path $P$ of order at most $3Kt$ that shares with $A$ only the initial vertex $v'$.
	Moreover, the vertices of $X$ shall be at distance at least $K$ in $P$ and each have degree $2$.
	To this end, select pairwise disjoint $(v_i,v_i')$-paths $P_i \subset G-A$ of order $2$ such that $x_i$ has degree $2$ in $P_i$ for $1\leq i \leq t$.
	(This is possible due the minimum degree assumptions on $X$.)
	Set $v_0'=v'$ and $v_{t+1}=v$.
	For each $0 \leq i \leq t$, we (repeatedly) apply~\cref{lem:connection-dense} to obtain a loose $(v'_i,v_{i+1})$-path $P_i'$ of order at least $K$ and at most $K+4k$.
	Note that these paths can be chosen to be pairwise disjoint.
	(This is possible since $\delta_d(G-X)\geq (\mu+\gamma)\binom{n-d}{k-d}$.)
	It follows that the concatenation of the paths $P_i'$ and $P_i$ forms the desired path~$P$.

	Set $A' = (A \cup V(P)) \sm \{u,v\}$, and note that $|A'|\leq 2 \alpha n$.
	Next, we use  \cref{lem:cover-dense} with $A',\, 2\alpha$ playing the role of $Q,\, \alpha$ to find a loose $(v,u)$-path $P'$ in $G-A'$ that covers all but
	$\eta n$ vertices of $G - A'$.
	By choice of $A$, we may integrate the remaining vertices into a loose path which together with $P \cup P'$ forms a loose Hamilton cycle of $G$.
\end{proof}

We remark that the same argument also gives the following result, which is of independent interest.
More details on this follow in \cref{sec:conclusion}.
\begin{theorem}\label{thm:hamilton-connectedness}
	Let $1/n \ll 1/k,\,\gamma$ with $n-1$ divisible by $k-1$.
	Let $G$ be an $n$-vertex graph with $\delta_d(G) \geq (\mu_d(k) + \gamma) \binom{n-d}{k-d}$.
	Then $G$ contains a loose Hamilton path between any two distinct vertices.
\end{theorem}
The rest of this section is dedicated to the proofs of \cref{lem:connection-dense,lem:absorption-dense}.

\subsection{Connectivity}
\label{sec:connection-lemma-dense-proof}

Here we prove \cref{lem:connection-dense}.
We require the following fact which follows from the Kruskal--Katona theorem.
(It is easiest to see using Lovász's formulation.)

\begin{proposition}\label{prop:mini-KK}
	For $1/n \leq \eps,\,\delta,\, 1/k$, let $G$ be a $k$-graph with at least $(\delta+\eps)\binom{n}{k}$ edges.
	Then the edges of $G$ span at least
		{$\delta^{1/k} n$} vertices.
\end{proposition}

We also need the following simplified version of \cref{lem:blow-up-new}, whose proof we omit as it follows almost line by line the original argument.

\begin{lemma}[Simple blow-up setup]\label{lem:blow-up-new-simple}
	Let $1/n \ll 1/s \ll \gamma,\, 1/k,\, 1/d,\, 1/t$, $\mu \geq 0$ and $1/n \ll 1/m$.
	Let $G$ be an $n$-vertex $k$-graph with $\delta_d(G) \geq (\mu + \gamma) \binom{n-d}{k-d}$.
	Let $X \subset V(G)$ be a $t$-set.
	Then there is an $s$-vertex $k$-graph $R$ with $X \subset V(R)$ such that
	\begin{enumerate}[\upshape(1)]
		\item \label{itm:R-good-degree-simple} $\delta_d(R) \geq (\mu + \gamma/2) \binom{s-d}{k-d}$;
		\item \label{itm:R-many-bu-copies-simple} $G$ contains a copy of $R^\ast(m,X)$ rooted in $X$.
	\end{enumerate}
\end{lemma}

\begin{proof}[Proof of \cref{lem:connection-dense}]
	First assume that $d=1$.
	Introduce $m,s$ with $1/n \ll 1/m,\,1/s \ll \gamma,\, 1/k,\, 1/d,\, \mu$.
	By \cref{lem:blow-up-new-simple} applied with $\{u,v\}$ playing the role of $X$ there is an $s$-vertex $k$-graph $R$ with $X \subset V(R)$ such that $\delta_d(R) \geq (\mu_d(k) + \gamma/2) \binom{s-d}{k-d}$ and there is a copy $R^\ast$ of $R^\ast(m,\{u,v\})$ rooted in $\{u,v\}$ in $G$.
	Let $w \in V(R-u-v)$.

	We claim that there are distinct edges $e,f \in E(R)$ such that $u \in e \sm f$, $w \in f \sm e$ and $e \cap f \neq \es$.
	To see this, let $L(u)$ denote the link graph of $u$, which is the $(k-1)$-graph on $V(R)\sm \{u\}$ with a $(k-1)$-edge $e$ whenever $e \cup \{x\}$ is an edge in $R$.
	Recall from the introduction that $\mu_1(k) \geq 2^{-(k-1)}$.
	Thus applying \cref{prop:mini-KK} to $L(u)$ shows that the edges of {$u$} cover more than $n/2$ vertices.
	We define $L(v)$ analogously and obtain the same conclusion.
	It follows that $R$ contains the desired edges $e$ and $f$.
	By the same argument, $R$ contains edges $e',f'$ such that $w \in e' \sm f'$, $v \in f' \sm e'$ and $e' \cap f' \neq \es$.
	Hence we may easily construct the desired $(u,v)$-path of order $4(k-1)+1$ in $R^\ast$.

	Now assume that $d\geq 2$.
	In this case, we have $\delta_2(G) \geq (\mu_d(k)+ \gamma) \binom{n-2}{k-2}$.
	So we can greedily construct a loose $(u,v)$-path of order $4(k-1)+1$ by considering
	three distinct vertices $u',v',w'\in  {V(G)}\sm \{u,v\}$ and
	then selecting appropriate edges that contain $\{u,u'\},\{u',w'\},\{w',v'\}$ and $\{v',v\}$.
\end{proof}

\subsection{Absorption}
For the proof of \cref{lem:absorption-dense}, we require the following simplified version of \cref{lem:absorber-dense}.
\begin{lemma}[Simple Dense Absorber Lemma]\label{lem:absorber-dense-simple}
	Let $1/n \ll 1/q \ll \gamma, 1/k,1/d$.
	Let $G$ be an $n$-vertex $k$-graph with $\delta_d(G) \geq (\mu_d(k) + \gamma )  \binom{n-d}{k-d}$ and $X$ be a $(k-1)$-set of vertices of $G$.
	Then $G$ contains an absorber of order~$q+2k$ rooted in~$X$.
\end{lemma}

The proof of \cref{lem:absorber-dense-simple} is based on a similar idea as the one of \cref{lem:absorber-dense}.
However, since we cannot force a particular vertex in a Hamilton cycle to have degree $2$, we need to change our approach a little bit.

\begin{proof}[Proof of \cref{lem:absorber-dense-simple}]
	Set $t=k-1$, $s=q+t$ and $\mu = \mu_d(k)$.
	Apply \cref{lem:blow-up-new-simple} with $m=2$  to $G$ in order to obtain
	an $s$-vertex $k$-graph $T$ with $X \subset V(T)$ such that
	\begin{enumerate}[\upshape (1)]
		\item \label{itm:sdal-mindegree} $\delta_d(T) \geq (\mu + \gamma/2) \binom{s-d}{k-d}$;
		\item \label{itm:many-absorbers} $G$ contains a copy of $T^\ast(2,X)$ rooted in $X$.
	\end{enumerate}

	We claim that $T$ contains a loose Hamilton cycle $C_1$, and $T-X$ contains a loose Hamilton cycle $C_2$.
	The existence of $C_1$ follows by property~\ref{itm:sdal-mindegree} and the definition of $\mu=\mu_d(k)$.
	For $C_2$, observe that $\delta_d(T-X) \geq (\mu+\gamma/4)\binom{q-d}{k-d}$ due to the choice of constants.
	So we also find $C_2$ by the definition of $\mu=\mu_d(k)$.
	Let $R = C_1\cup C_2$. For every vertex $v$ of $C_2$, we denote by
	$\bar v$ the vertex that is in the same colour class of $v$
	in $R^\ast(2,X)$.

	Our goal is now to find an absorber $\{P_1^1, P_1^2, P_2^1,
		P_2^2\}$ of order $q+2k$ rooted in $X$.
	To begin, let $P_1^1=( v_1,\ldots, v_k)$ consist of the vertices corresponding to an (ordered) edge of $C_2$.
	We define
	$P_2^1$ as the path that starts with $(v_1, \ldots, v_{k-1}, \bar v_k)$, follows the orientation of
	$C_2$ via the vertices of type $\bar v$ until it reaches $\bar v_1$ and ends with $(\bar v_1,\ldots,\bar v_{k-1}, v_k)$.
	Next, let $P_2^2=(w_1,\ldots, w_k)$  consist of the vertices corresponding to an (ordered) edge of $C_1$, which is disjoint of the vertices of $P_1^1$ and $X$.
	Finally, the path $P_1^2$ starts with
	$(w_1,\bar w_2, \ldots,\bar w_{k})$, follows the
	orientation of $C_1$ via the vertices $\bar v$ until $\bar w_1$ and ends with
	$(\bar w_1, w_2,\ldots, w_k)$.
	One can easily check
	that these four paths form an absorber $A$ of order
	$q+2k$ rooted in $X$.
\end{proof}

It remains to show \cref{lem:absorption-dense}.
The proof follows line by line the one of \cref{lem:absorption-sparse} once a deterministic version of \cref{cla:absorption-properties-random-graph} is established.
This is done by the next lemma.
We omit the remaining details of the proof of \cref{lem:absorption-dense}.
\begin{lemma} \label{lem:absorption-properties-dense}
	Let $1/n \ll \eta \ll \nu,\,\alpha,\,1/M \ll 1/k,\, 1/d,\,
        \gamma$ and $\eta \ll \rho \ll \gamma$.
	Suppose $G$ is an $n$-vertex $k$-graph with $\delta_d(G) \geq (\mu_d(k) + \gamma ) \binom{n-d}{k-d}$, and let $Q \subset V(G)$ with $|Q| \leq \alpha n$.
	Then the following hold:
	\begin{enumerate}[\upshape(1)]\setcounter{enumi}{-1}
		\item \label{itm:absorption-connection0-dense}
		      Let~$C=V(G)\setminus Q$.  For every pair of distinct
		      vertices $u,\,v\in C$, there is a loose
		      $(u,v)$-path~$P$ of order $\ell=4(k-1)+1$ in~$G$ with
		      $V(P)\subset C$.
		\item \label{itm:absorption-connection-dense}
		      Let $C \subset V(G)\setminus Q$ be a set of size $\nu n$ taken
		      uniformly at random.
		      Then with probability at least $2/3$ the following holds.
		      For any $R \subset C$ with $|R| \leq \rho n$ and distinct $u,\,v \in V(G)\setminus (Q\cup R) $,
		      there is a loose $(u,v)$-path $P$ of order
		      $\ell=4(k-1)+1$ in~$G$ with $V(P)\setminus\{u,v\}
			      \subset C\sm R$.

		\item \label{itm:absorption-richness-dense}
		      Let $Z \subset V(G)$ be chosen uniformly at random among all $\nu n$-sets.
		      Then {with probability at least $2/3$} for any $W \subseteq
			      V(G)\setminus Z$ with $|W| \leq \eta n$, there is a matching in $G$
		      covering all vertices in $W$, each edge of which contains one
		      vertex of $W$ and $k-1$ vertices of $Z$.

		\item \label{itm:absorption-absorber-dense} For any $(k-1)$-set~$X$ in
		      $V(G) \sm Q$, there is an absorber~$A$ in~$G$ rooted in
		      $X$ that avoids~$Q$ and has order at most~$M$.\qed
	\end{enumerate}
\end{lemma}

\begin{proof}
	Part \ref{itm:absorption-connection0-dense} simply follows from \cref{lem:connection-dense} since $\delta_d(G-Q) \geq (\mu_d(k) + \gamma/2 ) \binom{n-|Q|-d}{k-d}$.
	Part \ref{itm:absorption-connection-dense} follows by first observing that a random $\nu$-set $C \subset V(G) \sm Q$ satisfies  {$\delta_d(G[C \cup \{u,v\}]) \geq (\mu_d(k) + \gamma ) \binom{\nu' n-d}{k-d}$} for every $u,\,v \in V(G)\sm Q$ with probability $2/3$.
	(The same observation was used in the proof of \cref{lem:threshold-exploitation}.)
	One can then apply \cref{lem:connection-dense} to obtain the desired path.
	Part \ref{itm:absorption-richness-dense} follows in the same way.
	The matching can be constructed greedily.
	Finally, part \ref{itm:absorption-absorber-dense} follows by \cref{lem:absorber-dense-simple}.
\end{proof}

\section{Conclusion}
\label{sec:conclusion}

In this paper, we investigated for which probabilities $p$ a subgraph of the binomial random graph $\oH_k(n,p)$ whose relative minimum $d$-degree is above the corresponding dense threshold contains a loose Hamilton cycle.
Our main result determines the optimal value for $p$ when $d > (k+1)/2$.
While we do provide bounds for $p$ for $d \leq (k+1)/2$, it is unlikely that our results are optimal in this range.
Hence a natural question is whether one can improve on this.
It would furthermore be very interesting to understand whether one can obtain similar results for `tighter' cycles.
A first step in this direction was undertaken by  Allen, Parczyk and Pfenninger~\cite{APP21} for $d=k-1$.
Note that in this situation, the problem of finding a tight Hamilton cycle in a dense graph is quite well-understood.
Going beyond this, one potentially challenging problem would be to
prove such a result for values of $d$ and $k$ for which we do not yet know the precise value of the (dense) minimum degree threshold.

Finally, we remark that our work is related to the corresponding `random robustness' problem mentioned in \cref{sec:introduction}.
In this setting, we are given a (deterministic) $n$-vertex $k$-graph $G$ with $\delta_d(G) \geq (\mu_d(k) + \gamma) \binom{n-d}{k-d}$ for some $\gamma>0$.
Now let $G'$ be a random sparsification of $G$, which is obtained by keeping every edge independently with probability $p$.
The challenge is then to determine the threshold $p$ for which $G'$ typically contains a loose Hamilton cycle.
A simple coupling argument shows that $p$ can be taken as small as in \cref{thm:main}.
It was shown independently by Kelly, Müyesser and Pokrosvkiy~\cite{KMP23} and Joos, Lang and Sanhueza-Matamala~\cite{JLS23} that \cref{thm:hamilton-connectedness} can be used to improve this to $p \geq C n^{-k+1} \log n$ for a constant $C=C(k)$, which is asymptotically optimal.

\bibliographystyle{amsplain}
\bibliography{extracted}

\end{document}